\documentclass[twoside,11pt]{amsart}
\pdfoutput=1

\usepackage[utf8]{inputenc}
\usepackage[T1]{fontenc}
\usepackage[mathscr]{euscript}
\usepackage{amsfonts}
\usepackage{pbox}
\usepackage{amssymb}
\usepackage{amsmath}
\usepackage[nodayofweek]{datetime}
\usepackage{enumitem}
\usepackage{csquotes}
\usepackage{amsthm}
\usepackage{hyperref}
\usepackage{etoolbox}
\usepackage{stmaryrd}
\usepackage{xparse}
\usepackage{mathtools}
\usepackage{cleveref}

\theoremstyle{plain}
\newtheorem{thm}{Theorem}[section]
\newtheorem{conj}[thm]{Conjecture}
\newtheorem{prop}[thm]{Proposition}
\newtheorem{lem}[thm]{Lemma}
\newtheorem{cor}[thm]{Corollary}
\theoremstyle{definition}
\newtheorem{dfn}[thm]{Definition}

\newtheorem{situation}[thm]{Situation}
\newtheorem{cond}[thm]{Condition}
\theoremstyle{remark}
\newtheorem{rem}[thm]{Remark}

\usepackage{tikz}
\usetikzlibrary{matrix,arrows,calc}
\usepackage{tikzpfeile}

\DeclareMathAlphabet{\mathpzc}{OT1}{pzc}{m}{it}

\setenumerate[1]{leftmargin=*,labelindent=\parindent,label=(\alph*)}
\setenumerate[2]{leftmargin=*,labelindent=\parindent,label=(\roman*)}
\newlist{arabiclist}{enumerate}{2}
\setlist[arabiclist]{leftmargin=*,labelindent=\parindent,label=(\arabic*)}
\setlist[arabiclist,2]{label=(\roman*)}
\newlist{tfaelist}{enumerate}{1}
\setlist[tfaelist]{leftmargin=*,labelindent=\parindent,label=(\roman*)}

\usepackage[style=alphabetic]{biblatex}
\bibliography{hida-fam-palf}

\makeatletter
\newcommand*{\@old@slash}{}\let\@old@slash\slash
\def\slash{\relax\ifmmode\delimiter"502F30E\mathopen{}\else\@old@slash\fi}
\def\backslash{\delimiter"526E30F\mathopen{}}
\makeatother 

\DeclareRobustCommand{\rquot}[2]{
  \mathchoice
  { 
    \left.
    \raisebox{0.5ex}{\ensuremath{#1}}
    \middle\ensuremath\slash
    \raisebox{-0.4ex}{\ensuremath{#2}}
    \right.  
  }
  { 
    #1 / #2 
  }
  { 
    #1 / #2
  }
  { 
    #1 / #2
  }
}
\DeclareRobustCommand{\lquot}[2]{
  \mathchoice
  { 
    \left.
    \raisebox{-0.4ex}{\ensuremath{#2}}
    \middle\ensuremath\backslash
    \raisebox{0.5ex}{\ensuremath{#1}}
    \right.  }
  { 
    #2\ensuremath\backslash #1
  }
  { 
    #2\ensuremath\backslash #1
  }
  { 
    #2\ensuremath\backslash #1
  }
}

\DeclareRobustCommand{\quotient}[2]{\rquot{#1}{#2}}

\DeclareRobustCommand{\N}{{\mathbb{N}}}
\DeclareRobustCommand{\Z}{{\mathbb{Z}}}
\DeclareRobustCommand{\Q}{{\mathbb{Q}}}
\DeclareRobustCommand{\R}{{\mathbb{R}}}
\makeatletter
\let\C\@undefined
\let\P\@undefined
\makeatother
\DeclareRobustCommand{\C}{{\mathbb{C}}}
\DeclareRobustCommand{\P}{{\mathbb{P}}}

\DeclareRobustCommand{\e}{\mathrm{e}}
\makeatletter
\let\i\@undefined
\let\Re\@undefined
\let\Im\@undefined
\makeatother
\DeclareRobustCommand{\i}{\mathrm{i}}
\DeclareRobustCommand{\Re}{\operatorname{Re}}
\DeclareRobustCommand{\Im}{\operatorname{Im}}

\DeclareRobustCommand{\Mat}{\operatorname{M}}
\DeclareRobustCommand{\GL}{\operatorname{GL}}
\DeclareRobustCommand{\SL}{\operatorname{SL}}

\DeclareRobustCommand{\da}{\mathrel{\vcenter{\baselineskip0.5ex \lineskiplimit0pt
                     \hbox{\scriptsize.}\hbox{\scriptsize.}}}%
                     =}
\DeclareRobustCommand{\ad}{=\mathrel{\vcenter{\baselineskip0.5ex \lineskiplimit0pt
                     \hbox{\scriptsize.}\hbox{\scriptsize.}}}%
                     }
\DeclareRobustCommand{\ep}{\varepsilon}
\let\temp\phi
\let\phi\varphi
\let\varphi\temp

\makeatletter
\let\d\@undefined
\makeatother
\DeclareRobustCommand{\d}{\mathrm{d}}

\DeclareRobustCommand{\image}{\operatorname{im}}

\makeatletter
\let\ker\@undefined
\makeatother
\DeclareRobustCommand{\ker}{\ensuremath{\operatorname{ker}}}

\DeclareRobustCommand{\Aut}{\ensuremath{\operatorname{Aut}}}

\DeclareRobustCommand{\Spec}{\operatorname{Spec}}

\DeclareRobustCommand{\Div}{\operatorname{Div}}

\DeclareRobustCommand{\Quot}{\operatorname{Quot}}

\DeclareRobustCommand{\restrict}[1]{|_{#1}}
\DeclareRobustCommand{\paarung}[2]{
  \mathchoice
  { 
    {\left<#1,#2\right>}
  }
  { 
    {\langle#1,#2\rangle}
  }
  { 
    {\langle#1,#2\rangle}
  }
  { 
    {\langle#1,#2\rangle}
  }
}
\DeclareRobustCommand{\zmod}[1]{\quotient{\Z}{#1}} 
\DeclareRobustCommand{\zmodmal}[1]{\left(\zmod{#1}\right)^\times}
\DeclareRobustCommand{\psring}[2]{#1\hspace*{0pt}[\hspace*{-2pt}[#2]\hspace*{-2pt}]}

\DeclareRobustCommand{\binmatrix}[4]{
  \mathchoice
  { 
    \begin{pmatrix}#1&#2\\#3&#4\end{pmatrix}
  }
  { 
    \bigl( \begin{smallmatrix}#1&#2\\#3&#4\end{smallmatrix} \bigr)
  }
  { 
    \bigl( \begin{smallmatrix}#1&#2\\#3&#4\end{smallmatrix} \bigr)
  }
  { 
    \bigl( \begin{smallmatrix}#1&#2\\#3&#4\end{smallmatrix} \bigr)
  }
}

\DeclareRobustCommand{\ke}[2]{#1/#2} 
\DeclareRobustCommand{\Gal}[2]{\operatorname{Gal}(\ke{#1}{#2})}
\DeclareRobustCommand{\inertia}{\mathrm{I}}
\makeatletter
\let\mod\@undefined
\makeatother
\DeclareRobustCommand{\mod}{\operatorname{mod}}

\DeclareRobustCommand{\category}[1]{\mathpzc{#1}}

\makeatletter
\let\Top\@undefined
\makeatother
\DeclareRobustCommand{\Top}{\category{Top}}

\DeclareRobustCommand{\Mod}{\category{Mod}}
\DeclareRobustCommand{\Hom}{\operatorname{Hom}}

\makeatletter

\newcommand{\colim@}[2]{%
  \vtop{\m@th\ialign{##\cr
    \hfil$#1\operator@font colim$\hfil\cr
    \noalign{\nointerlineskip\kern1.5\ex@}#2\cr
    \noalign{\nointerlineskip\kern-\ex@}\cr}}%
}
\DeclareRobustCommand{\colim}{%
  \mathop{\mathpalette\colim@{\rightarrowfill@\textstyle}}\nmlimits@
}
\newcommand{\lim@}[2]{%
  \vtop{\m@th\ialign{##\cr
    \hfil$#1\operator@font lim$\hfil\cr
    \noalign{\nointerlineskip\kern1.5\ex@}#2\cr
    \noalign{\nointerlineskip\kern-\ex@}\cr}}%
}
\DeclareRobustCommand{\lim}{%
  \mathop{\mathpalette\lim@{\leftarrowfill@\textstyle}}\nmlimits@
}
\makeatother

\DeclareRobustCommand{\Qquer}{\smash{\overline{\Q}}}
\DeclareRobustCommand{\Qbar}{\Qquer}

\DeclareRobustCommand{\Frob}{\operatorname{Frob}}
\DeclareRobustCommand{\frakp}{\mathfrak{p}}

\NewDocumentCommand{\tensor}{t_}
 {%
  \IfBooleanTF{#1}
   {\tensop}
   {\otimes}%
 }
\NewDocumentCommand{\tensop}{m}
 {%
  \mathbin{\mathop{\otimes}\displaylimits_{#1}}%
 }
\NewDocumentCommand{\fibertimes}{t_}
 {%
  \IfBooleanTF{#1}
   {\fibertimesop}
   {\times}%
 }
\NewDocumentCommand{\fibertimesop}{m}
 {%
  \mathbin{\mathop{\times}\displaylimits_{#1}}%
 }

\makeatletter
\let\O\@undefined
\makeatother
\DeclareRobustCommand{\O}{\mathcal{O}}
\DeclareRobustCommand{\wnklevel}[1]{\prescript{#1}{k}\!\mathcal{W}}
\DeclareRobustCommand{\wnk}{\wnklevel{N}}
\DeclareRobustCommand{\Sym}{\operatorname{Sym}}
\DeclareRobustCommand{\TSym}{\operatorname{TSym}}
\DeclareRobustCommand{\parab}{\mathrm{p}}
\DeclareRobustCommand{\cp}{\mathrm{c}}

\DeclareRobustCommand{\RD}{\mathrm{R}}
\DeclareRobustCommand{\RRD}{\mathbf{R}}
\DeclareRobustCommand{\an}{\mathrm{an}}
\DeclareRobustCommand{\HL}{\mathrm{H}}

\DeclareRobustCommand{\Hf}{\HL_{\mathrm f}}
\DeclareRobustCommand{\Hg}{\HL_{\mathrm g}}
\DeclareRobustCommand{\ord}{\text{\normalfont ord}}
\DeclareRobustCommand{\antiord}{{\iota\text{\normalfont-ord}}}
\DeclareRobustCommand{\HHL}{\mathbf{H}}
\DeclareRobustCommand{\Hp}{\HL_\parab}
\DeclareRobustCommand{\Hpet}{\HL_{\parab,\et}}
\DeclareRobustCommand{\Hc}{\HL_\cp}

\DeclareRobustCommand{\GR}{\Galgp\R}
\DeclareRobustCommand{\KS}{\mathrm{KS}}
\DeclareRobustCommand{\preKS}{\mathrm{preKS}}
\DeclareRobustCommand{\MS}{\mathrm{MS}}

\DeclareRobustCommand{\Kit}{\mathrm{Kit}}
\DeclareRobustCommand{\FK}{\mathrm{FK}}

\DeclareRobustCommand{\dR}{\mathrm{dR}}
\DeclareRobustCommand{\Motive}{\mathcal{M}}
\DeclareRobustCommand{\Mf}{\Motive(f)}

\DeclareRobustCommand{\new}{\mathrm{new}}

\DeclareRobustCommand{\ModForms}{\mathrm{M}}
\DeclareRobustCommand{\calS}{\mathcal{S}}
\DeclareRobustCommand{\Sbar}{\overline\calS}

\makeatletter
\let\S\@undefined
\let\SS\@undefined
\let\H\@undefined
\makeatother
\DeclareRobustCommand{\S}{\mathrm{S}}
\DeclareRobustCommand{\H}{\mathfrak{h}}
\DeclareRobustCommand{\SS}{\mathbb{S}}

\DeclareRobustCommand{\GQ}{\Galgp{\Q}}

\DeclareRobustCommand{\betti}{\mathrm{B}}
\DeclareRobustCommand{\hodge}{\mathrm{H}}
\DeclareRobustCommand{\degzero}[1]{\underset{\substack{\uparrow\\0}}{#1}}
\DeclareRobustCommand*{\shortra}[1][]{\raisebox{-1pt}{\tikz{%
      \draw[xscale=.5,thin,shorten >=3pt, ->,font=\scriptsize] (0,0)%
      node{\hspace*{-2pt}} -- (0.5,0) node[above] {#1}%
      -- node{} (1,0);}}\penalty1000\relax}
\makeatletter
\let\setminus\@undefined
\makeatother
\DeclareRobustCommand{\setminus}{\smallsetminus}
\DeclareRobustCommand{\ES}{\mathrm{ES}}

\DeclareRobustCommand{\smuline}[1]{\underline{\smash{#1}}}
\DeclareRobustCommand{\fil}{\operatorname{fil}}
\DeclareRobustCommand{\gr}{\operatorname{gr}}
\DeclareRobustCommand{\D}[1]{\operatorname{D}_{#1}}

\DeclareRobustCommand{\tangentspace}[1]{\mathrm{t}_{#1}}
\DeclareRobustCommand{\Symgrp}[1]{\mathfrak{S}_{#1}}
\DeclareRobustCommand{\logOmega}{\mho}

\DeclareRobustCommand{\Qp}{{\Q_p}}

\DeclareRobustCommand{\Qell}{{\Q_\ell}}
\DeclareRobustCommand{\Cp}{{\C_p}}
\DeclareRobustCommand{\Zp}{{\Z_p}}
\DeclareRobustCommand{\H}{\mathfrak{h}}

\DeclareRobustCommand{\calMS}{\mathcal{MS}}
\DeclareRobustCommand{\UM}{\mathcal{UM}}
\DeclareRobustCommand{\MMSS}{\mathbb{MS}}
\DeclareRobustCommand{\HT}{\mathrm{HT}}
\DeclareRobustCommand{\cris}{\mathrm{cris}}
\DeclareRobustCommand{\frobcris}{\phi_\cris}

\DeclareRobustCommand{\B}[1]{\operatorname{B}_{#1}}

\DeclareRobustCommand{\BdR}{\B{\dR}}
\DeclareRobustCommand{\tdR}{\mathrm{t}_\dR}
\DeclareRobustCommand{\BHT}{\B{\HT}}

\DeclareRobustCommand{\Dcris}{\D{\cris}}

\DeclareRobustCommand{\DdR}{\D{\dR}}

\DeclareRobustCommand{\pst}{\mathrm{pst}}
\DeclareRobustCommand{\Dpst}{\D{\pst}}

\DeclareRobustCommand{\et}{\text{\normalfont ét}}
\DeclareRobustCommand{\nr}{\mathrm{nr}}

\DeclareRobustCommand{\Galgp}[1]{\operatorname{G}_{#1}}
\DeclareRobustCommand{\GQp}{\Galgp{\Qp}}
\DeclareRobustCommand{\GQell}{\Galgp{\Qell}}

\DeclareRobustCommand{\nrhat}[1]{\hat#1^\nr}

\DeclareRobustCommand{\Qpbar}{\overline\Q_p}

\DeclareRobustCommand{\tH}{\operatorname{t}_{\mathrm{H}}}
\DeclareRobustCommand{\DP}{\mathrm{DP}}

\DeclareRobustCommand{\wt}{\mathrm{wt}}
\DeclareRobustCommand{\Lambdawt}{\Lambda\!^\wt}
\DeclareRobustCommand{\cyc}{\mathrm{cyc}}

\DeclareRobustCommand{\diamondop}[1]{\langle#1\rangle}
\DeclareRobustCommand{\calI}{\mathcal I}
\DeclareRobustCommand{\calX}{\mathcal X}
\DeclareRobustCommand{\calL}{\mathcal Q}
\DeclareRobustCommand{\calLbar}{\overline{\calL}}
\DeclareRobustCommand{\arith}{\mathrm{arith}}
\DeclareRobustCommand{\arithm}{\arith}
\DeclareRobustCommand{\naive}{\mathrm{naive}}

\DeclareRobustCommand{\hecket}{\mathbf t}
\DeclareRobustCommand{\heckeT}{\mathbf T}

\DeclareRobustCommand{\HeckeKonkretadjoint}[3]{\heckeT^\iota_{#1}(#2,#3)}

\DeclareRobustCommand{\tord}{\hecket^\ord}
\DeclareRobustCommand{\Tord}{\heckeT^\ord}

\DeclareRobustCommand{\EP}{\boldsymbol\ep}

\DeclareRobustCommand{\cpiso}{\operatorname{cp}}

\DeclareRobustCommand{\Pet}{\mathrm{Pet}}
\DeclareRobustCommand{\errorterm}{\mathcal E}

\DeclareRobustCommand{\MotiveCanBasis}[1]{b^{#1}}

\DeclareRobustCommand{\DirichletCanBasis}{\MotiveCanBasis{\chi}}

\DeclareRobustCommand{\cc}{\Frob_\infty}
\DeclareRobustCommand{\phigamma}{\ensuremath{(\phi,\Gamma)}}
\DeclareRobustCommand{\Gcyc}{\mathrm{G}_\cyc}

\DeclareRobustCommand{\bigLambdaNormal}{\raisebox{\depth}{\scalebox{-1}[-1]{$\mathbb V$}}}
\DeclareRobustCommand{\bigLambdaSmall}{\raisebox{\depth}{\scalebox{-1}[-1]{\footnotesize$\mathbb V$}}}
\DeclareRobustCommand{\bigLambda}{
  \mathchoice
  { 
    \bigLambdaNormal
  }
  { 
    \bigLambdaNormal    
  }
  { 
    \bigLambdaSmall
  }
  { 
    \bigLambdaSmall
  }
}

\DeclareRobustCommand{\calT}{\mspace{2mu}\mathcal T\mspace{-2mu}} 
\DeclareRobustCommand{\calTexp}[1]{\mathcal T^{\raisebox{-0.2pt}{\scriptsize$#1$}}}

\DeclareRobustCommand{\Gausssum}{\mathrm G}

\DeclareRobustCommand{\rig}{\mathrm{rig}}
\DeclareRobustCommand{\robba}{\B\rig^\dagger}

\DeclareRobustCommand{\phigammafamcat}[1]{\robba\text-\Mod^{\phigamma}_A}

\DeclareRobustCommand{\Ref}{\mathrm{Ref}}

\DeclareRobustCommand{\conductor}{\operatorname{cond}}
\DeclareRobustCommand{\EulerFactorP}{\operatorname{P}}
\DeclareRobustCommand{\ordnung}{\operatorname{ord}}
\newcommand{\longra}[2][]{%
  \raisebox{-1pt}{\tikz{%
      \draw[xscale=#2,thin,shorten >=3pt, ->,font=\scriptsize] (0,0)%
      node{\hspace*{-2pt}} -- (0.5,0) node[above] {#1}%
      -- node{} (1,0);}}\penalty1000\relax%
}

\DeclareRobustCommand{\MotiveCanBasis}[1]{b^{#1}}
\DeclareDocumentCommand{\TateCanBasis}{O{1} O{\Q}}{\MotiveCanBasis{#2(#1)}}

\DeclareRobustCommand{\DirichletCanBasis}{\MotiveCanBasis{\chi}}

\begin{document}

\title{A $p$-adic $L$-function with canonical motivic periods for Hida families}%
\author{Michael Fütterer}
\address{Univeristät Heidelberg\\ Mathematisches Institut\\
Im Neuenheimer Feld 205\\ 69120 Heidelberg, Germany.} \email{mfuetterer@mathi.uni-heidelberg.de}
\urladdr{http://www.mathi.uni-heidelberg.de/\textasciitilde mfuetterer/}
\thanks{The author acknowledges support by the DFG.}
\date{\today}%

\maketitle

\tableofcontents

\thispagestyle{empty}

\section*{Introduction}

In \cite{MR2276851} Fukaya and Kato formulated a conjecture on $p$-adic $L$-functions in a
general setting for a wide class of motives, in particular they wrote down a precise
conjectural interpolation formula, involving complex and $p$-adic periods defined in terms
of comparison isomorphisms. They showed that their conjecture is
implied by the Equivariant Tamgawa Number Conjecture and Kato's local $\ep$-isorphism
conjecture (we abbreviate these conjectures by ETNC in the following), which gives their
formula a conceptual explanation. This naturally raises the question whether this formula
specializes to the known interpolation formulas in cases where a $p$-adic $L$-function has
already been constructed. For families this question is particularly interesting because one
needs to study the variation of the periods in families.

This text answers this question in the case of Hida families of elliptic cusp forms and
compares the conjectural $p$-adic $L$-function with the one constructed by Kitagawa in
\cite{MR1279604}. We show that, under a technical hypothesis on the Hida family, we can
modify Kitagawa's construction in such a way that it really produces a function having the
interpolation property predicted by Fukaya and Kato.\footnote{Up to a sign which we cannot
  interpolate, see \cref{rem:problem}.}  Moreover this modification is essentially just
multiplication by a unit in the Iwasawa algebra, so the two $p$-adic $L$-functions generate
the same ideal. See \cref{thm:main-thm} for the precise result we obtain.

The most important part in this work is the calculation of the complex and $p$-adic periods
of the motive attached to a modular form, which allows us to express them in terms of
Kitagawa's error terms. In the complex case, we make use of the fact that the comparison
isomorphism is essentially the classical Eichler-Shimura isomorphism (due to lack of a
reference we include a proof of this well-known fact), which admits a rather explicit
description. The same holds in the $p$-adic case as well. Here enters as the most important
ingredient to our proof the fact that the $p$-adic Eichler-Shimura isomorphisms can be
interpolated in a Hida family, which was proved by Kings, Loeffler and Zerbes
\cite{MR3637653} building on work of Ohta \cite{MR1332907}. This result enables us to make
the choices in Kitagawa's construction canonical, which leads to the desired $p$-adic
$L$-function.

This work is based on the author's PhD thesis \cite{Diss}, where many additional details can
be found. The author wants to thank his advisor Otmar Venjakob and all the colleagues at the
Mathematical Institute in Heidelberg for their support.

\subsection*{Notations and conventions}

In the whole work $p$ denotes an odd prime. We fix algebraic closures $\Qquer$ of $\Q$,
$\Qpbar$ of $\Qp$ and $\C$ of $\R$ and write $\C_p$ for the $p$-adic completion of $\Qpbar$.

We fix throughout the work a square root $\i\in\C$ of $-1$. By a {pair of embeddings of
  $\Qbar$}, we mean a pair $(\iota_\infty, \iota_p)$ of embeddings
${\iota_\infty}\colon\Qquer\inj\C$ and ${\iota_p}\colon\Qquer\inj\Qpbar\subseteq\C_p$. We
provisionally fix such a pair of embeddings. This fixes a choice of a compatible system of
$p$-power roots of unity ${\xi=(\xi_n)_{n\ge0}}$ with $\xi_n\in\Qp(\mu_{p^\infty})$ by
saying that the pair of embeddings should identify $\xi$ with the system
$(\e^{2\pi\i p^{-n}})_{n\ge0}$ of $p$-power roots of unity in $\C$.\label{c-cp-orientation}
Our choice of $(\iota_\infty,\iota_p)$ is only provisional, we may change it at some point
in this work. When we do so we thus also have to change $\xi$.  Our fixed choice
of 
$\xi=(\xi_n)_{n\ge0}$ determines a uniformizer of $\BdR^+$ which we denote by ${\tdR}$.

We also fix a number field $K$ and an embedding $K\subseteq\Qbar$, which will play the role
of the coefficient field for modular forms. We write $\O_K$ for its ring of integers. The
induced inclusion $K\subseteq\Qpbar$ induces a place of $K$ above $p$ that we call
$\frakp$. We write $L$ for the completion of $K$ at $\frakp$ and $\O$ for its ring of
integers.


We normalize the reciprocity map from class field theory such that it maps prime elements to
\emph{arithmetic} Frobenii. This is particularly important when we view Dirichlet characters
as Galois characters.

\section{Modular forms and their motives}

In this section we review the most important facts about the motive attached to a modular
newform and provide proofs for some facts for which we do not have a reference.

\subsection{Modular curves and Hecke correspondences}

Fix an integer $N\ge4$.
If $E\ra S$ is a generalized elliptic curve in the sense of Deligne and Rapoport
\cite[Déf. II.1.12]{MR0337993}, we consider the following types of level $N$
structures on $E$ (see also \cite[Def.\ 2.4.1--2]{MR2311664}):
\begin{enumerate}
\item A naive $\Gamma(N)$-structure on $E$ is a homomorphism of group schemes over $S$ \[ \phi\colon{\left(\zmod N\right)}^2\ra E[N] \] such that there is an equality of effective Cartier divisors on $E$ \[ E[N] = \!\!\!\!\!\!\sum_{(a,b)\in{\zmodmal{N}}^2}\!\!\!\!\![\phi(a,b)] \]
  and the above Cartier divisor meets each irreducible component in each geometric fiber.
\item A naive $\Gamma_1(N)$-structure on $E$ is a homomorphism of group schemes over $S$ \[ \phi\colon\zmod N\ra E[N] \] such that the effective Cartier divisor \[ \sum_{a\in\zmodmal N}\!\! [\phi(a)] \] is a subgroup scheme of $E$ and meets each irreducible component in each geometric fiber.
\item An arithmetic $\Gamma_1(N)$-structure on $E$ is a closed immersion of group schemes over $S$
  \[ \varphi\colon\mu_N\inj E[N]. \]
\end{enumerate}
Because $N\ge4$, the functors that associate to a scheme $S$ the set of isomorphism classes
of pairs $(E,\phi)$, where $E$ is a generalized elliptic curve over $S$ and $\phi$ is one of
the aforementioned level structures, are representable by a scheme over $\Z$, the so-called
modular curves. We denote these by $X(N)$, $X_1(N)$ and $X_1(N)^\arithm$, respectively, and
denote by $Y(N)$, $Y_1(N)$ resp.\ $Y_1(N)^\arithm$ the open subschemes parametrizing actual
(as opposed to generalized) elliptic curves with level structures.  All of these are finite
flat regular curves over $\Z$ which are finite \'etale over $\Z[1/N]$.\footnote{There is
  also the notion of an arithmetic $\Gamma(N)$-structure, and the corresponding moduli
  functor is also representable, but it will play no role in this work. Moreover, for the
  $\Gamma_1(N)$-structures it would suffice to assume $N\ge3$.}  The generalized
elliptic curves over $X(N)$ and $X_1(N)$ will be denoted by $\overline E(N)$ and
$\overline E_1(N)$ and the universal elliptic curves over $Y(N)$ and $Y_1(N)$ by $E(N)$ and
$E_1(N)$, respectively. Each of the natural maps from any of the universal (generalized)
elliptic curve to the corresponding modular curve will be denoted by $f$.

We have natural a morphism $X(N)\ra X_1(N)$, forgetting the second coordinate. Over any
commutative ring containing an $N$-th root of unity, $X_1(N)$ and $X_1(N)^\arithm$ become
canonically isomorphic. Further we have the well-known diamond operators and Hecke
correspondences on these modular curves, which induce endomorphisms of cohomology groups
attached to those curves, called the diamond and Hecke operators. We omit a detailed
description of these and just refer to \cite[(3.13)–(3.18)]{MR3077124}.

We will denote their analytifications by $X(N)^\an$, $X_1(N)^\an$ etc.\ (and similarly for other schemes). Note that because $\C$ contains an $N$-th
root of unity there is no need to distinguish between the naive and arithmetic versions in
this case. It is well-known that $Y_1(N)^\an$ is isomorphic to the quotient of the upper
half plane $\H$ by the congruence subgroup $\Gamma_1(N)$ and $X_1(N)$ is isomorphic to the
quotient of the extended upper half plane $\H^*=\H\cup\P^1(\Q)$ by $\Gamma_1(N)$. The point
of exact order $N$ in the fiber $E_\tau=\C/(\Z\oplus\Z\tau)$ over some $\tau\in\H$ is
$1/N$.\label{point-of-exact-order-N} In contrast, $Y(N)^\an$ is \emph{not} isomorphic to
the quotient of $\H$ by $\Gamma(N)$, but rather to a disjoint union of $\phi(N)$ copies of
it (where $\phi$ is the Euler totient function). This is because the quotient of $\H$ by
$\Gamma(N)$ only parametrizes elliptic curves over $\C$ with $\Gamma(N)$-structure such that
the two points have a fixed Weil pairing, and there are $\phi(N)$ possible values for the
Weil pairing. See \cite[§1.5]{MR2112196}, \cite[Introduction, p. 15]{MR0337993},
\cite[§1.8]{MR2104361}.

We describe the analytification of $E_1(N)$.
Define an action of $\Sigma^+\da\Mat_2(\Z)\cap\GL_2^+(\Q)$ on $\C\times\H$ by 
\begin{equation}\label{eqn:action-slz-ctimesh}
  \gamma(z,\tau) = (\det(\gamma)(c\tau+d)^{-1}z,\gamma\tau) \qquad \text{for }\gamma=\binmatrix{a}{b}{c}{d}\in\Sigma^+,\ z\in\C,\ \tau\in\H
\end{equation}
and let $\EP\da\binmatrix100{-1}$ act on $\C\times\H$ as
$\EP(z,\tau)\da(\overline z,-\overline\tau)$. It is then easy to
check that we get a well-defined action of $\Sigma\da\Mat_2(\Z)\cap\GL_2(\Q)$ (which is the
monoid generated by $\Sigma^+$ and $\EP$) on $\C\times\H$. The projection $\C\times\H\ra\H$
is equivariant. We define further $E_\H$ as the quotient of $\C\times\H$ by the left action
of $\Z^2$ given by
\[ (m,n)(z,\tau) = (z+m\tau+n,\tau) \qquad (m,n\in\Z) \] and let $f\colon E_\H\ra\H$ be the
projection onto the second factor (alternatively, define $\Lambda$ as the image of 
\begin{equation}
  \label{eqn:def-lambda}
  \Z^2\times\H\inj\C\times\H,\quad(m,n,\tau)\mapsto(m\tau+n,\tau), 
\end{equation}
so we have
\begin{equation}
  \label{eqn:def-lambda-alt}
  \Lambda=\bigcup_{\tau\in\H}\,\Z\oplus\Z\tau\times\{\tau\},
\end{equation}
and put $E_\H\da\quotient{(\C\times\H)}{\Lambda}$). Then the previous action of
$\Mat_2(\Z)\cap\GL_2(\Q)$ on $\C\times\H$ descends to $E_\H$. The analytification
$E_1(N)^\an$ can be described as the quotient of $E_\H$ by $\Gamma_1(N)$.

\subsection{Modular forms}
\label{sec:modular-forms}

In this section we abbreviate $X=X_1(N)^\arithm$, $\overline E=\overline E_1(N)^\arithm$ and
denote by $C$ be the divisor of cusps on $X_1(N)^\arithm$.

If $f\colon E\ra S$ is a generalized elliptic curve with unit section $e\colon S\ra E$ we
put $\omega_{E/S}=e^*\Omega^1_{E/S}$. This is a line bundle on $S$ which is canonically
isomorphic to $f_*\Omega^1_{ E/S}$ and stable under base change, see \cite[Prop.\ I.1.6
(ii)]{MR0337993}.

For any commutative ring $R$ we let
\[ \ModForms_k(\Gamma_1(N),R) \da \HL^0(X_1(N)^\arithm_{/R}, (\omega^{k-2}_{\overline
    E/X}\tensor_{\O_X}\Omega^1_X(C))_{/R}) \]
be the $R$-module of modular forms of level $N$ and weight $k$ with coefficients in $R$ and
\[ \S_k(\Gamma_1(N),R) \da \HL^0(X_1(N)^\arithm_{/R}, (\omega^{k-2}_{\overline
    E/X}\tensor_{\O_X}\Omega^1_X)_{/R}) \] the submodule of cusp forms. Here the subscript
\enquote{$/R$} means base change to $R$ (over $\Z$). With this definition we have that if
$R$ is a subring of $\C$ then $\ModForms_k(\Gamma_1(N),R)$ and $\S_k(\Gamma_1(N),R)$ can be
canonically identified with classical modular resp.\ cusp forms for the congruence subgroup
$\Gamma_1(N)$ and weight $k$ all of whose Fourier coefficients lie in $R$,\footnote{At this
  point we need that we work with the \emph{arithmetic} modular curve, see \cite[Rem.\
  4.4.2]{MR2311664}. The problem is that the cusp at infinity which belongs to the Tate
  curve is not defined over $\Z$ on the naive modular curve, but only over $\Z[\mu_N]$,
  while for $X_1(N)^\arithm$ it is defined over $\Z$; see also \cite[Rem.\
  12.3.6]{MR1357209}.}  and in general we have for any commutative ring $R$ and any
commutative flat $R$-algebra $S$ canonical isomorphisms
\[ \ModForms_k(\Gamma_1(N),S) = \ModForms_k(\Gamma_1(N),R) \tensor_R S,\quad  \S_k(\Gamma_1(N),S) = \S_k(\Gamma_1(N),R) \tensor_R
  S. \]


The modules we just introduced carry actions of Hecke and diamond operators which we denote
by $T_\ell$ for all primes and $\diamondop d$ for $d\in\zmodmal N$. They are induced by the
correspondences mentioned before and coincide with the classical ones if $R$ is a subring of
$\C$. We denote the Hecke algebras for these modules by $\heckeT_k(N,R)$ and
$\hecket_k(N,R)$ (with $\heckeT$ for modular and $\hecket$ for cusp forms).  Using the
diamond operators we can then define the subspaces with fixed nebentype
$\psi\colon\zmodmal N\ra\R^\times$, which we denote by $\ModForms_k(\Gamma_1(N),\psi,R)$ and
$\S_k(\Gamma_1(N),\psi,R)$, respectively.  Their Hecke algebras will be denoted by
$\heckeT_k(N,\psi,R)$ and $\hecket_k(N,\psi,R)$, respectively.

For the reasons explained at the end of the previous section, if we consider the analogous
definitions with the curve $X_1(N)^\arith$ replaced by $X(N)$ we do \emph{not} get the
classical space of $\Gamma(N)$-modular forms over $\C$, but rather a direct sum of $\phi(N)$
copies of it. Nonetheless denote the resulting modules by
\[ \ModForms_k(\Gamma(N),R), \quad \S_k(\Gamma(N),R) \]
although this is \emph{not} the standard notation in the literature.

We an call $f\in\S_k(\Gamma_1(N),R)$ for some commutative ring $R$
a newform if it is an eigenform and moreover there
does not exist a proper divisor $M\mid N$, a finite ring extension $S$ of $R$ and
$g\in\S_k(\Gamma_1(M),S)$ which is an eigenvector of almost all Hecke operators with the same
eigenvalues as $f$. It is easy to see that this is equivalent to the classical definition whenever $R$ is a subring of $\C$ (see Miyake, Thm. 4.6.12, Thm. 4.6.14).



\subsection{Motives for modular forms}
\label{sec:wnk}

We briefly review Scholl's construction of motives for modular forms from \cite{MR1047142}.

Fix $N\ge4$, $k\ge2$. In this section we consider the modular curves as curves over $\Q$,
but often we omit writing down the base change from $\Z$ to $\Q$ explicitly. We define
$\preKS(N,k)$ to be the $(k-2)$-fold fiber product of the universal generalized elliptic
curve $\overline E(N)$ with itself over $X(N)$, which is a singular projective variety over
$\Q$.
By \cite[Thm.\ 3.1.0]{MR1047142} it has a canonical desingularisation 
$\KS(N,k)$ which is called the Kuga-Sato variety of level $N$ and weight $k$.



We have natural actions on $\preKS(N,k)$ of the groups $(\zmod N)^2$ (acting on the torsion
basis in each fiber), $\{\pm1\}$ (by inversion on $\overline E(N)$) and the symmetric group
$\Symgrp{k-2}$ (by permuting the factors). This gives rise to an action of
\[ G(N,k)\da {\Big(\left(\zmod N\right)^2\rtimes\{\pm1\}\Big)}^{k-2}\rtimes \Symgrp{k-2} \]
on $\preKS(N,k)$ which extends to $\KS(N,k)$.

\begin{dfn}
  Let $\ep\colon G(N,k)\ra\{\pm1\}$ be the character that is trivial on each factor
  $\zmod N$, is the product map on $\{\pm1\}^{k-2}$ and the sign on $\Symgrp{k-2}$. Then let
  $\pi_\ep\in\Q[\Gamma_k]$ be the projector onto the $\ep$-eigenspace and define a Chow
  motive $\wnk\da(\KS(N,k),\pi_\ep)$.
\end{dfn}

The Hecke correspondences extend to correspondences on $\KS(N,k)$ and thus induce
endomorphisms of the motive $\wnk$. This defines actions of Hecke operators on $\wnk$, and
thus also on its realizations. Further the action of $\GL_2(\zmod N)$ on $E(N)\ra Y(N)$ also
carries over to endomorphisms of $\wnk$. We now fix a newform $f\in\S_k(\Gamma_1(N),K)$ with
coefficients in the number field $K$. In \cite[§4.2.0]{MR1047142}, Scholl explains that if we
view $\wnk\tensor_\Q K$ as a Grothendieck motive, then one can decompose $\wnk$ into
Hecke eigenspaces, so that the following definition makes sense.

\begin{dfn}
  The motive $\Mf$ attached to $f$ is the (Grothendieck) submotive of
  $\wnk\tensor_\Q K$ on which the Hecke operator $T_\ell$ acts precisely by the Hecke
  eigenvalue $a_\ell$ for all primes $\ell\nmid N$ (where the $a_\ell$ are the Hecke eigenvalues
  of $f$) and where the subgroup \[ \binmatrix1*0*\subseteq\GL_2\left(\zmod{N}\right) \]
  acts trivially.
\end{dfn}

Regarding the Betti and $p$-adic realizations, we have the following.
\begin{thm}\label{thm:wnk-betti-p}
  The {Betti realization of $\wnk$} is the parabolic cohomology group
  \[ \wnk_\betti=\Hp^1(Y(N)^\an,\Sym^{k-2}\RD^1f_*\smuline\Q). \] The action of $\GR$ comes
  from its natural action on $Y(N)^\an$ and $E(N)^\an$ and the resulting $\GR$-sheaf
  structure on $\Sym^{k-2}\RD^1f_*\smuline\Q$.

  The $p$-adic realization of $\wnk$ is
  \[ \wnk_p=\Hpet^1(Y(N)\fibertimes_\Z\Qbar,\Sym^{k-2}\RD^1f_*\smuline\Q_p). \] It has
  Hodge-Tate weights $k-1$ and $0$.
\end{thm}
\begin{proof}
  \cite[Thm. 1.2.1 and §1.2.0--1]{MR1047142}
\end{proof}

The isomorphism comes each time from the Leray spectral sequence for the morphism
$\KS(N,k)\ra X(N)$.
The corresponding realization
of $\Mf$ is by construction the subspace where the Hecke operator $T_\ell$ acts
precisely by the Hecke eigenvalue $a_\ell$ for all primes $\ell\nmid N$ (where the $a_\ell$ are
the Hecke eigenvalues of $f$) and where the subgroup
\[ \binmatrix1*0*\subseteq\GL_2\left(\zmod{N}\right) \] acts trivially.
In particular, the $\frakp$-adic realization of $\Mf$ is Deligne's Galois representation
attached to $f$.

We will describe the de Rham realization in \cref{sec:dr}.

Using monodromy and the fact that $X(N)^\an$ is isomorphic to a disjoint union of $\phi(N)$
copies of the quotient of the upper half plane by $\Gamma(N)$ we can describe $\wnk_\betti$
more explicitely as
\begin{equation}
  \label{eqn:wnk-betti-gpcohom}
  \wnk_\betti\cong\bigoplus_{\phi(N)}\Hp^1(\Gamma(N),\Sym^{k-2}\Q^2) 
\end{equation}
and similarly for $\wnk_p$. \label{choice-of-basis} Here we need to be a bit precise because
this isomorphism depends on two choices. First we need to choose a base point on the
universal cover of (each connected component of) $Y(N)^\an$, which is the upper half plane
$\H$, and we choose $\i\in\H$ for this. Second, we need to choose a basis of the fiber of
the representation $\Sym^{k-2}\RD^1f_*\smuline\Q$ at this base point. Such a choice is
induced by choosing a base of the homology $\HL^1(E_\i,\Q)$ (where
$E_\i=\quotient{\C}{(\Z\oplus\Z\i)}$) and we choose the ordered basis $(\i,1)$ for this
homology group.

The representation $\Sym^{k-2}\Q^2$ of $\GL_2(\Q)$ is canonically isomorphic to the space of
homogeneous polynomials of degree $k-2$ in two variables $X,Y$ over $\Q$, where a matrix
$\binmatrix a b c d\in\GL_2(\Q)$ acts by sending $X$ to $aX+cY$ and $Y$ to $bX+dY$ (and
similarly for other coefficient rings). Thus we will denote elements of $\Sym^{k-2}\Q^2$ as
homogeneous polynomials. In the context of group cohomology double coset operators are
described e.\,g.\ in \cite[§8.3]{MR1291394} or \cite[§4, p.\ 563]{MR848685}, in particular
this describes the action of Hecke and diamond operators. Moreover, the action of complex
conjugation corresponds to the action of the matrix \[ \EP=\binmatrix100{-1} \] (i.\,e.\ the
double coset operator $[\Gamma(N)\EP\Gamma(N)]$).\footnote{This is correct with the choice
  of basis described before; the other obvious choice $(1,\i)$ would lead to $\EP$ being
  replaced by $-\EP$.}

\subsection{The Eichler-Shimura isomorphism and the comparison isomorphism}

The purpose of this section is to prove the well-known fact that the two maps alluded to in
the title are the same.

\subsubsection{The Eichler-Shimura isomorphism}
\label{sec:es}

The Eichler-Shimura isomorphism relates the space of cusp forms to the parabolic cohomology
group of the local system $\Sym^{k-2}\RD^1f_*\smuline\Z$ on a modular curve. We recall the
construction of the Eichler-Shimura map from
\cite[§4.10]{MR2104361}.

In this whole section we write $Y$ for either $Y(N)^\an$ or $Y_1(N)^\an$, $f\colon E\ra Y$ for the universal elliptic curve over it and $X$ for the corresponding compactification.

We have a commutative diagram with exact rows of sheaves on $E$
\begin{equation*}
  \begin{tikzpicture}
    \matrix (m) [matrix of math nodes, row sep=2em, column sep=1.5em, text height=1.5ex, text depth=0.25ex]
    { 0 & \smuline\C & \O_E & \Omega^1_E & 0  \\
      0 & \smuline\C\tensor_{\smuline\C}f^{-1}\O_Y & \O_E & \Omega^1_{E/Y} & 0 \\};
    \path[->,font=\scriptsize]
    (m-1-1) edge (m-1-2)
    (m-1-2) edge (m-1-3)
    (m-1-3) edge (m-1-4)
    (m-1-4) edge (m-1-5)
    (m-2-1) edge (m-2-2)
    (m-2-2) edge (m-2-3)
    (m-1-4) edge (m-2-4)
    (m-2-3) edge (m-2-4)
    (m-2-4) edge (m-2-5)
    (m-1-2) edge (m-2-2);
    \draw [double equal sign distance] (m-1-3) -- (m-2-3);
  \end{tikzpicture}
\end{equation*}
which yields a commutative square of morphisms of complexes of sheaves on $E$
\begin{equation*}
  \begin{tikzpicture}
    \matrix (m) [matrix of math nodes, row sep=2em, column sep=2.7em, text height=1.5ex, text depth=0.25ex]
    { \Omega_E^\bullet & \Omega_{E/Y}^\bullet \\
      \smuline\C[0] & \smuline\C\tensor_{\smuline\C}f^{-1}\O_Y[0] \\};
    \path[->,font=\scriptsize]
    (m-1-1) edge (m-1-2)
    (m-1-1) edge node [left] {qis} (m-2-1)
    (m-1-2) edge node [left] {qis} (m-2-2)
    (m-2-1) edge (m-2-2);
  \end{tikzpicture}
\end{equation*}
in which the vertical maps are quasi-isomorphisms. By applying $\RRD^1f_*$ to this square
and using the projection formula, we obtain the diagram
\begin{equation}
  \label{eqn:commutative-square-diff-sheaves}  
  \begin{tikzpicture}
    \matrix (m) [matrix of math nodes, row sep=2em, column sep=2.7em, text height=1.5ex, text depth=0.25ex]
    { \RRD^1f_*\Omega_E^\bullet & \RRD^1f_*\Omega_{E/Y}^\bullet \\
      \RD^1f_*\smuline\C & \RD^1f_*\smuline\C\tensor_{\smuline\C}\O_Y &
      \RD^1f_*\smuline\R\tensor_{\smuline\R}\O_Y \\};
    \path[->,font=\scriptsize]
    (m-1-1) edge (m-1-2)
    (m-1-1) edge node [left] {$\sim$} (m-2-1)
    (m-1-2) edge node [left] {$\sim$} (m-2-2)
    (m-2-1) edge (m-2-2);
    \draw [double equal sign distance] (m-2-2) -- (m-2-3);
  \end{tikzpicture}
\end{equation}
in which the vertical maps are isomorphisms.
The Hodge filtration of $\RRD^1f_*\Omega^\bullet_{E/Y}$ is given by the injection
\begin{equation}
  \label{eqn:hodge-filt-e}
  \omega_{E/Y}\inj\RRD^1f_*\Omega^\bullet_{E/Y},
\end{equation}
which comes from applying $\RRD^1f_*$ to the morphism of complexes \[ \Omega^1_{E/Y}[-1] \ra \Omega^\bullet_{E/Y}. \]
We take $(k-2)$-th powers and then tensor with $\Omega^1_{Y}$ over $\O_{Y}$. Using the above isomorphism, this gives us a map
\begin{multline}\label{eqn:pre-eichler-sh-a}
  \omega_{E/Y}^{ k-2}\tensor_{\O_{Y}}\Omega_{Y}^1\ra\Sym^{k-2}_{\O_Y}(\RRD^1f_*\Omega^\bullet_{E/Y})\tensor_{\O_Y}\Omega^1_Y\\
  \isom\Sym^{k-2}_{\O_Y}(\RD^1f_*\smuline\R\tensor_{\smuline\R}\O_{Y})\tensor_{\O_Y}\Omega^1_Y
  =\Sym^{k-2}_{\smuline\R}\RD^1f_*\smuline\R\tensor_{\smuline\R}\Omega^1_{Y}.
\end{multline}
Let us write $\mathcal{D}^k_R$ for $\Sym^{k-2}_{\smuline R}\RD^1f_*\smuline R$ for $R=\R$ or $R=\C$ in the following.

On the other hand, we have an exact sequence \[ 0\ra\smuline\C\ra\O_{Y}\ra[$d$]\Omega^1_{Y}\ra0 \] which we tensor over $\smuline\R$ with the (locally free, hence flat) sheaf $\mathcal{D}^k_\R$ to obtain
\begin{equation}\label{eqn:pre-eichler-sh-b}
  0\ra \mathcal{D}^k_\C\ra \mathcal{D}^k_\R\tensor_{\smuline\R}\O_{Y}\ra[$d$]\mathcal{D}^k_\R\tensor_{\smuline\R}\Omega^1_{Y}\ra0.
\end{equation}

\begin{dfn}\label{dfn:es}
  The {Eichler-Shimura map} is defined to be the composition
  \begin{multline}\label{eqn:es-map}
    \S_k(\Gamma,\C)=\HL^0(X,\omega_{\overline E/X}^{ k-2}\tensor_{\O_X}\Omega^1_{X})\ra \HL^0(Y,\omega_{E/Y}^{ k-2}\tensor_{\O_Y}\Omega^1_{Y})\\\ra \HL^0(Y,\mathcal{D}^k_\R\tensor_{\smuline\R}\Omega^1_{Y})\ra \HL^1(Y,\mathcal{D}^k_\C)
  \end{multline}
  where the first map is the restriction map, the second map is induced by \eqref{eqn:pre-eichler-sh-a} and the third map is the boundary homomorphism in the long exact cohomology sequence attached to \eqref{eqn:pre-eichler-sh-b}.
  Then its image lies in the parabolic cohomology group $\Hp^1(Y,\mathcal{D}^k_\C)$, so the
  Eichler-Shimura map goes
  \[ {\ES}\colon\S_k(\Gamma,\C) \ra \Hp^1(Y,\Sym^{k-2}\RD^1f_*\smuline \C). \]
\end{dfn}

Denote by ${\overline{\S_k(\Gamma,\C)}}$ the {complex conjugate vector space} of
$\S_k(\Gamma,\C)$, which we identify with the space of antiholomorphic modular forms.

\begin{thm}\label{thm:eichler-shimura-conrad}
  The Eichler-Shimura map induces a Hecke equivariant isomorphism of complex vector spaces 
  \[ \ES\oplus\overline\ES\colon\S_k(\Gamma,\C)\oplus\overline{\S_k(\Gamma,\C)}\isom\Hp^1(Y,\Sym^{k-2}\RD^1f_*\smuline \C). \] 
  There is also the variant for modular forms
  \[ \ES\oplus\overline\ES\colon\ModForms_k(X,\C)\oplus\overline{\S_k(\Gamma,\C)}\isom\Hc^1(Y,\Sym^{k-2}\RD^1f_*\smuline \C). \]
\end{thm}


There is a more classical description of the Eichler-Shimura isomorphism using group
cohomology which has the advantage of being rather explicit. This description is used
e.\,g.\ in the classical reference \cite[chap.\ 8]{MR1291394}, although in a slightly
different formulation. We briefly describe it here, for simplicity only for the
$\Gamma_1(N)$ situation. So for the moment we specialize to the case $X=X_1(N)^\an$,
$Y=Y_1(N)^\an$ and put $\Gamma=\Gamma_1(N)$. As explained at the end of \cref{sec:wnk}
\begin{equation}
  \label{eqn:identif-gp-cohom}
  \Hp^1(Y,\Sym^{k-2}\RD^1f_*\smuline\C) \cong   \Hp^1(\Gamma,\Sym^{k-2}\C^2).
\end{equation}
We view the element of the latter grounp as inhomogeneous cocycles on $\Gamma$ with values
in the space of homogeneous polynomials over $\C$ of degree $k-2$ in two variables $X,Y$.

\begin{prop}\label{prop:es-explicit}
  \begin{enumerate}
  \item The Eichler-Shimura isomorphism followed by \eqref{eqn:identif-gp-cohom}
    maps $f\in\S_k(\Gamma,\C)$ to the cocycle
    \[ \gamma\mapsto\int_{\tau_0}^{\gamma \tau_0} \omega_f ,\qquad\gamma\in\Gamma \] with
    ${\omega_f}=(2\pi\i)^{k-1}(zX+Y)^{k-2}f\d z$ (a $\Sym^{k-2}\C^2$-valued $1$-form) and
    $\tau_0\in\H$ being a lift of the base point in $Y$.
  \item The Eichler-Shimura isomorphism followed by \eqref{eqn:identif-gp-cohom}
    maps $\overline g\in\overline{\S_k(\Gamma,\C)}$ to the cocycle
    \[ \gamma\mapsto\int_{\tau_0}^{\gamma \tau_0} \omega_{\overline g} ,\qquad\gamma\in\Gamma \]
    with
    $\omega_{\overline g}=(-2\pi\i)^{k-1}(\overline zX+Y)^{k-2}\overline g\d \overline z$
    and $\tau_0$ as before.
    Here, the $\overline g$ in the definition of $\omega_{\overline g}$ literally means
    the complex conjugate function of the $\C$-valued function $g$.
  \end{enumerate}
\end{prop}
\begin{proof}
  The first statement is well-known (although note that the precise formula depends on the
  choice of a basis described on \cpageref{choice-of-basis}).
%
%
  The second statement follows easily from the first one applying the definitions and basic
  properties of the complex curve integral.
\end{proof}

Now we examine how the complex conjugation on $\Hp^1(Y,\Sym^{k-2}\RD^1f_*\smuline\Z)$ behaves under the Eichler-Shimura isomorphism. This works again in the general setting, so $X$ may now be $X(N)^\an$, $Y$ may be $Y(N)^\an$ and so on. The action of $\GR$ on $Y$ and the fact that $\Sym^{k-2}\RD^1f_*\smuline\Z$ is a $\GR$-sheaf gives an action of $\GR$ on $\Hp^1(Y,\Sym^{k-2}\RD^1f_*\smuline\Z)$ and hence also on $\Hp^1(Y,\Sym^{k-2}\RD^1f_*\smuline\C)=\Hp^1(Y,\Sym^{k-2}\RD^1f_*\smuline\Z)\tensor\C$, where we let $\GR$ act \emph{trivially} on the second tensor factor $\C$.

\begin{dfn}\label{dfn:gr-on-mod-forms}
  We define an action of $\GR$ on $\S_k(\Gamma,\C)\oplus\overline{\S_k(\Gamma,\C)}$ by letting the nontrivial element act as
  \[ f\oplus \overline g \mapsto -(g^*\oplus\overline{f^*}), \]
  where $f^*$ denotes the dual cusp form
  \[ f^*=\sum_{n=1}^\infty \overline a_n q^n \] of a cusp form\footnote{This is to be
    understood in each of the $\phi(N)$ components separately in the case
    $\Gamma=\Gamma(N)$.} \[ f=\sum_{n=1}^\infty a_n q^n\in\S_k(\Gamma,\C). \]
\end{dfn}
\begin{lem}\label{lem:es-gr-equiv}
  With the above definition, the Eichler-Shimura isomorphism is $\GR$-equivariant.
\end{lem}
\begin{proof}
  To simplify notation, we give the argument only in the case $X=X_1(N)$, so $\Gamma=\Gamma_1(N)$; the proof for $X=X(N)$ works similar.
  For $g\in\S_k(\Gamma_1(N),\C)$ let 
  \begin{equation*}
    u_{\overline g}\colon\gamma\mapsto\int_{\tau_0}^{\gamma \tau_0} \omega_{\overline g} ,\qquad\gamma\in\Gamma,
  \end{equation*}
  with $\omega_{\overline g}=(-2\pi\i)^{k-1}(\overline zX+Y)^{k-2}\overline g\d\overline z$
  being the cocycle from \cref{prop:es-explicit}. We have to check that complex conjugation on
  $\Hp^1(\Gamma,\Sym^{k-2}\C^2)$ sends $u_f$ to $-u_{\overline{f^*}}$ for
  $f\in\S_k(\Gamma_1(N),\C)$.
  The action of complex conjugation is given by the matrix $\EP$, explicitely it
  maps a cocycle $u$ to the cocycle
  \begin{equation*}
    \gamma \mapsto \EP u(\EP\gamma\EP).
  \end{equation*}
  Applying this to $u_f$ the claim follows from a straightforward calculation.
  
\end{proof}

\subsubsection{The de Rham realization}
\label{sec:dr}

For describing the de Rham realization we need to use the language of log schemes
\cite{MR1463703}; we follow \cite[§1.2.4]{MR2103471}.

The cuspidal divisor $C=X(N)\setminus Y(N)$ defines a logarithmic structure on $X(N)$ and its preimage in $\overline E(N)$ defines a logarithmic structure on $\overline E(N)$.
We denote by $\logOmega$ the sheaf of logarithmic relative differentials, see
\cite[§1.7]{MR1463703} (it is denoted $\omega$ in the reference
\cite{MR1849260} and $\Omega$ in \cite{MR1463703}, but we
want to avoid any confusion with the usual differentials or the sheaf $\omega$ defined in
\cref{sec:modular-forms}).

We consider the logarithmic de Rham complex
\[ \logOmega^\bullet_{\overline E(N)/X(N)}=(\degzero{\O_{X(N)}}\ra[d]\logOmega^1_{\overline
    E(N)/X(N)}) \] (here and in the following, a \enquote{$0$} below a complex indicates
degree $0$) and put
\[ \mathcal{E}=\RRD^1f_*\logOmega^\bullet_{\overline E(N)/X(N)}, \qquad \mathcal{E}_k =
  \Sym_{\O_{X(N)}}^{k-2}\!\!\mathcal{E}, \qquad \mathcal E_{k,\cp}=\mathcal E_k(-C). \] The
{logarithmic Gauß-Manin connection} on $\mathcal E$
\[ \nabla\colon \mathcal{E} \ra\mathcal E \tensor_{\O_{X(N)}} \logOmega^1_{X(N)} \] induces
flat logarithmic connections on $\mathcal E_k$ and $\mathcal E_{k,\cp}$ which we denote by
$\nabla_k$ and $\nabla_{k,\cp}$, respectively.  We denote the complexes of sheaves on $X(N)$
defined by these connections by $\mathcal E^\bullet_k$ and $\mathcal E^\bullet_{k,\cp}$,
respectively.

\begin{thm}\label{thm:wnk-dr}
  The de Rham realization of $\wnk$ is \[ \wnk_\dR=\operatorname{image}(\HHL^1(X(N), \mathcal{E}^\bullet_{k,\cp})\ra\HHL^1(X(N), \mathcal{E}_k^\bullet))\tensor\Q. \]
\end{thm}
\begin{proof}
  \cite[§2.2, p.\ 15]{MR1849260}
\end{proof}

We will also need the Hodge realization.

\begin{prop}\label{thm:wnk-hodge-filt}
  The {Hodge filtration} of the de Rham realization of $\wnk$ is given by
  \[ \fil^i\wnk_\dR =
  \begin{cases}
    \; \wnk_\dR & \text{for } i\le0,\\
    \; \S_k(\Gamma(N),\Q) & \text{for } 1\le i\le k-1,\\
    \; 0 & \text{for } i\ge k.
  \end{cases}
  \]
\end{prop}
\begin{proof}
  \cite[(11.2.5)]{MR2104361}
\end{proof}

We will now describe explicitly where the embedding of the intermediate filtration step
$\S_k(\Gamma(N),\Q)$ comes from. Over $Y(N)$, $\omega_{E(N)/Y(N)}$ is isomorphic to
$f_*\Omega^1_{E(N)/Y(N)}$ and $\mathcal E$ is isomorphic to
$\RRD^1f_*\Omega^\bullet_{E(N)/Y(N)}$, and this gives a canonical map
\[ \omega_{E(N)/Y(N)}\ra\mathcal E\restrict{Y(N)} \]
just as in \eqref{eqn:hodge-filt-e}, which is just the Hodge filtration of $\mathcal E\restrict{Y(N)}$. It can be extended to the cusps to give a morphism
\begin{equation}
  \label{eqn:map-differentials-complex}
  \omega_{\overline E(N)/X(N)}^{ k-2}\tensor_{\O_{X(N)}}\Omega^1_{X(N)}\ra \mathcal E_{k,\cp} \tensor_{\O_{X(N)}} \logOmega^1_{X(N)}.
\end{equation}
The injection 
\[ \S_k(\Gamma(N),\Z) \inj \HHL^1(X(N), \mathcal{E}_{k,\cp}^\bullet) \]
is then obtained by considering $\omega_{\overline E(N)/X(N)}^{
  k-2}\tensor_{\O_{X(N)}}\Omega^1_{X(N)}$ as a complex concentrated in degree $1$, which via
\eqref{eqn:map-differentials-complex} gives us a morphism of complexes \[ (\omega_{\overline
    E(N)/X(N)}^{ k-2}\tensor_{\O_{X(N)}}\Omega^1_{X(N)})[-1] \ra \mathcal
  E^\bullet_{k,\cp} \] to which we can apply $\HHL^1$.

We also want to say how the cokernel of the above injection, i.\,e.\ $\gr^0\wnk_\dR$, looks like.
\begin{prop}\label{prop:gr-wnk-sk-dual}
  We have canonically  \[ \gr^0\wnk_\dR = \HL^1(X(N)_{/\Q},(\omega_{\overline E(N)/X(N)}^{2-k})_{/\Q}) \cong \S_k(\Gamma(N),\Q)^\vee. \]
  Here $(\cdot)^\vee$ denotes the dual $\Q$-vector space and the last isomorphism comes from Serre duality.
\end{prop}
\begin{proof}
  \cite[p.\ 15/16]{MR1849260}
\end{proof}

\begin{cor}\label{cor:hodge-realization}
  The {Hodge realization of $\wnk$} is \[ \wnk_\hodge = \HL^1(X(N)_{/\Q},(\omega_{\overline E(N)/X(N)}^{2-k})_{/\Q}) \oplus \S_k(\Gamma(N),\Q) \] with the first summand sitting in degree $0$ and the second summand sitting in degree $k-1$.
\end{cor}

\subsubsection{The complex comparison isomorphism}
\label{sec:comparison-iso}

A description of the {complex comparison isomorphism of $\wnk$} is given in
\cite[§2.2]{MR1849260}. We recall this here and refer to this text for more details.

Write $j\colon Y(N)\ra X(N)$ for the inclusion. Via the analytification map $X(N)^\an\ra X(N)\fibertimes_\Z\C$, we can pull back all the involved sheaves and connections to the analytic setting. By abuse of notation, 
we write $j$ also for the analytic inclusion.
In \cite[§2.2]{MR1849260} it is proved that this does not change the cohomology groups and the de Rham realization (over $\C$, of course). So we can work in the analytic category and write for the rest of the section $X=X(N)^\an$, $Y=Y(N)^\an$ and $E=E(N)^\an$.

In this section, let us write again $\mathcal{D}^k_\C=\Sym_{\smuline\C}^{k-2}\RD^1f_*\smuline\C$, as we did in \cref{sec:es}. 
We further use some of the notation from \cref{sec:dr}. Write $\mathcal G_k^\bullet$ for the restriction of $\mathcal E_k^\bullet$ or $\mathcal E_{k,\cp}^\bullet$ to $Y$ (both are the same), which is by definition just the complex \begin{equation*} (\degzero{\mathcal G_k}\shortra\mathcal G_k\tensor_{\O_{Y}}\Omega^1_{Y}),\quad \text{with } \mathcal G_k=\Sym_{\O_{Y}}^{k-2}\RRD^1f_*\Omega^\bullet_{E/Y}. \end{equation*}

Using \eqref{eqn:commutative-square-diff-sheaves} and taking $(k-2)$-th symmetric powers, one can construct an exact sequence
\begin{equation}\label{eqn:exact-seq-comp-iso-mn}
  0\ra \mathcal{D}^k_\C \ra \mathcal{G}_{k}\ra \mathcal{G}_{k}\tensor_{\O_{Y}}\Omega^1_{Y}\ra0,
\end{equation}
where the right map is just the restriction of the connection $\nabla_k$ to $Y$. This means that the two right entries in this sequence are precisely the complex $\mathcal G^\bullet_k$, so this gives rise to a quasi-isomorphism of complexes \[ \mathcal G^\bullet_{k}\ra \mathcal{D}^k_\C[0] \] which after applying $\HHL^1$ gives an isomorphism \[ \HHL^1(Y, \mathcal{G}^\bullet_{k})\isom\HL^1(Y, \mathcal{D}^k_\C). \]
The exact sequence above can be extended to the cusps to
\begin{equation}
  \label{eqn:comp-iso-complex}
  0\ra j_!\mathcal{D}^k_\C \ra \mathcal{E}_{k,\cp}\ra \mathcal{E}_{k,\cp} \tensor_{\O_{X}}\logOmega^1_{X}\ra0
\end{equation}
where the right map is now $\nabla_{k,\cp}$. This gives in the same way an isomorphism \[ \HHL^1(X, \mathcal{E}^\bullet_{k,\cp})\isom\Hc^1(Y,\mathcal{D}^k_\C). \]
Finally, the restriction morphism \[ \HHL^1(X, \mathcal{E}^\bullet_{k})\ra\HHL^1(Y, \mathcal{G}^\bullet_{k}) \] is an isomorphism by \cite[§2.4, p. 21]{MR1849260}.

We assemble the isomorphisms we have so far in a diagram
\begin{equation}\label{eqn:diagram-cp-iso}
\begin{tikzpicture}
  \matrix (m) [matrix of math nodes, row sep=2em, column sep=1.5em, text height=1.5ex, text depth=0.25ex]
  { \HHL^1(X, \mathcal E_{k,\cp}^\bullet) & \HHL^1(X, \mathcal{E}^\bullet_{k}) & \HHL^1(Y, \mathcal G_{k}^\bullet) \\
    \Hc^1(Y,\mathcal{D}^k_\C) & &  \HL^1(Y, \mathcal{D}^k_\C). \\};
  \path[->,font=\scriptsize]
  (m-1-1) edge node [above] {(1)} (m-1-2)
  (m-1-2) edge node [above] {$\sim$} (m-1-3)
  (m-1-1) edge node [left] {$\sim$} (m-2-1)
  (m-1-3) edge node [left] {$\sim$} (m-2-3)
  (m-2-1) edge node [above] {(2)} (m-2-3);
\end{tikzpicture}
\end{equation}
The image of the map (1) is $\wnk_\dR\tensor\C$ and the image of the map (2) is $\wnk_\betti\tensor\C$. Therefore this provides us with the desired comparison isomorphism \[ \wnk_\dR\tensor\C\isom\wnk_\betti\tensor\C. \]

\subsubsection{Relating Eichler-Shimura and the comparison isomorphism}
\label{sec:de-rham-eichler-shimura}

In this section we prove the compatibility of the Eichler-Shimura isomorphism with the
canonical map from cusp forms to the de Rham realization of $\wnk$ from the Hodge filtration
(see \cref{thm:wnk-hodge-filt}) and the comparison isomorphism from
\cref{sec:comparison-iso}. This is also stated without proof in
\cite[§11.3]{MR2104361}. 

\begin{thm}\label{thm:compatibility-es-comparison}
  The diagram
  \[\begin{tikzpicture}
    \matrix (m) [matrix of math nodes, row sep=2.5em, column sep=0pt, text height=1.5ex, text depth=0.25ex]
    { & \S_k(\Gamma(N),\C) \\ 
      \wnk_\dR\tensor\C && \wnk_\betti\tensor\C \\ };
    \path[->,font=\scriptsize]
    (m-1-2) edge node[left] {Hodge$\,$} (m-2-1)
    (m-1-2) edge node[right] {$\,\ES$} (m-2-3)
    (m-2-1) edge node[above] {$\sim$} (m-2-3);
  \end{tikzpicture}\]
  commutes. Here the left map comes from the Hodge filtration of $\wnk_\dR$ as in \cref{thm:wnk-hodge-filt}, the right map is the Eichler-Shimura map and the bottom map is the comparison isomorphism.
\end{thm}
\begin{proof}
  In this proof, we abbreviate $X=X(N)^\an$, $Y=Y(N)^\an$ and $E=E(N)^\an$.


  From \eqref{eqn:commutative-square-diff-sheaves} we have an isomorphism
  \begin{equation}
    \label{eqn:isom-gk-dk}
    \mathcal G_k 
    \cong\mathcal{D}_\C^k\tensor_{\smuline\C}\O_{Y}
    \end{equation}
and this allows us to identify the exact sequence \eqref{eqn:pre-eichler-sh-b} used in the definition of the Eichler-Shimura map with the exact sequence \eqref{eqn:exact-seq-comp-iso-mn} used to define the comparison isomorphism to obtain a commutative diagram with exact rows
\begin{equation}\label{eqn:identif-exact-sequences}
\begin{tikzpicture}
  \matrix (m) [matrix of math nodes, row sep=2em, column sep=1.5em, text height=1.5ex, text depth=0.25ex]
  { 0 & \mathcal D_\C^k & \mathcal D_\C^k\tensor_{\smuline\C}\O_Y & \mathcal D_\C^k\tensor_{\smuline\C}\Omega^1_Y & 0  \\
    0 & \mathcal D_\C^k & \mathcal G_k & \mathcal G_k\tensor_{\O_Y}\Omega^1_Y & 0. \\};
  \path[->,font=\scriptsize]
  (m-1-1) edge (m-1-2)
  (m-1-2) edge (m-1-3)
  (m-1-3) edge (m-1-4)
  (m-1-4) edge (m-1-5)
  (m-2-1) edge (m-2-2)
  (m-2-2) edge (m-2-3)
  (m-1-4) edge node [right] {$\sim$} (m-2-4)
  (m-2-3) edge (m-2-4)
  (m-2-4) edge (m-2-5)
  (m-1-3) edge node [right] {$\sim$} (m-2-3);
  \draw [double equal sign distance] (m-1-2) -- (m-2-2);
\end{tikzpicture}
\end{equation}

  Using this and the definition of the map from the Hodge filtration (labelled \enquote{Hodge} below), we extend the diagram \eqref{eqn:diagram-cp-iso}, which is part (1) below, to
  
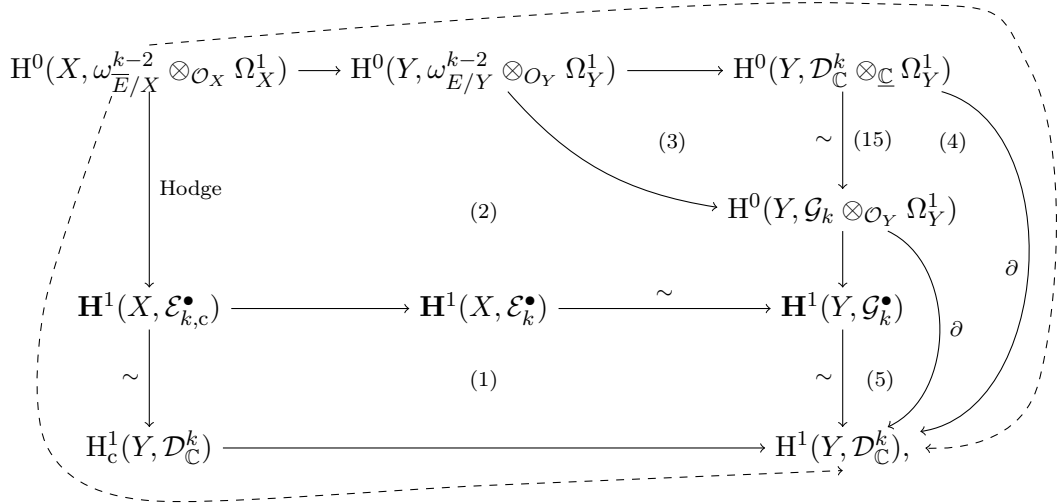
\begin{equation*}
    \begin{tikzpicture}
      \matrix (m) [matrix of math nodes, row sep=1em, column sep=.75em, text height=1.5ex, text depth=0.25ex]
      { \HL^0(X,\omega_{\overline E/X}^{k-2}\tensor_{\O_X}\Omega_X^1) & & \HL^0(Y,\omega_{E/Y}^{k-2}\tensor_{O_Y}\Omega^1_Y) & &  \HL^0(Y,\mathcal D^k_\C\tensor_{\smuline\C}\Omega^1_Y) \\
        & & & \text{\scriptsize(3)} & & \hspace*{-1cm}\text{\scriptsize(4)} \\
        & & \text{\scriptsize(2)} & & \HL^0(Y,\mathcal G_k\tensor_{\O_Y} \Omega_Y^1) \\
        \\
        \HHL^1(X, \mathcal E_{k,\cp}^\bullet) & & \HHL^1(X, \mathcal{E}^\bullet_{k}) & & \HHL^1(Y, \mathcal G_{k}^\bullet) \\
        & & \text{\scriptsize(1)} & & \hspace*{1cm}\text{\scriptsize(5)} \\
        \Hc^1(Y,\mathcal{D}^k_\C) & & \; & &  \HL^1(Y, \mathcal{D}^k_\C), \\};
      \path[->,font=\scriptsize]
      (m-1-1) edge node[right] {Hodge} (m-5-1)
      (m-1-1) edge (m-1-3)
      (m-1-3) edge (m-1-5)
      (m-1-5) edge node[left] {$\sim$} node[right] {\eqref{eqn:isom-gk-dk}} (m-3-5)
      (m-3-5) edge (m-5-5)
      (m-1-3) edge [bend right=20] (m-3-5)
      (m-5-1) edge (m-5-3)
      (m-1-5) edge [bend left=78] node[left] {$\partial$} (m-7-5)
      (m-3-5) edge [bend left=65] node[right] {$\partial$}  (m-7-5)
      (m-5-3) edge node [above] {$\sim$} (m-5-5)
      (m-5-1) edge node [left] {$\sim$} (m-7-1)
      (m-5-5) edge node [left] {$\sim$} (m-7-5)
      (m-7-1) edge (m-7-5);
      \draw[dashed, ->] plot[smooth,tension=.7]
      coordinates{([yshift=.34cm] m-1-1) ([yshift=0.9cm] m-1-5) ([xshift=2.5cm] m-1-5) ([xshift=2.6cm, yshift=-1cm] m-5-5) ([xshift=1.1cm] m-7-5)};
      \draw[dashed, ->] plot[smooth,tension=.5]
      coordinates{([xshift=-.5cm, yshift=-0.34cm] m-1-1) ([xshift=-1.4cm] m-5-1) ([xshift=-1.35cm] m-7-1) ([yshift=-.7cm] m-7-1) ([yshift=-.5cm] m-7-3) ([yshift=-0.3cm] m-7-5)};
    \end{tikzpicture}
  \end{equation*}
  where the two maps labelled \enquote{$\partial$} are boundary maps in the long exact cohomology sequences attached to the short exact sequences \eqref{eqn:pre-eichler-sh-b} and \eqref{eqn:exact-seq-comp-iso-mn}, respectively.

  We want to prove that the outermost (dashed) arrows coincide, since the composition along the lower dashed arrow is the composition of the Hodge filtration map with the comparison isomorphism, while the composition along the upper dashed arrow is the Eichler-Shimura map. We prove this by showing that each of the partial diagrams (1)--(5) commutes.

  We know already that (1) commutes. That (2) commutes is clear since both ways are basically the same map. Part (3) commutes just by definition of the map $\omega_{E/Y}^{k-2}\tensor_{\O_{Y}}\Omega_{Y}^1\ra\mathcal D_\C^k\tensor_{\smuline\C}\Omega^1_{Y}$ in \eqref{eqn:pre-eichler-sh-a}. Part (4) commutes because \eqref{eqn:identif-exact-sequences} commutes.

  To see that (5) commutes, we are hence left to prove that the boundary map
  \begin{equation*}
    \HL^0(Y,\mathcal G_k\tensor_{\O_{Y}}\Omega^1)\ra\HL^1(Y,\mathcal{D}^k_\C)
  \end{equation*}
  for the exact sequence \eqref{eqn:exact-seq-comp-iso-mn} is equal to the composition of the map
  \begin{equation*}
    \HL^0(Y,\mathcal G_k\tensor_{\O_{Y}}\Omega^1)\ra\HHL^1(Y,\mathcal G_k^\bullet)
  \end{equation*}
  induced by the inclusion of complexes
  \begin{equation*}
    (\mathcal G_k\tensor_{\O_{Y}}\Omega^1)[-1]\ra \mathcal G_k^\bullet
  \end{equation*}
  with the comparison isomorphism. By the description of the comparison isomorphism in
  \cref{sec:comparison-iso}, 
  this follows from a standard fact in homological algebra, see \cite[Ex.\
  1.5.6]{MR1269324}.
\end{proof}

\subsection{Refinements}
\label{sec:refinements}

Recall that we fixed a number field $K$ with ring of integers $\O_K$.

From the classical viewpoint on modular forms as functions on the upper half plane it is
clear that we have an inclusion of $\O_K$-modules
$\S_k(\Gamma_1(N),\O_K)\inj\S_k(\Gamma_1(Np),\O_K)$. One has to be careful here because this
inclusion does \emph{not} respect the action of $T_p$ (but it does respect all the other
Hecke operators).

Let $f\in\S_k(\Gamma_1(N),\psi,\O_K)$ be an eigenform away from the level (i.\,e.\ it is an
eigenvector of all the operators $T_\ell$ for all primes not dividing $N$ and $\diamondop d$ for
$d\in\zmodmal N$) and denote the eigenvalue of
$T_\ell$ by $a_\ell$ for primes $\ell\nmid N$. We look at the polynomial 
\begin{equation}
  \label{eqn:polynomial-alpha-beta}
   X^2-a_pX+\psi(p)p^{k-1}
\end{equation}
and call its roots $\alpha$ and $\beta$; we assume without loss of generality that $\O_K$ is
large enough to contain  both of them. Then define two functions on the upper half plane by
\[ f_\alpha(\tau)=f(\tau)-\beta f(p\tau),\quad f_\beta(\tau)=f(\tau)-\alpha
  f(p\tau)\quad(\tau\in\H). \]
The following statement is well-known and easy to check.

\begin{prop}\label{lem:refinements-are-modular-forms}
  The functions $f_\alpha$ and $f_\beta$ define cusp forms in $\S_k(\Gamma_1(Np),\psi,\O_K)$,
  where $\psi$ is now viewed as a character of $\zmodmal{Np}$. They are eigenforms away from
  the level with the same eigenvalues as $f$ and they are moreover eigenvectors of $T_p$
  with eigenvalues $\alpha$ and $\beta$, respectively. If $f$ was a normalized eigenform,
  then so are $f_\alpha$ and $f_\beta$, with the same eigenvalues for all Hecke operators
  except $T_p$.
\end{prop}

\begin{dfn}\label{dfn:refinements}
  The forms $f_\alpha$ and $f_\beta$ are called the refinements of $f$ at $p$. They are also commonly referred to as the $p$-stabilisations of $f$.
\end{dfn}

Note that if $p\mid N$ then one of $\alpha$ and $\beta$ equals $a_p$ and the other one is
$0$, so one of $f_\alpha$ and $f_\beta$ is just $f$ itself. 
Also note that if $f$ is ordinary at $p$, then exactly one of $\alpha$, $\beta$ is a unit in
$\O_K$ (called the unit root; without loss of generality assume it is $\alpha$). In this case
$f_\alpha$ is again ordinary, while $f_\beta$ isn't, so there is a unique ordinary refinement.


Fix integers $M,N\ge4$ with $N\mid M$. Using the moduli description we define two morphisms between modular curves
\begin{align*}
   \sigma_{M,N}\colon X(M)\ra X(N),\quad &(E,P,Q)\mapsto\left(E,\frac M N P,\frac M N Q\right),\\
   \theta_{M,N}\colon X(M)\ra X(N),\quad &(E,P,Q)\mapsto\left(\quotient E {NP},
                                          P,
                                          Q\right),
\end{align*}
and similarly, denoted by the same symbols,
\begin{align*}
   \sigma_{M,N}\colon X_1(M)\ra X_1(N),\quad &(E,P)\mapsto\left(E,\frac M N P\right),\\
  \theta_{M,N}\colon X_1(M)\ra X_1(N),\quad &(E,P)\mapsto\left(\quotient E {NP},
                                              P\right),
\end{align*}
and call them the {change of level morphisms}.
Further we define morphisms 
\begin{equation}\label{eqn:dfn-change-of-level-morphisms}
  \begin{aligned}
    \Sigma_{M,N}\colon \overline E(M)\ra \overline E(N),\quad &(E,P,Q,x)\mapsto\left(E,\frac M N P,\frac M N Q,x\right),\\
    \Theta_{M,N}\colon \overline E(M)\ra \overline E(N),\quad &(E,P,Q,x)\mapsto\left(\quotient E {NP},
      P,
      Q,x\right),\\
    \Sigma_{M,N}\colon \overline E_1(M)\ra \overline E_1(N),\quad &(E,P,x)\mapsto\left(E,\frac M N P,x\right),\\
    \Theta_{M,N}\colon \overline E_1(M)\ra \overline E_1(N),\quad &(E,P,x)\mapsto\left(\quotient E {NP},
      P,x\right)
  \end{aligned}
\end{equation}
lying over the ones from before (where in each case $x$ denotes a point of $E$).
It is easy to see that the diagrams
\begin{equation}
  \label{eqn:level-change-diagram-cartesian}
  \begin{tikzpicture}
    \matrix (m) [matrix of math nodes, row sep=2em, column sep=1.5em, text height=1.5ex, text depth=0.25ex]
    {  \overline E_?(M) &  \overline E_?(N) \\
      X_?(M) &  X_?(N), \\};
    \path[->,font=\scriptsize]
    (m-1-1) edge (m-1-2)
    (m-1-1) edge (m-2-1)
    (m-1-2) edge (m-2-2)
    (m-2-1) edge (m-2-2);
  \end{tikzpicture}
\end{equation}
where \enquote{$?$} is either $1$ or nothing, the vertical maps are $f$ and the horizontal ones are either $\Sigma_{M,N}$ and $\sigma_{M,N}$ or $\Theta_{M,N}$ and $\theta_{M,N}$,
are cartesian in all cases.

\begin{prop}\label{prop:change-of-level-on-motives}
  For each $k\ge2$, the maps introduced above induce morphisms of motives
  \[ \sigma_{M,N},\theta_{M,N}\colon\wnk \ra\wnklevel M. \]
  We call them {change of level morphisms}, too.
\end{prop}
\begin{proof}
  It is clear that both $\Sigma_{M,N}$ and $\Theta_{M,N}$ induce morphisms $\preKS(M,k)\ra\preKS(N,k)$. By construction it is easy to see that these are equivariant for the actions of the groups $G(M,k)$ resp.\ $G(N,k)$ introduced in \cref{sec:wnk} (via the natural morphism $G(M,k)\ra G(N,k)$). Hence the graphs of these morphisms in $\preKS(N,k)\fibertimes_\Q\preKS(M,k)$ are invariant under $G(N,k)\times G(M,k)$. Since the actions extend to the desingularisations, the same thus holds for the closures of their preimages under the morphism $\KS(N,k)\fibertimes_\Q\KS(M,k)\ra\preKS(N,k)\fibertimes_\Q\preKS(M,k)$, and they are still closed subvarieties of codimension $\dim\KS(N,k)$. Hence we get induced morphisms of motives.
\end{proof}

The morphisms thus induce maps on all realizations of $\wnk$. In particular, by looking at the intermediate step in the Hodge filtration on de Rham realizations and using \cref{thm:wnk-hodge-filt}, we get maps
\begin{equation}
  \label{eqn:level-change-s-k}
  \sigma_{M,N},\theta_{M,N}\colon\S_k(\Gamma(N),\Q)\ra\S_k(\Gamma(M),\Q).
\end{equation}

\begin{prop}\label{prop:change-of-level-pullback-h}
  After tensoring with $\C$, we have:
  \begin{enumerate}
  \item\label{prop:change-of-level-pullback-h:sigma} The map $\sigma_{M,N}\colon\S_k(\Gamma(N),\C)\ra\S_k(\Gamma(M),\C)$ sends an $f$, viewed as function on the upper half plane,\footnote{Strictly speaking, since we use the modular curve $X(N)$ here, $f$ is a $\phi(N)$-tuple of functions on the upper half plane, where $\phi(N)$ is the Euler totient function. The description of the map given here has to be applied to each entry in the tuple.\label{footnote:tuple}}  to itself.
  \item\label{prop:change-of-level-pullback-h:theta} The map $\theta_{M,N}\colon\S_k(\Gamma(N),\C)\ra\S_k(\Gamma(M),\C)$ sends an $f$, viewed as function on the upper half plane, to the function $\tau\mapsto {\left(\frac M N\right)}^kf(\frac M N\tau)$.
  \end{enumerate}
\end{prop}
\begin{proof}
  Let $\vartheta$ stand for either $\sigma_{M,N}$ or $\theta_{M,N}$.
  Because the diagrams \eqref{eqn:level-change-diagram-cartesian} are cartesian and the
  formation of $\omega$ is compatible with base change, we have a canonical isomorphism $\vartheta^*\omega_{\overline E(N)/X(N)}\isom\omega_{X(M)}$, so $\vartheta$ induces a morphism
  \begin{equation}
    \label{eqn:morphism-m-k-m-n}\tag{$*$}
    \HL^0(X(N),\omega_{\overline E(N)/X(N)}^{\tensor k})\ra\HL^0(X(M),\omega_{\overline E(M)/X(M)}^{\tensor k}).
  \end{equation}
  Since the explicit descriptions of the realizations of $\wnk$ come from the Leray spectral
  sequence for the morphism $\KS(N,k)\ra X(N)$ (see \cref{thm:wnk-dr}), the morphism
  \eqref{eqn:level-change-s-k} is the restriction of the morphism
  \eqref{eqn:morphism-m-k-m-n} (tensored with $\Q$) to $\S_k(\Gamma(N),\Q)$. The morphism
  \eqref{eqn:morphism-m-k-m-n} can be constructed analogously with $X(N)$ replaced by
  $X_1(N)$. For simplicity and to avoid the technical complications mentioned in
  \cref{footnote:tuple}, we prove the claim for this morphism instead; it is clear that the
  actual statement can be proved in the same way.

  To prove this, we define  maps
  \begin{align*}
    \Sigma_{M,N}^\an\colon \C\times\H\ra\C\times\H,\quad & (z,\tau)\mapsto(
                                                           z,\tau) \\
    \Theta_{M,N}^\an\colon \C\times\H\ra\C\times\H,\quad & (z,\tau)\mapsto(\textstyle\frac M N z,\textstyle\frac M N \tau) \\
    \sigma_{M,N}^\an\colon \H\ra\H,\quad & \tau\mapsto\tau \\
    \theta_{M,N}^\an\colon \H\ra\H,\quad & \tau\mapsto\textstyle\frac M N \tau
  \end{align*}
  and show that they induce the corresponding maps named in the same way without
  \enquote{$^\an$} on the analytifications of the modular curves. To simplify the notation,
  we henceforth omit the subscripts \enquote{${}_{M,N}$}. First observe that obviously
  $\Sigma^\an$ lies over $\sigma^\an$ and $\Theta^\an$ lies over $\theta^\an$. Next we check
  that $\Sigma^\an$ and $\Theta^\an$ are compatible with the matrix action introduced in
  \eqref{eqn:action-slz-ctimesh}. For $\Sigma^\an$ this is 
  trivial.  For $\Theta^\an$ we note that it is just the action of the matrix
  $\binmatrix{M/N}{}{}{1}$ and that
  \begin{equation*}
    \binmatrix{M/N}{}{}{1}\Gamma_1(M){\binmatrix{M/N}{}{}{1}}^{-1}\subseteq\Gamma_1(N).
  \end{equation*}
  Hence we get induced maps
  \begin{equation*}
    \Sigma^\an,\Theta^\an\colon\lquot{(\C\times\H)}{\Gamma_1(M)}\ra\lquot{(\C\times\H)}{\Gamma_1(N)}
  \end{equation*}
  and analogously with $\sigma^\an$ and $\theta^\an$. Finally we need to check that for the \enquote{relative lattice} $\Lambda\subseteq\C\times\H$ introduced in \eqref{eqn:def-lambda} we have $\Sigma^\an(\Lambda),\Theta^\an(\Lambda)\subseteq\Lambda$. But this is clear from the representation \eqref{eqn:def-lambda-alt} and the definitions of $\Sigma^\an$ and $\Theta^\an$.

  To summarize, we now have commuting maps
  \begin{equation*}
    \begin{tikzpicture}
      \matrix (m) [matrix of math nodes, row sep=2em, column sep=1.5em, text height=1.5ex, text depth=0.25ex]
      {  E_1(M)^\an   &  E_1(N)^\an & & &  E_1(M)^\an   &  E_1(N)^\an\\
        Y_1(M)^\an   & Y_1(N)^\an & & & Y_1(M)^\an   & Y_1(N)^\an. \\};        
      \path[->,font=\scriptsize]
      (m-1-1) edge node [auto] {$\Sigma^\an$} (m-1-2)
      (m-1-1) edge node [left] {$f$} (m-2-1)
      (m-1-2) edge node [left] {$f$} (m-2-2)
      (m-2-1) edge node [above] {$\sigma^\an$} (m-2-2)
      (m-1-5) edge node [auto] {$\Theta^\an$} (m-1-6)
      (m-1-5) edge node [left] {$f$} (m-2-5)
      (m-1-6) edge node [left] {$f$} (m-2-6)
      (m-2-5) edge node [above] {$\theta^\an$} (m-2-6);
    \end{tikzpicture}
  \end{equation*}
  In terms of the moduli description by definition the maps are given by
  \begin{equation*}
    \Sigma^\an\left(E_\tau,\textstyle\frac1M,z\right) = \left(E_\tau,\textstyle\frac1N,
      z\right),\quad \Theta^\an\left(E_\tau,\textstyle\frac1M,z\right) = \left(E_{\frac M N\tau},\textstyle\frac1N,\textstyle\frac M N z\right)
  \end{equation*}
  (where $E_\tau=\C/(\Z\oplus\Z\tau)$ etc.).
  We need to compare this to the definitions in \eqref{eqn:dfn-change-of-level-morphisms}. For $\Sigma$ it is clear that the definitions are compatible because $\frac M N\cdot\frac1 M=\frac 1N$. For $\Theta$ we note that $\frac M N(\Z\oplus\Z\tau)\subseteq\Z\oplus\frac M N\Z\tau$, so the multiplication-by-$\frac M N$ map $\C\ra\C$ induces a surjective homomorphism $E_\tau\surj E_{\frac M N\tau}$ with kernel $\frac N M$. Hence as a point in the moduli space $E_1(N)^\an$ we have
  \begin{equation*}
    \left(E_{\frac M N \tau},\frac1N,\frac M N z\right)=\left(\quotient{E_\tau}{\frac N M},\frac1M,z\right),
  \end{equation*}
  so the definitions of $\Theta$ are also compatible.

  Now statement \ref{prop:change-of-level-pullback-h:sigma} is clear. For
  \ref{prop:change-of-level-pullback-h:theta} we identify $f$ with the differential form
  $\tilde f(2\pi\i\d z)^{\tensor k}$ on $\H$, i.\,e.\ its image under the canonical map
  \[ \S_k(\Gamma_1(N),\C)\ra\HL^0(\H,(\omega_{E_\H/\H})^k) \] (where
  $\omega_{E_\H/\H}\da f_*\Omega^1_{E_\H/\H}$ and $z$ is a coordinate on $\C$). Then
  $\theta$ sends $f$ to
  \begin{equation*}
    (\theta^\an)^*(\tilde f(2\pi\i\d z)^{\tensor k})=(\tilde f\circ\theta^\an)(2\pi\i \d(z\circ\Theta^\an))^{\tensor k}={\left(\frac M N\right)}^k(\tilde f\circ\theta^\an)(2\pi\i \d z)^{\tensor k}.
  \end{equation*}
\end{proof}

\begin{cor}\label{cor:motivic-refinements}
  Let $N\ge4$, 
  $f\in\S_k(\Gamma_1(N)^\arithm,K)$ be an eigenform away from the level and let $\alpha,\beta$ be the roots of the $p$-th Hecke polynomial \eqref{eqn:polynomial-alpha-beta}. Assume that they lie in $K$ and that $K$ contains the $Np$-th roots of unity. Then there exist two canonical morphisms of motives \[ \Ref_\alpha,\Ref_\beta\colon\wnk\tensor_\Q K\ra\wnklevel{Np}\tensor_\Q K \] such that the induced morphisms on the intermediate step in the Hodge filtration on de Rham realizations $\S_k(\Gamma(N),K)\ra\S_k(\Gamma(Np),K)$ map $f$ to its two refinements $f_\alpha$ and $f_\beta$, respectively.
\end{cor}
\begin{proof}
  Just define the morphisms as $\sigma_{Np,N}+p^{-k}\gamma\theta_{Np,N}$ for $\gamma\in\{\alpha,\beta\}$ with $\sigma$ and $\theta$ as in \cref{prop:change-of-level-on-motives}. The claim then follows from the definition of the refinements and \cref{prop:change-of-level-pullback-h}.
\end{proof}

\subsection{Poincar\'e duality}
\label{sec:poincare}

At this point we need the Atkin-Lehner operator and the adjoint Hecke algebra. From the
classical point of view the Atkin-Lehner operator $[w_{N}]$ on $\S_k(\Gamma_1(N),\C)$ is given by
the action of the matrix \[ w_{N}=\binmatrix{0}{-1}{N}{0}. \] Using this we define the
adjoint Hecke operator $T_\ell^\iota$ (for any prime $\ell$) to be
$[w_{N}]T_\ell[w_{N}]^{-1}$ and similarly for diamond operators (for these we actually
have $\diamondop d^\iota=\diamondop d^{-1}$).
Replacing the
original Hecke operators by the adjoint ones in the definition of Hecke algebras we obtain
adjoint Hecke algebras, which we denote by $\heckeT_k^\iota(N,R)$ etc.

The Atkin-Lehner operator has a description as an involution $w_{N}$ of the moduli space
$Y_1(N)^\arithm$ for which we refer to \cite[1.4.2]{FukayaKatoConjecturesSharifi} inducing
endomorphisms of cohomology groups of $Y_1(N)^\arithm$. On modular forms this recovers
the previous classical description.

The theory of Poincare duality for Chow motives (see \cite[§VI.1, §VI.2.1.5]{MR1623774})
applied to $\wnk$ yields the following result (see also \cite[§2.4--5]{MR1849260} for more
details, where our motive $\wnk$ is denoted $M_!$ or $M_{N,!}$).

\begin{thm}
  There is a perfect pairing of motives
  \begin{equation}
    \label{eqn:wnk-poincare-pairing}
    \paarung{\cdot}{\cdot}\colon\wnk \times \wnk \ra \Q(1-k).
  \end{equation}
  The adjoints of the Hecke operators $T_p$ and $\diamondop d$ with respect to this pairing
  are $T_p^\iota$ and $\diamondop d^\iota$, respectively.
\end{thm}

Of course this pairing yields perfect pairings on each realization. Moreover, these pairings
on the realization are compatible with the comparison isomorphisms in the following sense: If $?_1,?_2\in\{\betti, \dR, p\}$ and $B$ is the period ring using to compare the $?_1$- and $?_2$-realizations (i.\,e.\ $B\in\{\C,\Qp,\BdR\})$, then the diagram
\[\begin{tikzpicture}
  \matrix (m) [matrix of math nodes, row sep=6ex, column sep=4em, text height=1.5ex, text depth=0.25ex]
  { (\wnk_{?_1}\tensor B)\times(\wnk_{?_1}\tensor B) & \Q(1-k)_{?_1}\tensor B \\
    (\wnk_{?_2}\tensor B)\times(\wnk_{?_2}\tensor B) & \Q(1-k)_{?_2}\tensor B \\};
  \path[->,font=\scriptsize]
  (m-1-1) edge node [above] {$\paarung{\cdot}{\cdot}_{?_1}$} (m-1-2)
  (m-1-1) edge node [left] {$\sim$} (m-2-1)
  (m-1-2) edge node [left] {$\sim$} (m-2-2)
  (m-2-1) edge node [above] {$\paarung{\cdot}{\cdot}_{?_2}$} (m-2-2);
\end{tikzpicture}\]
commutes (where the vertical maps are the comparison isomorphisms).

Via the Eichler-Shimura isomorphism, we can view the Betti realization of the pairing
$\paarung{\cdot}{\cdot}$ as a pairing on the space
$\S_k(\Gamma,\C)\oplus\overline{\S_k(\Gamma,\C)}$. It is closely related to the Petersson scalar
product on $\S_k(\Gamma,\C)$, which we denote by $\paarung{\cdot}{\cdot}_{\Pet}$.
\begin{prop}\label{prop:petersson-skp}
  We have \[ {\paarung{\ES(f_1\oplus\overline g_1)}{\ES(f_2\oplus\overline g_2)}}=C\cdot\left(\paarung{f_1}{g_2}_{\Pet}+(-1)^{k+1}\paarung{f_2}{g_1}_{\Pet}\right) \]
  with a nonzero constant $C\in\C^\times$ depending only on $k$.
\end{prop}
\begin{proof}
  \cite[(10), p.\ 680]{MR2103471}
\end{proof}

Via the exact sequence
\[ 0\ra\S_k(\Gamma(N),\Q)\ra\wnk_\dR\ra\HL^1(X(N),\omega_{X(N)}^{2-k})\ra0 \] coming from the
Hodge filtration, the pairing on the de Rham realization induces a perfect pairing
\[ \S_k(\Gamma(N),\Q)\times\HL^1(X(N),\omega_{X(N)}^{2-k})\ra\Q(1-k)_\dR \] which coincides with
the Serre duality pairing under the canonical identification $\Q(1-k)_\dR\cong\Q$.  To see
that this is well-defined, it suffices to check that the pairing
${\paarung{\cdot}{\cdot}}_\dR$ restricted to $\fil^0\wnk_\dR=\S_k(\Gamma(N),\Q)$ vanishes, which
by \cref{thm:compatibility-es-comparison} is equivalent to the vanishing of the pairing on
$\S_k(\Gamma(N),\C)$ induced by the pairing ${\paarung{\cdot}{\cdot}}_\betti$ via the
Eichler-Shimura isomorphism, and this follows immediately from
\cref{prop:petersson-skp}. From this we obtain the following corollary.

\begin{cor}\label{cor:gr-zero-mf}
  Let $f$ be a newform with coefficients in the number field $K$.
  The space $\gr^0\Mf_\dR$ is canonically isomorphic to the dual space of the subspace of $\S_k(\Gamma_1(N),K)$ generated by $f[w_N]$.
\end{cor}

Since it may be handy to have the Hecke operators self-adjoint, it is common to modify the
pairing in the following way: Let
\[ {{\paarung{\cdot}{\cdot}}^\iota_\betti}\da{\paarung\cdot{\cdot[w_N]}}_\betti \] where
$w_N$ is the Atkin-Lehner involution (we will use this only on the Betti realization).  This
modified version is also called {twisted Poincar\'e duality pairing} by some
authors. It is clear that the Hecke operators are self-adjoint with respect to it. We will
not use the modified pairing too much, but some statements are easier to formulate using the
modified version. The following statement is easily checked after tensoring with $\C$
using the description in terms of the Petersson scalar product.

\begin{lem}\label{lem:perfectness-restricted-pairing}
  Let $f$ be a newform with coefficients in the number field $K$.
  The pairing ${\paarung{\cdot}{\cdot}}_\betti^\iota$ restricts to perfect pairings
  \[ \Mf_\betti^\pm\times\Mf_\betti^\mp\ra K(1-k), \]
  while its restriction to $\Mf_\betti^\pm\times\Mf_\betti^\pm$ vanishes.
\end{lem}

\section{Modular symbols and Hida families}

\subsection{General and classical modular symbols}

Fix $N\ge4$.
Let $\Sigma\da\GL_2(\Q)\cap\Mat_2(\Z)$.  The group $\GL_2(\Q)$ acts on
$\P^1(\Q)=\Q\cup\{\infty\}$ by fractional linear transformations. Let $\Div(\P^1(\Q))$ be
the free abelian group over the set $\P^1(\Q)$ and let ${\Div^0(\P^1(\Q))}$ be the subgroup
consisting of elements of degree $0$. The left action of $\GL_2(\Q)$ on $\P^1(\Q)$ induces a
left action of $\GL_2(\Q)$ on $\Div(\P^1(\Q))$ and $\Div^0(\P^1(\Q))$.  Let $M$ be an
$R$-module with an $R$-linear right action of $\Sigma$ denoted by
$(m,\alpha)\mapsto m[\alpha]$ for $m\in M$, $\alpha\in\Sigma$. Then
\[ \Hom_\Z(\Div^0(\P^1(\Q)),M)=\Hom_R(R\tensor_\Z\Div^0(\P^1(\Q)),M) \] carries a natural
right $\Sigma$-action defined by
\[ \varphi[\alpha](x) \da \varphi(\alpha x)[\alpha],\quad \text{for }x\in\Div^0(\P^1(\Q)) \]
for $\varphi\in\Hom_\Z(\Div^0(\P^1(\Q)),M)$ and $\alpha\in\Sigma$.  We define the $R$-module
of {modular symbols of level $N$ with coefficients in $M$} to be the $\Gamma_1(N)$-invariants
\[ {\MS(N,M)} \da \HL^0(\Gamma_1(N),\Hom_\Z(\Div^0(\P^1(\Q)),M)). \] On this
module we have Hecke operators $T_n$ for each $n\in\N$ and diamond operators $\diamondop d$
for each $d\in\zmodmal N$ acting as well as an action of the matrix $\EP$. See also
\cite[§2.1]{MR3046279}. We define an action of $\GR$ on $\MS(N,M)$ by letting
the nontrivial element act as $\EP$.

The group $\GL_2(\Q)$ carries the so-called main involution $\iota$ defined by
\[ {\binmatrix a b c d }^\iota := \binmatrix{d}{-b}{-c}{a}. \]
Via this involution we can turn left actions into right actions and vice versa on all the
objects introduced so far. Some texts in the literature use actions from different sides as
we used here, but it is well-known that the resulting spaces of modular symbols and the
Hecke actions on them are the same. From now on we will freely switch between left and right
actions using the involution $\iota$.

The most important special case is that of classical modular symbols, where we use as $M$
the symmetric tensor linear representation $\Sym^{k-2}R^2$ of $\Mat_2(\Z)\cap\GL_2(\Q)$. In
this case we denote it as follows.
\begin{dfn}\label{dfn:classical-ms}
  The $R$-module of {classical modular symbols of weight $k$ and level $N$} is defined as
  \[ {\MS_k(N,R)} \da \MS(N, \Sym^{k-2}R^2). \]
\end{dfn}

From the definition it is clear that we have \[ \MS_k(N,S) = \MS_k(N,R)\tensor_RS \] if $S$ is a flat $R$-algebra.

\begin{prop}\label{prop:iso-classical-ms-honec}
  Write $f\colon E_1(N)^\an\ra Y_1(N)^\an$ for the universal analytic elliptic curve. Then
  there is a canonical Hecke equivariant isomorphism
  \[ \MS_k(N,R) \cong \Hc^1(Y_1(N)^\an, \Sym^{k-2}_{\smuline R}\RD^1f_*\smuline R). \]
\end{prop}
\begin{proof}
%
%
  \cite[Prop.\ 4.2]{MR0860675}
\end{proof}



The above result allows us to relate modular symbols to the Betti realization of $\wnk$.
Consider the composition
\begin{multline}\label{eqn:map-ms-betti}
  \MS_k(N,\Q)\isom\Hc^1(Y_1(N)^\an,\Sym^{k-2}\RD^1f_*\smuline \Q)
\surj\Hp^1(Y_1(N)^\an,\Sym^{k-2}\RD^1f_*\smuline \Q) \\ \inj\Hp^1(Y(N)^\an,\Sym^{k-2}\RD^1f_*\smuline \Q)=\wnk_\betti
\end{multline}
where the first map is from \cref{prop:iso-classical-ms-honec}, the second is tautological
and the last one comes from the morphism $Y(N)\ra Y_1(N)$. This composition is Hecke
equivariant and $\GR$-equivariant. Moreover, the second map is by definition surjective and
the last one is injective since it is so after tensoring with $\C$ (as can be seen using the
Eichler-Shimura isomorphism).

\begin{prop}\label{prop:ms-eigenspaces-free-of-rank-one}
  Let $f\in\S_k(\Gamma_1(N),K)$ be a newform with coefficients in the number field $K$. Then
  the $\O_K$-modules $\MS_k(N,\O_K)^\pm[f]$ are free of rank $1$.
\end{prop}
\begin{proof}
  \cite[Prop.\ 3.3]{MR1279604}
\end{proof}

\begin{lem}\label{lem:ms-wnk-betti-iso}
  For each newform $f\in\S_k(\Gamma_1(N),K)$ with coefficients in the number field $K$, the map \eqref{eqn:map-ms-betti} induces isomorphisms \[ \MS_k(N,K)^\pm[f]\isom\Mf_\betti^\pm. \]
\end{lem}
\begin{proof}
  By Hecke- and $\GR$-equivariance it is clear that we get a map between the spaces in the statement.
  If we look at Hecke eigenspaces in \eqref{eqn:map-ms-betti}, we get
  \begin{multline*}
    \Hc^1(Y_1(N)^\an,\Sym^{k-2}\RD^1f_*\smuline K)[f]\ra\Hp^1(Y_1(N)^\an,\Sym^{k-2}\RD^1f_*\smuline K)[f]  \\\ra\Hp^1(Y(N)^\an,\Sym^{k-2}\RD^1f_*\smuline K)[f].
  \end{multline*}
  That this composition is an isomorphism can be checked after tensoring with $\C$, and then we can use the Eichler-Shimura isomorphisms (\cref{thm:eichler-shimura-conrad}). If we do so, we see first that all spaces involved here are two-dimensional and further that the first map is surjective. The right map is injective since it is the restriction of an injective map. Hence the composition is an isomorphism.
\end{proof}

\begin{dfn}\label{dfn:xi-f}
  Fix $f\in\S_k(\Gamma_1(N),\C)$.
  \begin{enumerate}
  \item The group homomorphism
    \[ {\xi_f}\colon\Div^0(\P^1(\Q))\ra\Sym^{k-2}\C^2,\quad (x)-(y) \mapsto (2\pi\i)^{k-1}\int_y^x(zX+Y)^{k-2}f(z)\d z \]
    is  invariant under the action of $\Gamma_1(N)$, so we have
    \[ \xi_f\in\MS_k(N,\C) \]
    and $\xi_f$ is called the {modular symbol attached to $f$}.
    Moreover one can check that if $f$ is a Hecke eigenform, then
    \[ \xi_f\in\MS_k(N,\C)[f]. \]
  \item Let ${\xi^\pm_f}$ be the image of $\xi_f$ in the respective part of the
    decomposition \[ \MS_k(N,\C)=\MS_k(N,\C)^+\oplus\MS_k(N,\C)^- \]
    into eigenspaces for the action of complex conjugation.
    Note that if $f$ is a Hecke eigenform, then $\xi_f^\pm\in\MS_k(N,\C)^\pm[f]$. 
  \end{enumerate}
\end{dfn}

\begin{lem}\label{lem:xi-maps-to-f}
  Consider the composition
  \begin{multline*}
    \MS_k(N,\C)\isom\Hc^1(Y_1(N)^\an,\Sym^{k-2}\RD^1f_*\smuline\C)\\\surj\Hp^1(Y_1(N)^\an,\Sym^{k-2}\RD^1f_*\smuline\C)\isom\S_k(\Gamma_1(N),\C)\oplus\overline{\S_k(\Gamma_1(N),\C)}
  \end{multline*}
  where the first map is from \cref{prop:iso-classical-ms-honec}, the second one is tautological and the last one is the (inverse) Eichler-Shimura isomorphism.
  This composition maps $\xi_f$ to $f=f\oplus 0$ and \[ \xi_f^\pm\mapsto \frac12(f\oplus(\pm\overline{f^*})). \]
\end{lem}
\begin{proof}
  The first assertion is easy to see using the definition of $\xi_f$ and \cref{prop:es-explicit}. The second assertion follows from the first one using \cref{lem:es-gr-equiv}.
\end{proof}

\subsection{Hida families}

In this section we summarize the most important statements from Hida's theory, mainly to fix
notations.

We write ${\Gamma^\wt}\da1+p\Z_p$, ${\Gamma^\wt_r}\da1+p^r\Z_p\subseteq\Gamma^\wt$ and
${\Lambdawt}\da\psring\O{\Gamma^\wt}$. Further fix an integer $N$ prime to $p$ such that
$Np\ge4$ and regard $\Gamma^\wt$ as a subgroup as well as a quotient of
$\Z_{p,N}^\times\da\lim_r\zmod{Np^r}$. On finite levels, we regard
$\quotient{\Gamma^\wt}{\Gamma_r^\wt}$ as a subgroup as well as a quotient of
$\zmodmal{Np^r}=\quotient{\Gamma_0(Np^r)}{\Gamma_1(Np^r)}$. 

If $\ep$ is a character of $\Gamma^\wt$, we regard it also as a character of
$\Z_{p,N}^\times$ via the projection to $\Gamma^\wt$.
If $\ep$ is a character of $\Gamma^\wt$ of finite order, factoring over
$\quotient{\Gamma^\wt}{\Gamma_r^\wt}$, then we denote by
$\ModForms_k(\Gamma_1(Np^r),\quotient{\Gamma^\wt}{\Gamma_r^\wt},\ep,\O)$ the subspace where
$\quotient{\Gamma^\wt}{\Gamma_r^\wt}$ acts by $\ep$ (via diamond operators) and by
$\heckeT_k(Np^r,\quotient{\Gamma^\wt}{\Gamma_r^\wt},\ep,\O)$ the corresponding Hecke algebra
(similarly for cusp forms).

The {weight space} is defined as ${\calX^\wt}\da\Spec\Lambdawt$. If $\mathcal K$ is a finite extension of the fraction field of $\Lambdawt$ and ${\calI}$ is the integral closure of $\Lambdawt$ in $\mathcal K$, then we write ${\calX_\calI^\wt}\da\Spec\calI$.

Write ${\kappa_\wt}$ for the canonical embedding 
\[ \kappa_\wt\colon\Gamma^\wt\inj\O^\times. \]
For each $k\in\Z$ and each $\O^\times$-valued character $\ep$ of $\Gamma^\wt$ of finite order, we let \[ {\varphi_{k,\ep}}\colon\Lambdawt\ra\O \] be the $\O$-algebra morphism induced by
\[ \Gamma^\wt\ra\O^\times, \quad \gamma\mapsto \ep(\gamma)\kappa_\wt(\gamma)^k \] and we
write ${P_{k,\ep}}$ for its kernel, which is then an element of $\calX^\wt$.  The arithmetic
points in the weight space are
\[ {\calX^\arith} \da \{ P_{k,\ep} : k\ge2,\ \ep\colon\Gamma^\wt\ra\O^\times\text{ character
    of finite order} \} \subseteq\calX^\wt. \] If we have fixed $\calI$ as above, let
${\calX^\arith_\calI}$ be the preimage of $\calX^\arith$ under the natural map
$\calX_\calI^\wt\ra\calX^\wt$. We say that $P\in\calX^\wt_\calI$ is of {type $(k,\ep,r)$} if
$P\cap\Lambdawt=P_{k,\ep}$ with $k$ and $\ep$ as above and $\ker\ep=\Gamma^\wt_r$.  Finally
define the ideals \[ \omega_{k,r}\da\prod_\ep P_{k,\ep} \] for each fixed $k\in\Z$,
$r\ge0$, where $\ep$ runs through all $\O^\times$-valued characters of
$\quotient{\Gamma^\wt}{\Gamma^\wt_r}$.

We will often look at $\calX_\calI^\arithm(\O)$ or similar objects. The elements are by definition certain $\O$-algebra morphisms $\calI\ra\O$ and if we identify them with their kernels, we can view $\calX_\calI^\arithm(\O)$ as a subset of $\calX_\calI^\arithm$ as usual. Sometimes however it is important to distinguish the morphisms and the kernels. We will typically denote morphisms as $\varphi$ and prime ideals as $P$, so for example if we write $P\in\calX_\calI^\arithm(\O)$ we mean the kernel and not the morphism. If we want to make clear which morphism belongs to which prime ideal, we will use notations like $\varphi_P$ and $P_\varphi$.

For $k\ge2$ we let
\[ \mathcal M_k(Np^\infty,\O) \da \colim_{r}\ModForms_k(Np^r,\O), \] and let
$\overline{\mathcal{M}}_k(Np^\infty,\O)$ be the completion of $\mathcal M_k(Np^\infty,\O)$
with respect to the supremum norm on Fourier coefficients. Further we let
\[ \heckeT_k(Np^\infty,\O) \da \lim_{r}\heckeT_k(Np^r,\O) \] be Hida's big Hecke algebra of
level $Np^\infty$ for modular forms.  We make analogous definitions with cusp forms instead
of modular forms and denote the resulting objects by ${\calS_k(Np^\infty,\O)}$,
${\Sbar_k(Np^\infty,\O)}$ and ${\hecket_k(Np^\infty,\O)}$.  The $\O$-algebras
$\heckeT_k(Np^\infty,\O)$ and $\hecket_k(Np^\infty,\O)$ are canonically algebras over
$\psring\O{\Z^\times_{p,N}}$, in particular over $\Lambdawt$.  The ordinary parts are
denoted with superscript \enquote{${}^\ord$} and we omit the subscript \enquote{${}_k$} in
this case because the ordinary parts don't depend on the weight. The ordinary Hecke algebras
$\Tord(Np^\infty,\O)$ and $\tord(Np^\infty,\O)$ are free of finite rank over $\Lambdawt$.



\begin{thm}[Hida]\label{thm:control-thm}
  Let $k\ge2$ and $\ep\colon\Gamma^\wt\ra\O^\times$ be a character of finite order. Then
  there are canonical isomorphisms of $\O$-algebras
    \begin{align*}
    \tord(Np^\infty,\O)\tensor_{\Lambdawt}\left(\quotient{\Lambdawt}{P_{k,\ep}}\right) & \cong \tord_k(Np^r,\textstyle\quotient{\Gamma^\wt}{\Gamma_r^\wt},\ep,\O),\\
    \tord(Np^\infty,\O)\tensor_{\Lambdawt}\left(\quotient{\Lambdawt}{\omega_{k,r}}\right) & \cong \tord_k(Np^r,\O)
  \end{align*}
  and analogously for $\hecket$ instead of $\heckeT$.
\end{thm}
\begin{proof}
  \cite[Thm.\ 1.2]{MR848685}, \cite[Thm.\ 1.5.7 (iii)]{MR1674001}
\end{proof}

Now let $\calL$ be the quotient field of $\Lambdawt$, and fix an algebraic closure $\calLbar$ of it.

\begin{dfn}
  A {Hida family} is a morphism of $\Lambdawt$-algebras
  \[ F\colon\tord(Np^\infty,\O)\ra\calLbar. \] Its nebentype is defined to be the character
  $\psi\colon\zmodmal{Np}\ra\calI^\times$ obtained as the composition
  \[ \zmodmal{Np}\inj\Z_{p,N}^\times\inj\psring\O{\Z_{p,N}}^\times\ra
    \tord(Np^\infty,\O)^\times\ra[$F$]\calI^\times. \]
\end{dfn}

For a Hida family $F\colon\tord(Np^\infty,\O)\ra\calI$ and an arithmetic point
$P\in\calX_\calI^\arithm(\O)$ of type $(k,\ep,r)$ we denote by $F_P$ the member of $F$ at
$P$, which is the unique cusp form
$F_P\in\S_k(\Gamma_1(Np^r),\quotient{\Gamma^\wt}{\Gamma_r^\wt},\ep,\O)$ correponding to the
morphism $\tord_k(Np^r,\quotient{\Gamma^\wt}{\Gamma_r^\wt},\ep,\O)\ra\O$ which is the
reduction of $F$ modulo $P$ (using \cref{thm:control-thm}). If $\psi$ is the nebentype of
$F$, then $F_P$ has nebentype $\ep\psi\omega^{-k}$, where $\omega$ is the Teichmüller
character. When we view the elements of $\calX^\arith_\calI(\O)$ as morphisms instead of
ideals, we shall also write $F_\varphi$ instead of $F_P$ for $\varphi=\varphi_P$.

Since $\tord(Np^\infty,\O)$ is free of finite rank, the image of a Hida family generates a
finite field extension of $\calL$, say $\mathcal K$, and the image even lies in the integral
closure of $\Lambdawt$ inside $\mathcal K$, which we call $\calI$. Moreover note that since
the kernel of $F$ contains a minimal prime ideal, there are only finitely many $\calI$ that
can occur in this way as long as $N$ and $\O$ are fixed. By \cite[Lem.\ 3.1]{MR0976685} each
such $\calI$ is free of finite rank over $\Lambdawt$.

\begin{dfn}
  If $F$ is a Hida family, then we call the ring $\calI$ from above the {coefficient ring} of $F$.
  We call the finitely many $\calI$ that can occur the {coefficient rings} of $\tord(Np^\infty,\O)$.
\end{dfn}

The above definition of coefficient rings and Hida families is not totally standard in the
literature, but for us it will be more convenient to work with this definition, see
\cref{rem:i-integrally-closed} below.

It is well-known that by possibly enlarging $L$ (and thus $\O$), one can assume that
$\calX_\calI^\arith(\O)$ is Zariski dense in $\calX^\wt_\calI(\Qpbar)$ (which we both view
as subsets of $\calX^\wt_\calI=\Spec\calI$). We will assume this from now on. For later
reference, let us summarize the notations and assumptions we are now using.
\begin{situation}\label{setting:hida-families}
  We have fixed a number field $K$ with embedding $K\subseteq\Qbar$ and a place $\frakp\mid
  p$ with completion $L$ and ring of integers $\O$, further an integer $N$
  prime to $p$ such that $Np\ge4$ and a coefficient ring $\calI$ of
  $\tord(Np^\infty,\O)$. We assume that $L$ and $\O$ are large enough such that
  $\calX_\calI^\arith(\O)$ is Zariski dense in $\calX^\wt_\calI(\Qpbar)$. Further we assume
  that $L$ is the maximal subfield inside $\mathcal K$ which is algebraic over $\Qp$.
\end{situation}

\begin{dfn}
  Let $F\colon\tord(Np^\infty,\O)\ra\calLbar$ be a Hida family. Then we call $F$ {new} if there does not exist a proper divisor $M\mid N$ and a Hida family $G\colon\tord(Mp^\infty,\O)\ra\calLbar$ such that $F(T_\ell)=G(T_\ell)$ for almost all primes $\ell$.
\end{dfn}

Note the similarity of this definition to that of a newform in
\cref{sec:modular-forms}.

\begin{thm}[Hida]\label{thm:hida-fam-new}
  Let $F\colon\tord(Np^\infty,\O)\ra\calI$ be a Hida family of nebentype $\psi$. Then the following are equivalent:
  \begin{tfaelist}
  \item\label{thm:new-hida-fam:new} $F$ is new.
  \item\label{thm:new-hida-fam:some} For some $P\in\calX_\calI^\arith(\O)$, $F_P$ is new.
  \item\label{thm:new-hida-fam:almost-all} For infinitely many $P\in\calX_\calI^\arith(\O)$, $F_P$ is new.
  \item\label{thm:new-hida-fam:r} $F_P$ is new for all $P\in\calX_\calI^\arith(\O)$ of type $(k,\ep,r)$ with $r>1$.
  \item\label{thm:new-hida-fam:char} $F_P$ is new for all $P\in\calX_\calI^\arith(\O)$ of type $(k,\ep,r)$ such that the $p$-part of $\ep\psi\omega^{-k}$ is nontrivial.
  \end{tfaelist}

  Now assume that the above equivalent statements hold and $P\in\calX_\calI^\arith(\O)$ of type $(k,\ep,r)$ is such that the $p$-part of $\ep\psi\omega^{-k}$ is trivial (in particular $r=1$). We can then view $\ep\psi\omega^{-k}$ as a character of $\zmodmal N$. In this situation $F_P\in\S_k(\Gamma_1(Np),\ep\psi\omega^{-k})$ can either be new (in which case $k=2$), or $F_P$ is the unique ordinary refinement of an ordinary newform $F_P^0\in\S_k(\Gamma_1(N),\ep\psi\omega^{-k})$.
\end{thm}
\begin{proof}
  Clearly we have implications $\text{\ref{thm:new-hida-fam:char}} \Rightarrow
  \text{\ref{thm:new-hida-fam:r}} \Rightarrow \text{\ref{thm:new-hida-fam:almost-all}}
  \Rightarrow \text{\ref{thm:new-hida-fam:some}}$, so it remains to see
  $\text{\ref{thm:new-hida-fam:some}}\Rightarrow\text{\ref{thm:new-hida-fam:new}}$ and
  $\text{\ref{thm:new-hida-fam:new}}\Rightarrow\text{\ref{thm:new-hida-fam:char}}$. For
  these implications see \cite[Thm.\ 2.4]{MR934243} and for the final
  statement see \cite[Thm.\ 4.1]{MR976685}.
\end{proof}

\begin{dfn}\label{dfn:notation-f-p-new}
  Let $\calI$ be a coefficient ring of  $\tord(Np^\infty,\O)$ and fix a Hida family $F$ which is new. Then by \cref{thm:hida-fam-new}, for almost all $P\in\calX^\arith_\calI(\O)$ the form $F_P$ is new, and for the $P$ such that $F_P$ is not new, there exists a newform $F_P^0$ such that $F_P$ is a refinement of $F_P^0$. Let us write ${F_P^\new}$ to mean either $F_P$ if $F_P$ itself is new, or $F_P^0$ if $F_P$ is not new.
\end{dfn}

If $F$ is a new Hida family, then from \cref{thm:hida-fam-new} it is clear that by possibly
enlarging $\O$ we can assume that the points $P\in\calX^\arithm_\calI(\O)$ such that $F_P$
is a newform are Zariski dense in $\calX^\wt_\calI$.

\begin{dfn}
  Let $\calI$ be a coefficient ring of $\tord(Np^\infty,\O)$.
  We define the module of {$\calI$-adic cusp forms of level $Np^\infty$} as
  \[ {\SS^\ord(Np^\infty,\calI)}\da\Hom_{\Lambdawt}(\tord(Np^\infty,\O),\calI) \]
  (here we mean morphisms of $\Lambdawt$-\emph{modules}).
  Let the Hecke algebra $\tord(Np^\infty,\O)$ act on this module by duality, i.\,e.\ $(TF)(X)=F(TX)$ for $F\in\SS^\ord(Np^\infty,\calI)$, $T,X\in\tord(Np^\infty,\O)$.
\end{dfn}


By construction there is a perfect $\calI$-bilinear pairing
\begin{equation}
  \label{eqn:ss-ord-perf-pairing}
  \SS^\ord(Np^\infty,\calI)\times(\tord(Np^\infty,\O)\tensor_{\Lambdawt}\calI)\ra \calI.
\end{equation}
For a fixed $F\in\SS^\ord(Np^\infty,\calI)$, we define $F$-eigenspace
\[ {\SS^\ord(Np^\infty,\calI)[F]} \da \{ G\in\SS^\ord(Np^\infty,\calI) : \forall T\in\tord(Np^\infty,\O)\colon TG = F(T)G \}. \]
It is then clear that 
$\SS^\ord(Np^\infty,\calI)[F]$ is free of rank $1$ over $\calI$.

\begin{thm}[Hida]\label{thm:grosse-galdarst-hida-fam}
  \begin{enumerate}
  \item Fix a Hida family $F\in\SS^\ord(Np^\infty,\calI)$ which is new. Then there is a
    unique (up to isomorphism) free $\calI$-module $\calT$ of rank $2$ and a continuous odd
    irreducible Galois representation unramified outside $Np\infty$
    \[ {\rho_F}\colon\GQ\ra\Aut_\calI(\calT) \]
    such that for each $P\in\calX^\arith_\calI(\O)$, the reduction of $\rho_F$ modulo $P$ is equivalent to the Galois representation attached to $F_P^\new$.
  \item There is a free rank $1$ $\calI$-direct summand $\calT^0$ of $\calT$ which is an unramified $\GQp$-subrepresentation.
  \end{enumerate}
\end{thm}
\begin{proof}
  The first statement 
  is a variation of
  \cite[Thm.\ 2.1]{MR848685}. It can easily be obtained from the
  form stated there using a standard compactness and continuity argument to obtain a
  Galois-stable lattice and then taking the reflexive closure (recall that our ring $\calI$
  is integrally closed by definition!). The second statement follows from \cite[Thm.\
  4]{MR1039770}. More precisely: The representation from the
  theorem there is by uniqueness the same as ours, and the theorem states that it is
  ordinary in the sense of \cite[Def.\ 1]{MR1039770}. It is
  easy to see that this definition of ordinariness implies that we have $\calT^0$ as
  claimed. The restriction to $p\ge7$ there can be removed by \cite[p.\
  991]{MR1854117}.
\end{proof}

\begin{rem}\label{rem:i-integrally-closed}
  Recall that our ring $\calI$ is integrally closed by definition. This convention is not
  standard in the literature. Often Hida families are defined as irreducible components of
  $\Spec\tord(Np^\infty,\O)$ (which need not be normal), whose underlying rings are then
  used as coefficient rings. In this case, the image of $\rho_F$ does in general not lie in
  $\GL_2(\calI)$ but only in $\GL_2(\Quot(\calI))$ (after choosing a basis), and one needs
  extra assumptions on $\calI$ to have it in $\GL_2(\calI)$, such as $\calI$ being a unique
  factorisation domain or the residual representation $\overline\rho_F$ being absolutely
  irreducible. See \cite[§9]{MR3334891} for a discussion of these issues. We
  chose to take $\calI$ always as integrally closed, which is maybe not so directly related
  to the geometry of $\tord(Np^\infty,\O)$ as in our definition a Hida family will in
  general not surject onto its coefficient ring, but allows us to work with representations
  into $\GL_2(\calI)$, which seems more suitable for our purpose.
\end{rem}

\subsection{Modular symbols for Hida families}

We introduce $\calI$-adic modular symbols following Kitagawa, which are the modular symbols pendant of $\calI$-adic cusp forms. We proceed in several steps, which will be motivated afterwards.

\begin{dfn}
  For $k\ge2$ we put \[ {\calMS_k(Np^\infty,\O)}\da\colim_r\MS_k(Np^r,\O), \]
  where the maps are induced from the canonical maps $Y_1(Np^s)\ra Y_1(Np^r)$ on modular curves for $s\ge r\ge 0$. We write ${\overline\calMS_k(Np^\infty,\O)}$ for the $p$-adic completion of $\calMS_k(Np^\infty,\O)$.
  We can define the same with $\O$ replaced by $\quotient\O{p^t}$ for some $t\ge0$.
\end{dfn}

From the Hecke action on each of the modules $\MS_k(Np^r,\O)$ we get a
$\heckeT_k(Np^\infty,\O)$-module structure on these modules. So in particular, we get a
$\psring\O{\Z_{p,N}^\times}$-module structure and a $\Lambdawt$-module structure. Moreover,
it is also clear that the transition maps used to form the limit are compatible with the
action of $\EP\in\GL_2(\Z)$ since $\EP$ describes the action of complex conjugation on the
modular curves. Hence $\EP$ acts on $\overline\calMS_k(Np^\infty,\O)$ in a well-defined way.

\begin{dfn}
  The module of {universal $p$-adic modular symbols} is defined as
  \[ {\UM(Np^\infty,\O)} = \Hom_\O(\overline\calMS_2(Np^\infty,\O),\O). \]
  Here we mean the $\O$-Banach dual, i.\,e.\ continuous homomorphisms.
\end{dfn}
There is then a perfect pairing
\begin{equation}
  \label{eqn:perfect-pairing-um}
  \overline\calMS_2(Np^\infty,\O)\times\UM(Np^\infty,\O)\ra\O.
\end{equation}

The Hecke action is the dual Hecke action coming from the action on $\overline\calMS(Np^\infty,\O)$, that is, $(T\alpha)(\xi)=\alpha(T\xi)$ for $\alpha\in\UM(Np^\infty,\O)$, $\xi\in\overline\calMS(Np^\infty,\O)$ and $T\in\Tord(Np^\infty,\O)$. In particular, this makes $\UM(Np^\infty,\O)$ a $\Lambdawt$-module.
From this definition of the Hecke action, it is easy to see that
\begin{equation}
  \label{eqn:um-ord}
  \UM^\ord(Np^\infty,\O) = \Hom_\O(\overline\calMS^\ord_2(Np^\infty,\O),\O).
\end{equation}


\begin{dfn}\label{dfn:i-adic-ms}
  The $\calI$-module of {$\calI$-adic ordinary modular symbols} is defined as
  \[ {\MMSS^\ord(Np^\infty,\calI)} \da \Hom_{\Lambdawt}(\UM^\ord(Np^\infty,\O),\calI). \]
  The Hecke action is again the dual action of the action on $\UM^\ord(Np^\infty,\O)$.
\end{dfn}

From \cite[Prop.\ 5.7]{MR1279604} it is clear that $\MMSS^\ord(Np^\infty,\calI) = \MMSS^\ord(Np^\infty,\Lambdawt)\tensor_{\Lambdawt}\calI$.

To motivate these definitions, recall that philosophically modular symbols and modular forms
are two incarnations of the same phenomenon, as suggested by the Eichler-Shimura
isomorphism. The definition of $\overline\calMS_k^\ord(Np^\infty,\O)$ parallels in some way
the definition of $\Sbar_k(Np^\infty,\O)$.  There is a perfect pairing
\[ \paarung{\cdot}{\cdot} \colon \hecket_2(Np^\infty,\O) \times \Sbar_2(Np^\infty,\O) \ra \O, \]
so $\calI$-adic cusp forms are
$\Hom_{\Lambdawt}(\Hom_\O(\Sbar_2^\ord(Np^\infty,\O),\O),\calI)$, while by the perfect
pairing \eqref{eqn:perfect-pairing-um} $\calI$-adic modular symbols are
$\Hom_{\Lambdawt}(\Hom_\O(\overline\calMS^\ord(Np^\infty,\O),\O),\calI)$, so these
definitions have some analogy.

For an $\calI$-algebra morphism $F\colon\tord(Np^\infty,\O)\tensor_{\Lambdawt}\calI\ra\calI$ (that is, an $\calI$-adic eigenform), we denote the induced morphism $\Tord(Np^\infty,\O)\tensor_{\Lambdawt}\calI\ra\calI$ still by $F$, by abuse of notation. Then $\MMSS^\ord(Np^\infty,\calI)^\pm[F]$ is well-defined.
An important condition on this module we will need to impose later is the following.
\begin{cond}\label{cond:mmss-free-rank-one}
  $\MMSS^\ord(Np^\infty,\calI)^\pm[F]$ is free of rank $1$ over $\calI$.
\end{cond}

\begin{rem}\label{rem:conditions-imply-mmss-free-rank-one}
  There are several conditions which are known to imply \cref{cond:mmss-free-rank-one},
  among them that of $\calI$ being factorial (which is satisfied for example for $\calI=\Lambdawt$). We do not list the other conditions, see \cite[Lem.\ 5.11]{MR1279604} for this.
\end{rem}

One of the main technical properties of $\calI$-adic modular symbols is the following
control theorem proved by Kitagawa. We will explain some steps in its proof in
\cref{sec:control-theory-ms} because they will be important later.
 In the following, if $M$ is some $\O$-module with an action of $\Gamma^\wt$
and $\ep$ is an $\O^\times$-valued character of $\Gamma^\wt$, we write $M[\ep]$ for the
submodule where the action of $\Gamma^\wt$ is given by
$\ep$.

\begin{thm}[Kitagawa]\label{thm:control-thm-ms}
  \begin{enumerate}
  \item\label{thm:control-thm-ms:ohne-f} There is a canonical isomorphism of $\Tord(Np^\infty,\O)\tensor_{\Lambdawt}\calI$-modules
    \[ \MMSS^\ord(Np^\infty,\calI)\tensor_{\calI}\left(\quotient\calI P\right) \isom \MS_k^\ord(Np^r,\O)[\ep] \]
    which is compatible with the action of $\EP$.
  \item\label{thm:control-thm-ms:omega} There is a canonical isomorphism of $\O$-modules
    \[ \MMSS^\ord(Np^\infty,\Lambdawt)\tensor_{\Lambdawt}\left(\quotient{\Lambdawt}{\omega_{k,r}}\right) \isom \MS_k^\ord(Np^r,\O). \]
  \item\label{thm:control-thm-ms:f} Fix an $\calI$-adic eigenform $F\in\SS^\ord(Np^\infty,\calI)$ and write $F_P$ for the member at $P$ of the Hida family associated to $F$. Assume further that \cref{cond:mmss-free-rank-one} is satisfied. Then there is a canonical isomorphism of $\O$-modules
    \[ 
    \MMSS^\ord(Np^\infty,\calI)^\pm[F]\tensor_{\calI}\left(\quotient\calI P\right) \isom \MS_k^\ord(Np^r,\O)^\pm[F_P]. \]
  \item\label{thm:control-thm-ms:teilc} Let $\Xi\in\MMSS^\ord(Np^\infty,\calI) = \Hom_{\Lambdawt}(\UM^\ord(Np^\infty,\O),\calI)$, $u\in\UM^\ord(Np^\infty,\O)=\Hom_\O(\overline\calMS^\ord(Np^\infty,\O),\O)$ and let $\Xi_\varphi$ be the image of $\Xi$ in the right hand side in statement \ref{thm:control-thm-ms:ohne-f}. Then \[ \varphi(\Xi(u)) = u(\Xi_\varphi). \]
  \end{enumerate}
\end{thm}

The module of $\calI$-adic modular symbols admits an alternative description using
trace-compatible projective systems which we will need in the next section. In the context
of $\calI$-adic modular forms an analogous result was proved by Ohta in
\cite[§2.3--4]{MR1332907}, and the techniques carry over essentially unchanged to modular
symbols. We first cite Ohta's results and then formulate the modular symbols analogue. It
involves the Atkin-Lehner operator and the adjoint Hecke operators introduced in
\cref{sec:poincare}. Also it involves anti-ordinary parts, which are obtained by using
$T_p^\iota$ instead of $T_p$ in the definition of the ordinary projection and denoted by
$\heckeT^{\antiord}(Np^\infty,\O)$ etc.

Put \[ {\S_k^\iota(Np^r,\O)} \da \{ \xi\in\S_k(\Gamma_1(Np^r),\C_p) : \xi[w_{Np^r}]\in\S_k(\Gamma_1(Np^r),\O) \} \]
and \[ {\mathfrak{S}_k^\iota(Np^\infty,\O)} \da \lim_r\S_k^\iota(Np^r,\O). \]
Here the limit is taken along the maps induced by the change-of-level morphism
$\Sigma_{Np^{r+1},Np^r}$ from \cref{sec:refinements}.
We denote by ${\mathfrak S_k^{\antiord}(Np^\infty,\O)}$ the anti-ordinary part of
${\mathfrak S_k^{\iota}(Np^\infty,\O)}$.


\begin{thm}[Ohta]\label{thm:hida-fam-als-proj-systeme}
  For any $k\ge 2$ there is a canonical isomorphism of $\Lambdawt$-modules
  \[\begin{tikzpicture}
      \matrix (m) [matrix of math nodes, row sep=7ex, column sep=2em, text height=1.5ex, text depth=0.25ex]
      { \SS^\ord(Np^\infty,\Lambdawt) & \mathfrak{S}^{\antiord}_k(Np^\infty,\O) \\
        F & \displaystyle(f_r)_r\text{ with }f_r=\frac1{p^{r-1}}\Bigg(\sum_{\ep\in\hat\Gamma^\wt_{\mathrm f,r}}F_{k,\ep}[T_p^{-r}]\Bigg)[w_{Np^r}]^{-1}\\
        \pbox{0.5\textwidth}{
            the unique $F$ such that\\
            $F_{k,\ep}=\displaystyle\!\!\!\sum_{\alpha\in\quotient{\Gamma^\wt}{\Gamma^\wt_r}}\!\!\!\ep(\alpha)f_r[w_{Np^r}][T_p^r]\diamondop\alpha^{-1}$
        }
          & (f_r)_r. \\};
      \path[<->,font=\scriptsize]
      (m-1-1) edge node [above] {$\sim$} (m-1-2);
      \path[|->,font=\scriptsize]
      (m-2-1) edge (m-2-2)
      (m-3-2) edge (m-3-1);
    \end{tikzpicture}\]
  Under this isomorphism, a Hecke operator from $\hecket^\ord(Np^\infty,\O)$ on the left side corresponds to its adjoint in $\hecket^\antiord(Np^\infty,\O)$ on the right side.
\end{thm}
\begin{proof}
  \cite[Thm.\ 2.3.6]{MR1332907}
\end{proof}

One can check that the Atkin-Lehner involution on $Y_1(Np^{r+1})^\arithm$ interchanges the
change of level morphisms $\Sigma_{Np^{r+1},Np^r}$ and $\Theta_{Np^{r+1},Np^r}$ from
\cref{sec:refinements}. Therefore we obtain the following corollary.

\begin{cor}\label{cor:ss-as-proj-limit}
  There is a canonical isomorphism of $\hecket^\ord(Np^\infty,\O)$-modules
  \begin{align*}
    \SS^\ord(Np^\infty,\Lambdawt) &\isom \lim_r \S^\ord_k(X_1(Np^r),\O),\\
    F &\mapsto \displaystyle(f_r)_r\text{ with }f_r=\frac1{p^{r-1}}\Bigg(\sum_{\ep\in\hat\Gamma^\wt_{\mathrm f,r}}F_{k,\ep}[T_p^{-r}]\Bigg)[w_{Np^r}]^{-1}
  \end{align*}
  the limit now being taking along the maps induced by the change-of-level morphism
  $\Theta_{Np^{r+1},Np^r}$ from \cref{sec:refinements}.
\end{cor}

Now we turn to modular symbols and prove an analogue of \cref{thm:hida-fam-als-proj-systeme}. Fix $k\ge2$ and let 
\[ {\MS_k^\iota(Np^r,\O)} \da \{ \xi\in\MS_k(Np^r,\C_p) : \xi[w_{Np^r}]\in\MS_k(Np^r,\O) \} \]
for $r\ge0$ and \[ {\mathfrak{MS}_k(Np^\infty,\O)}\da \lim_r\MS_k^\iota(Np^r,\O) \]
again taking the limit along $\Sigma_{Np^{r+1},Np^r}$. Then the adjoint Hecke eigenalgebra of $\MS_k^\iota(Np^r,\O)$ is $\HeckeKonkretadjoint k {Np^r} \O$. So we can consider $\mathfrak{MS}_k(Np^\infty,\O)$ as a module over $\heckeT^\iota(Np^\infty,\O)$ and via this also as a $\Lambdawt$-module, and we have again an anti-ordinary part $\mathfrak{MS}_k^{\antiord}(Np^\infty,\O)$.

As a preparation to the proof the analogue of Ohta's result we need some lemmas.

\begin{lem}\label{lem:eindeutiger-lift-in-lambda-lemma-ohta}
  Fix $u\in\UM^\ord(Np^\infty,\O)$ and $(x_r)_r\in \mathfrak{MS}_2^{\antiord}(Np^\infty,\O)$. Then there is a unique $X(u)\in\Lambdawt$ such that
  \[ X(u) \mod P_{2,\ep} =\!\!\! \sum_{\alpha\in\quotient{\Gamma^\wt}{\Gamma^\wt_r}}\!\! \ep(\alpha)u(x_r[w_{Np^r}][T_p^r]\diamondop\alpha^{-1}) \]
  for all $\ep\in\hat\Gamma^\wt_{\mathrm f}$.
\end{lem}
\begin{proof}
  For $\alpha\in\quotient{\Gamma^\wt}{\Gamma^\wt_r}$, abbreviate $u_\alpha\da u(x_r[w_{Np^r}][T_p^r]\diamondop\alpha^{-1})$.
  Define a map
  \begin{equation*}
    F\colon\hat\Gamma^\wt_{\mathrm f}\ra\O,\quad \ep\mapsto\!\!\!\sum_{\alpha\in\quotient{\Gamma^\wt}{\Gamma^\wt_r}}\!\!\!\ep(\alpha) u_\alpha.
  \end{equation*}
  Fix $\alpha_0\in\Gamma^\wt$. Then we calculate
  \begin{align*}
    \sum_{\ep\in\hat\Gamma^\wt_{\mathrm f,r}}\ep(\alpha_0)^{-1} F(\ep) &= \sum_{\ep\in\hat\Gamma^\wt_{\mathrm f,r}}\ep(\alpha_0)^{-1}\!\!\!\sum_{\alpha\in\quotient{\Gamma^\wt}{\Gamma^\wt_r}}\!\!\! \ep(\alpha) u_\alpha \\
        &= \!\!\!\sum_{\alpha\in\quotient{\Gamma^\wt}{\Gamma^\wt_r}}\sum_{\ep\in\hat\Gamma^\wt_{\mathrm f,r}}\ep(\alpha_0^{-1}\alpha)u_\alpha \\
        &= p^{r-1}u_{\alpha_0} + \sum_{\alpha\neq\alpha_0} u_\alpha \sum_\ep\ep(\alpha_0^{-1}\alpha).
  \end{align*}
  Since $\sum_\ep\ep(\alpha_0^{-1}\alpha)=0$ if $\alpha\neq\alpha_0$, it follows that
  \begin{equation*}
    \sum_{\ep\in\hat\Gamma^\wt_{\mathrm f,r}}\ep(\alpha_0)^{-1} F(\ep)\in p^{r-1}\O.
  \end{equation*}
  Then the claim follows from \cite[Lem.\ 2.4.2]{MR1332907}.
\end{proof}

In the following, for $X\in\MMSS^\ord(Np^\infty,\Lambdawt)$ we denote its image in $\MS_2(Np^r,\O)[\ep]$ under the morphism from \cref{thm:control-thm-ms} by $X_{2,\ep}$ (with $k=2$ there).
\begin{lem}\label{lem:x-u-lambda-linear}
  Fix $(x_r)_r\in \mathfrak{MS}_2^{\antiord}(Np^\infty,\O)$.
  \begin{enumerate}
  \item The map \[ X\colon\UM^\ord(Np^\infty,\O)\ra\Lambdawt,\quad u\mapsto X(u) \]
    with $X(u)$ as in \cref{lem:eindeutiger-lift-in-lambda-lemma-ohta}
    is $\Lambdawt$-linear.
  \item $X$ is the unique element in $\MMSS^\ord(Np^\infty,\Lambdawt)$ such that if
    $X_{2,\ep}$ is its image mod $P_{2,\ep}$ in $\MS_2(Np^r,\O)[\ep]$ under the isomorphism
    from \cref{thm:control-thm-ms}, then
    \[ X_{2,\ep}=\sum_{\alpha\in\quotient{\Gamma^\wt}{\Gamma^\wt_r}}
      \ep(\alpha)x_r[w_{Np^r}][T_p^r]\diamondop\alpha^{-1}. \]
  \end{enumerate}
\end{lem}
\begin{proof}
  \begin{enumerate}
  \item  First, let $u,v\in\UM^\ord(Np^\infty,\O)$. Then $X(u)+X(v)$ has the property that
    \begin{align*}
      X(u)+X(v)\mod P_{2,\ep} = \sum_{\alpha\in\quotient{\Gamma^\wt}{\Gamma^\wt_r}} &\ep(\alpha)u(x_r[w_{Np^r}][T_p^r]\diamondop\alpha^{-1})\\ & + \sum_{\alpha\in\quotient{\Gamma^\wt}{\Gamma^\wt_r}} \ep(\alpha)v(x_r[w_{Np^r}][T_p^r]\diamondop\alpha^{-1})\\
      = \sum_{\alpha\in\quotient{\Gamma^\wt}{\Gamma^\wt_r}} &\ep(\alpha)(u+v)(x_r[w_{Np^r}][T_p^r]\diamondop\alpha^{-1})
    \end{align*}
    for all $\ep\in\hat\Gamma^\wt_{\mathrm f}$.
    Since $X(u+v)$ was defined to be the unique element in $\Lambdawt$ with this property, it follows $X(u+v)=X(u)+X(v)$.
    
    Now let $u\in\UM^\ord(Np^\infty,\O)$ and $\lambda\in\Lambdawt$. Without loss of generality, assume $\lambda\in\Gamma^\wt$. Then
    \begin{align*}
      \lambda X(u)\mod P_{2,\ep} &= \ep\kappa_\wt^2(\lambda)\sum_{\alpha\in\quotient{\Gamma^\wt}{\Gamma^\wt_r}} \ep(\alpha)u(x_r[w_{Np^r}][T_p^r]\diamondop\alpha^{-1})\\
      &= \kappa_\wt^2(\lambda)\sum_{\alpha\in\quotient{\Gamma^\wt}{\Gamma^\wt_r}} \ep(\lambda\alpha)u(x_r[w_{Np^r}][T_p^r]\diamondop\alpha^{-1}),
    \end{align*}
    while
    \begin{align*}
      X(\lambda u)\mod P_{2,\ep} &= \sum_{\alpha\in\quotient{\Gamma^\wt}{\Gamma^\wt_r}} \ep(\alpha)u(x_r[w_{Np^r}][T_p^r]\diamondop\alpha^{-1}\diamondop\lambda\kappa^2_\wt(\lambda))\\
      &= \kappa_\wt^2(\lambda)\sum_{\alpha\in\quotient{\Gamma^\wt}{\Gamma^\wt_r}} \ep(\alpha)u(x_r[w_{Np^r}][T_p^r]\diamondop{\alpha\lambda^{-1}}^{-1}).
    \end{align*}
    Replacing $\alpha$ in the second calculation by $\lambda\alpha$, which then also travels through all elements in $\quotient{\Gamma^\wt}{\Gamma^\wt_r}$, shows that the two expressions are equal.
  \item   Since by \eqref{eqn:um-ord} and the perfectness of the pairing \eqref{eqn:perfect-pairing-um} the element $X_{2,\ep}$ is determined by its images under $u$ for all $u\in\UM^\ord(Np^\infty,\O)$, it suffices to prove that 
    \begin{equation*}
      u(X_{2,\ep}) = u\Big(\sum_{\alpha\in\quotient{\Gamma^\wt}{\Gamma^\wt_r}} \ep(\alpha)x_r[w_{Np^r}][T_p^r]\diamondop\alpha^{-1}\Big)
    \end{equation*}
    for all $u\in\UM^\ord(Np^\infty,\O)$. But by
    \cref{thm:control-thm-ms}~\ref{thm:control-thm-ms:teilc}
    $u(X_{2,\ep}) = X(u) \mod P_{2,\ep}$, so the claim follows from
    \cref{lem:eindeutiger-lift-in-lambda-lemma-ohta}.
  \end{enumerate}
\end{proof}

Now we can prove the modular symbols version of Ohta's \cref{thm:hida-fam-als-proj-systeme}.

\begin{thm}\label{thm:mmss-trace-comp-systems-iso}
  There is a canonical isomorphism of $\Lambdawt$-modules
  \[\begin{tikzpicture}
      \matrix (m) [matrix of math nodes, row sep=4ex, column sep=2em, text height=1.5ex, text depth=0.25ex]
      { \MMSS^\ord(Np^\infty,\Lambdawt) & \mathfrak{MS}_2^{\antiord}(Np^\infty,\O) \\
        X & \displaystyle(x_r)_r\text{ with }x_r=\frac1{p^{r-1}}\Bigg(\sum_{\ep\in\hat\Gamma^\wt_{\mathrm f,r}}X_{2,\ep}[T_p^{-r}]\Bigg)[w_{Np^r}]^{-1}\\
        X\text{ as in \cref{lem:x-u-lambda-linear}} & (x_r)_r. \\};
      \path[<->,font=\scriptsize]
      (m-1-1) edge node [above] {$\sim$} (m-1-2);
      \path[|->,font=\scriptsize]
      (m-2-1) edge (m-2-2)
      (m-3-2) edge (m-3-1);
    \end{tikzpicture}\]
  Under this isomorphism, a Hecke operator from $\heckeT^\ord(Np^\infty,\O)$ on the left side corresponds to its adjoint in $\heckeT^\antiord(Np^\infty,\O)$ on the right side.
\end{thm}
\begin{proof}
  First, that $(x_r)_r$ as in the statement forms a compatible system for the trace maps can be shown exactly in the same manner as in the proof given in \cite[§2.4]{MR1332907}. Hence we know that both maps are well-defined. They are obviously $\O$-linear, and that they are in fact $\Lambdawt$-linear follows by definition of the $\Lambdawt$-module structure from the final statement about Hecke operators.

  We sketch the calculations that show that the two maps are inverse to each other. First, for $X\in\MMSS^\ord(Np^\infty,\Lambdawt)$, let $Y$ be the image of $X$ under the composition of the two maps. Then for $\ep_0\in\hat\Gamma^\wt_{\mathrm f,r}$
  \begin{align*}
    Y_{2,\ep_0} &= \sum_{\alpha\in\quotient{\Gamma^\wt}{\Gamma^\wt_r}} \ep_0(\alpha)\frac1{p^{r-1}} \Big(\sum_{\ep\in\hat\Gamma^\wt_{\mathrm f,r}} X_{2,\ep}[T_p^{-r}]\Big)[w_{Np^r}^{-1}][w_{Np^r}][T_p^r]\diamondop\alpha^{-1} \\
    &= \frac1{p^{r-1}}\sum_\alpha\ep_0(\alpha)\sum_\ep X_{2,\ep}\diamondop\alpha^{-1} = \frac1{p^{r-1}}\sum_\alpha\ep_0(\alpha)\sum_\ep \ep(\alpha^{-1})X_{2,\ep} \\
    &= \frac1{p^{r-1}}\sum_\alpha\ep_0(\alpha)\ep_0(\alpha^{-1}) X_{2,\ep_0} + \frac1{p^{r-1}}\sum_\alpha\ep_0(\alpha)\sum_{\ep\neq\ep_0}\ep(\alpha^{-1})X_{2,\ep} \\
    &= X_{2,\ep_0} + \frac1{p^{r-1}}\sum_{\ep\neq\ep_0}\Big(\sum_\alpha\ep_0\ep^{-1}(\alpha)\Big) X_{2,\ep}.
  \end{align*}
  Since $\sum_\alpha\ep_0\ep^{-1}(\alpha)=0$ for $\ep\neq\ep_0$, it follows $Y_{2,\ep_0}=X_{2,\ep_0}$. By the Zariski density of arithmetic points in $\Spec\Lambdawt$ this shows $X=Y$.

  On the other hand, for $(x_r)_r\in \mathfrak{MS}_2^{\antiord}(Np^\infty,\O)$, let $(y_r)_r$ be the image of $(x_r)_r$ under the composition of the two maps. Then
  \begin{align*}
    y_r &=\frac1{p^{r-1}}\sum_{\ep\in\hat\Gamma^\wt_{\mathrm f,r}}\sum_{\alpha\in\quotient{\Gamma^\wt}{\Gamma^\wt_r}}\ep(\alpha)x_r[w_{Np^r}][T_p^r]\diamondop\alpha^{-1}[T_p^{-r}][w_{Np^r}^{-1}]\\
    &= \frac1{p^{r-1}}\sum_{\ep\in\hat\Gamma^\wt_{\mathrm f,r}}\sum_{\alpha\in\quotient{\Gamma^\wt}{\Gamma^\wt_r}}\ep(\alpha) x_r\diamondop\alpha\\
    &= \frac1{p^{r-1}}\sum_{\alpha\in\quotient{\Gamma^\wt}{\Gamma^\wt_r}}\bigg(\sum_{\ep\in\hat\Gamma^\wt_{\mathrm f,r}}\ep(\alpha)\bigg)x_r\diamondop\alpha.
  \end{align*}
  Since $\sum_\ep\ep(\alpha)$ is $p^{r-1}$ if $\alpha=1$ and $0$ otherwise, it follows $y_r=x_r$.

  The last claim about the Hecke operators follows with an easy calculation using the
  well-known relations between Atkin-Lehner and Hecke operators.
\end{proof}

Similarly as in the case of modular forms we obtain the following corollary.

\begin{cor}\label{cor:mmss-as-proj-limit}
  For each $k$ there is a canonical isomorphism of $\heckeT^\ord(Np^\infty,\O)$-modules
  \begin{align*}
    \MMSS^\ord(Np^\infty,\Lambdawt) &\isom \lim_r \MS^\ord_k(Np^r,\O),\\
    X&\mapsto \displaystyle(x_r)_r\text{ with }x_r=\frac1{p^{r-1}}\Bigg(\sum_{\ep\in\hat\Gamma^\wt_{\mathrm f,r}}X_{2,\ep}[T_p^{-r}]\Bigg),
  \end{align*}
  where the limit is taken along the maps induced by the change-of-level morphism
  $\Theta_{Np^{r+1},Np^r}$ from \cref{sec:refinements}.
\end{cor}

\subsection{On control theory for $\calI$-adic modular symbols}
\label{sec:control-theory-ms}

We keep the notations from the previous sections and fix a
$\varphi\in\calX^\arith_\calI(\O)$ of type $(k,\ep,r)$ and write $P=P_\varphi$ for its
kernel. Here we discuss some steps in the proof of the control \cref{thm:control-thm-ms}
because we will need them in the next section.

We first cite an important lemma which was proved by Kitagawa. Here we fix $t\in\N$ for the
moment and write $R=\quotient\O{p^t}$. Further we write
\[ \Delta\da\left\{\alpha=\binmatrix a b c d \in\Mat_2(\Z): \alpha\equiv\binmatrix * * 0 *\
    (\mod N),\ (a,N)=1,\ ad-bc\neq0 \right\}. \]

The map
\[ \binmatrix a b c d \mapsto a \mod p^t \in R \] defines a character \[
  \chi\colon\Delta\ra R^\times. \] We let $\Delta$ act on $R$ via
powers of 
this character and denote the resulting module by $R(\chi^j)$ for $j\in\Z$.
We then look at the $R$-linear map
\[ \Sym^{k-2}R^2\ra R(\chi^{k-2}),\quad f\mapsto f(1,0), \]
where we view $f$ as a homogeneous polynomial in two variables $X,Y$ of degree $k-2$ (so
this map is the projection onto the coefficient of the monomial $X^{k-2}$).
It is an easy calculation to check that this map is
$\Delta$-equivariant.

We apply the functor $\MS(Np^r,-)$ to this map. The $\Delta$-action on both modules makes
the resulting modular symbols into Hecke modules and we obtain a Hecke equivariant morphism
\begin{equation*}
  \label{eqn:fixed-level-comparison-of-diff-weights}
  \MS_k(Np^r,R)\isom\MS(Np^r,R(\chi^{k-2})).
\end{equation*}
Here the right hand side is as an $R$-module isomorphic to $\MS_2(Np^r,R)$, but it carries
a twisted Hecke action (more precisely, the action of the diamond operators is twisted,
while the $T_\ell$ operators remain the same because the character $\chi$ is trivial on the
matrices $\binmatrix100\ell$ which represent the double coset for $T_\ell$). We therefore
denote the right hand side by $\MS_2(Np^r,R(\chi^{k-2}))$.

\begin{prop}[Kitagawa]
  After restricting to the ordinary part, the above map induces an isomorphism
  \[ \MS^\ord_k(Np^r,R)\isom\MS^\ord_2(Np^r,R(\chi^{k-2})). \]
\end{prop}
\begin{proof}
  \cite[Cor.\ 5.2]{MR1279604}
\end{proof}

We now take the colimit for $r\to\infty$ (while $t$ is still fixed). This gives us an isomorphism
\begin{multline*}
  \calMS^\ord_k(Np^\infty,R)\da\colim_{r\in\N}\MS^\ord_k(Np^r,R)\isom\\\calMS^\ord_2(Np^\infty,R(\chi^{k-2}))\da\colim_{r\in\N}\MS^\ord_2(Np^r,R(\chi^{k-2})).
\end{multline*}
Now taking the limit for $t\to\infty$ and looking at eigenspaces, we get the following
result comparing modular symbols of different weights, which is \cite[Thm.\
5.3]{MR1279604}. Here the object on the right side is defined to be the limit
$\lim_t\calMS^\ord_2(Np^\infty,R(\chi^{k-2}))$.

\begin{cor}\label{cor:comparison-of-diff-weights}
  There is a canonical $\heckeT^\ord(Np^\infty,\O)$-linear isomorphism
  \[ \overline\calMS^\ord_k(Np^\infty,\O)\isom\overline\calMS^\ord_2(Np^\infty,\O(\chi^{k-2})). \]
  After restricting to the eigenspace of a character $\ep\colon\Gamma^\wt\ra\O^\times$ it
  induces an isomorphism of $\O$-modules
  \[
    \overline\calMS^\ord_k(Np^\infty,\O)[\ep]\isom\overline\calMS^\ord_2(Np^\infty,\O)[\ep\kappa_\wt^{k-2}]. \]
  These isomorphisms commute with the $T_\ell$ operators and with the action of $\EP$.
\end{cor}

We now sketch the proof of \cref{thm:control-thm-ms}. It results immediately from combining
the following three lemmas, whose detailed proofs we omit.


In the following, we regard $\O$ as an $\calI$-algebra via $\varphi$; in particular, we regard $\O$ as a $\Lambdawt$-algebra. Moreover, we consider modules over $\Tord(Np^\infty,\O)\tensor_{\Lambdawt}\calI$ or similar rings having an additional action of $\GR$ as $\Tord(Np^\infty,\O)\tensor_{\Lambdawt}\calI[\GR]$-modules

\begin{lem}\label{lem:ms-control-lemma-red}
  There are canonical isomorphisms of $\Tord(Np^\infty,\O)\tensor_{\Lambdawt}\calI[\GR]$-modules induced by $\varphi$
  \[ \MMSS^\ord(Np^\infty,\calI)\tensor_{\calI}\left(\quotient\calI P\right)\isom\Hom_{\Lambdawt}(\UM^\ord(Np^\infty,\O),\O) \]
  and, if we assume that \cref{cond:mmss-free-rank-one} is satisfied, also an isomorphism of $\O$-modules
  \[ \MMSS^\ord(Np^\infty,\calI)^\pm[F]\tensor_{\calI}\left(\quotient\calI P\right)\isom\Hom_{\Lambdawt}(\UM^\ord(Np^\infty,\O),\O)^\pm[F]. \]
\end{lem}
\begin{proof}
  This follows by applying the (exact!) functor $\Hom_{\Lambdawt}(\UM^\ord(Np^\infty,\O),-)$
  to
  \begin{equation*}
     0\ra P \ra \Lambdawt \ra[$\varphi$] \O \ra 0
   \end{equation*}
   and using the fact that height $1$ prime ideals of $\Lambdawt$ are principal by
   \cite[Lem.\ 5.3.7]{MR2392026}.
\end{proof}

We put \[ \overline\calMS^\ord_2(Np^\infty,\O)[F_\varphi] \da \overline\calMS^\ord_2(Np^\infty,\O)[\varphi\circ F], \]
by which we mean the submodule of $\overline\calMS^\ord_2(Np^\infty,\O)$ where the action of the Hecke algebra $\Tord(Np^\infty,\O)\tensor_{\Lambdawt}\calI$ is given by the character $\varphi\circ F\colon\Tord(Np^\infty,\O)\tensor_{\Lambdawt}\calI\ra\O$.

\begin{lem}\label{lem:ms-control-lemma-duality}
  The canonical biduality map
  \begin{align*}
     \Phi\colon\overline\calMS^\ord_2(Np^\infty,\O)\ra &\Hom_\O(\UM(Np^\infty,\O),\O) \\ &= \Hom_\O(\Hom_\O(\overline\calMS_2(Np^\infty,\O)),\O) \\
     \xi \mapsto &[f\mapsto f(\xi)]
  \end{align*}
  induces an isomorphism of $\Tord(Np^\infty,\O)\tensor_{\Lambdawt}\calI[\GR]$-modules
  \[ \overline\calMS^\ord_2(Np^\infty,\O)[\ep\kappa_\wt^{k-2}]\isom\Hom_{\Lambdawt}(\UM^\ord(Np^\infty,\O),\O). \]
  We further have an inclusion 
  \[ \overline\calMS^\ord_2(Np^\infty,\O)[F_\varphi]\subseteq\overline\calMS^\ord_2(Np^\infty,\O)[\ep\kappa_\wt^{k-2}] \]
  and the restriction of the above isomorphism gives an isomorphism of $\O$-modules
  \[ \overline\calMS^\ord_2(Np^\infty,\O)^\pm[F_\varphi]\isom\Hom_{\Lambdawt}(\UM^\ord(Np^\infty,\O),\O)^\pm[F]. \]
\end{lem}
\begin{proof}
  The biduality map is injective because the pairing \eqref{eqn:perfect-pairing-um} is
  perfect. That the image lies in the correct subspaces is a straightforward
  calculation. Finally, surjectivity is again obtained using the perfectness of the pairing.
\end{proof}

\begin{lem}\label{lem:ms-control-lemma-iotainfty}
  There are canonical isomorphisms of $\Tord(Np^\infty,\O)\tensor_{\Lambdawt}\calI[\GR]$-modules
  \[ \MS_k^\ord(Np^r,\O)[\ep]\isom\overline\calMS^\ord_2(Np^\infty,\O)[\ep\kappa_\wt^{k-2}] \]
  and of $\O$-modules
  \[ \MS_k^\ord(Np^r,\O)^\pm[F_\varphi]\isom\overline\calMS^\ord_2(Np^\infty,\O)^\pm[F_\varphi]. \]
\end{lem}
\begin{proof}
  We know that the map $\MS^\ord_k(Np^r,\O)\ra\calMS^\ord_k(Np^\infty,\O)$ is compatible with the action of $\Tord(Np^\infty,\O)\tensor_{\Lambdawt}\calI[\GR]$, so in particular with the action of $\Gamma^\wt$, and therefore induces a map
  \begin{equation*}
    \MS_k^\ord(Np^r,\O)[\ep]\ra\overline\calMS^\ord_k(Np^\infty,\O)[\ep].
  \end{equation*}
  By \cite[Thm.\ 5.5 (1)]{MR1279604}\footnote{In \cite[§5.3]{MR1279604}, the character $\ep$ is assumed to have kernel $\Gamma_s^\wt$ instead of $\Gamma^\wt_r$ (for some $s\le r$). This must be a typo since in this case the claim \enquote{$L_n(A)=L_n(\ep,A)$ as $\Gamma_1(Np^r)$-module} there is not true.} this map is in fact an isomorphism.
  The first isomorphism follows therefore from \cref{cor:comparison-of-diff-weights}.
  Now since $P$ is of type $(k,\ep,r)$, we have 
  \begin{align*}
    \MS_k^\ord(Np^r,\O)[F_\varphi]&\subseteq\MS_k^\ord(Np^r,\O)[\ep], \\
    \overline\calMS^\ord_2(Np^\infty,\O)[F_\varphi]&\subseteq\overline\calMS^\ord_2(Np^\infty,\O)[\ep\kappa_\wt^{k-2}].
  \end{align*}
  The second inclusion comes from \cref{lem:ms-control-lemma-duality}. Since the first isomorphism from the statement is compatible with the Hecke action as well as with the action of $\EP\in\GL_2(\Z)$, it induces the second isomorphism.
\end{proof}

\subsection{The Galois action on modular symbols}

The comparison isomorphism from étale cohomology yields an isomorphism of $\O$-modules
\[
  \MS_k(Np^r,\O)\cong\HL^1_{\et,\cp}(Y_1(Np^r)\fibertimes_\Z{\Qquer},\Sym^{k-2}\RD^1f_*\smuline\O) \]
commuting with the Hecke action. Hence $\MS_k(Np^r,\O)$ carries a canonical $\O$-linear
action of $\GQ$ which commutes with the Hecke action.



For $s\ge r\ge 0$, the maps $\MS_k(Np^r,\O)\ra\MS_k(Np^s,\O)$ are $\GQ$-equivariant since
they are induced by the canonical maps of curves $Y_1(Np^s)\ra Y_1(Np^r)$, which are defined
over $\Q$. Hence by definition of $\overline\calMS_k(Np^\infty,\O)$ we can extend the action
of $\GQ$ to this $\O$-module. Of course it still commutes with the Hecke action. We then
endow $\UM(Np^\infty,\O)$ with the dual $\GQ$-action and $\MMSS^\ord(Np^\infty,\calI)$ again
with the dual $\GQ$-action. It is clear that the $\GQ$-action on
$\MMSS^\ord(Np^\infty,\calI)$ is then $\calI$-linear and still commutes with the Hecke
action.

By $\calI$-linearity, we thus get a $\GQ$-action on the reduction of
$\MMSS^\ord(Np^\infty,\calI)$ at arithmetic points. We want to show that the control theory
isomorphisms from the previous section are $\GQ$-equivariant. For doing so we will need the
following
well-known
result on elliptic curves.
In the
statement we identify the scheme $E[p^t]$, which is finite \'etale over $S$, with the
\'etale sheaf it represents.

\begin{prop}\label{prop:ec-relative-tate-module-etale-cohom}
  Let $S$ be a $\Q$-scheme, $f\colon E\ra S$ an elliptic curve and $t\in\N$. Then $\RD^1f_*\smuline{\zmod{p^t}}$ and $E[p^t]$ are canonically $\zmod{p^t}$-dual to each other as \'etale sheaves on $S$.
\end{prop}

\begin{thm}\label{thm:control-theory-gq-equiv}
  Fix a $P\in\calX^\arith_\calI(\O)$ of type $(k,\ep,r)$.
  The $\O$-linear isomorphism
  \[ \MMSS^\ord(Np^\infty,\calI)\tensor_{\calI}\left(\quotient\calI P\right) \isom
    \MS_k^\ord(Np^r,\O)[\ep] \] from \cref{thm:control-thm-ms} is $\GQ$-equivariant.
\end{thm}
\begin{proof}
  The isomorphism we study was defined as the composition of three maps: the reduction map from \cref{lem:ms-control-lemma-red}, the biduality map from \cref{lem:ms-control-lemma-duality} and the map which compares modular symbols of different weights from \cref{lem:ms-control-lemma-iotainfty}. By construction and the definition of the Galois actions, it is clear that the first two respect the action of $\GQ$. The third one was defined using the canonical isomorphism
  \begin{equation*}
    \overline\calMS^\ord_k(Np^\infty,\O)[\ep]\isom\overline\calMS^\ord_2(Np^\infty,\O)[\ep\kappa_\wt^{k-2}]
  \end{equation*}
  which came from the map
  \begin{equation*}
    \Hc^1(X_1(Np^r),\Sym^{k-2}\RD^1f_*\smuline{\O/{p^t}})=
    \MS_k(Np^r,\textstyle\quotient\O{p^t})\ra\MS_2(Np^r,\textstyle\quotient\O{p^t}(\chi^{k-2}))
  \end{equation*}
  from \eqref{eqn:fixed-level-comparison-of-diff-weights} (with $r\ge t\ge0$ fixed),
  and it remains to prove that this map is $\GQ$-equivariant.
  The latter map is induced from a morphism of sheaves on $Y_1(Np^r)^\an$
  \begin{equation*}
    \Sym^{k-2}\RD^1f_*\underline{\textstyle\quotient\O{p^t}}\ra\underline {\textstyle\quotient\O{p^t}}
  \end{equation*}
  which is $\Sym^{k-2}$ of a map
  $\RD^1f_*\underline{\textstyle\quotient\O{p^t}}\ra\underline
  {\textstyle\quotient\O{p^t}}$. With respect to our fixed trivialisation of
  $\RD^1f_*\smuline\Z$ on $\H$ (by the choice of a basis we made on
  \cpageref{choice-of-basis}), it comes from projection onto the first coordinate. We note
  here that in fibers of $E_1(N)^\an\ra Y_1(N)^\an$ the point of order $N$ lies in the
  \emph{second} coordinate with respect to this basis (see
  \cpageref{point-of-exact-order-N}). This will be used below.

  Since the category of locally constant torsion sheaves with finite fibres on $Y_1(N)^\an$ is equivalent to the category of locally constant constructible torsion sheaves \'etale on $Y_1(N)\fibertimes_\Z\C$ by \cite[Exp.\ XI, Thm.\ 4.4 (i)]{MR0354654}, the above morphism of sheaves corresponds to a morphism of such \'etale sheaves.
  To see $\GQ$-equivariance, we need to show that this map of \'etale sheaves already exists over $\Q$, i.\,e.\ it comes from a map of \'etale sheaves defined over $\Q$
  \begin{equation*} 
    \RD^1f_*\underline{\textstyle\quotient\O{p^t}}\ra\underline {\textstyle\quotient\O{p^t}}
  \end{equation*}
  on $Y_1(Np^r)_{/\Q}$. We construct such a map which after base change to $\C$ gives back the morphism from before.

  We now work over $\Q$. Take an \'etale open $U$ in $Y_1(Np^r)_{/\Q}$. Then on $E_1(Np^r)_{/\Q}\fibertimes_{Y_1(Np^r)_{/\Q}}U\ad E_U$ the level structure gives us a point of exact order $p^r$, i.\,e.\ a morphism of group schemes
  \begin{equation*}
    \alpha\colon{\underline{\textstyle\zmod{p^r}}}_{/U}\inj E_U[p^r].
  \end{equation*}
  If we compose $\alpha$ with the map $E_U[p^r]\ra E_U[p^t]$ which is multiplication by $p^{r-t}$, then it is easy to see that it factors through $\textstyle\zmod{p^t}$ and gives a point of exact order $p^t$
  \begin{equation*}
    \beta\colon{\underline{\textstyle\zmod{p^t}}}_{/U}\inj E_U[p^t]
  \end{equation*}
  (here we just applied the change of level morphism $\sigma_{p^r,p^t}$ from
  \cref{sec:refinements}). Dualizing $\beta$ and using
  \cref{prop:ec-relative-tate-module-etale-cohom} we obtain a surjection
  \begin{equation}
    \label{eqn:dual-to-point-of-order-pt}\tag{$*$}
    \RD^1f_*\underline{\textstyle\zmod{p^t}}\surj\underline{\textstyle\zmod{p^t}}
  \end{equation}
  of \'etale sheaves on $U$.

  By \cref{prop:ec-relative-tate-module-etale-cohom} and the well-known structure of the
  $N$-torsion of an elliptic curve (see \cite[Thm.\ 2.3.1]{MR0772569}),
  $\RD^1f_*\underline{\textstyle\zmod{p^t}}$ is \'etale locally on $Y_1(Np^r)_{/\Q}$
  isomorphic to the constant sheaf $\underline{(\textstyle\zmod{p^t})}^2$. Now assume $U$ is
  small enough and choose an isomorphism of sheaves on $U$
  \begin{equation*}
    \psi_U\colon\RD^1f_*\underline{\textstyle\zmod{p^t}}\isom\underline{(\textstyle\zmod{p^t})^2}
  \end{equation*}
  such that the composition $\pi_1\circ\psi_U$ is the map \eqref{eqn:dual-to-point-of-order-pt} (where $\pi_1,\pi_2\colon \underline{(\zmod{p^t})}^2\ra\underline{(\zmod{p^t})}$ are the projections on the first and second factor, respectively). We then define a morphism of sheaves on $U$
  \begin{equation*}
    \phi_U\colon\RD^1f_*\underline{\textstyle\zmod{p^t}}\ra\underline{\textstyle\zmod{p^t}}
  \end{equation*}
  as $\phi_U\da\pi_2\circ\psi_U$ and claim that this globalizes to a morphism of sheaves $\RD^1f_*\underline{\zmod{p^t}}\ra\underline{\zmod{p^t}}$ on $Y_1(Np^r)_{/\Q}$. To check this, take two small \'etale open sets $U_1,U_2$ as above. Then $\psi_{U_2}^{-1}\restrict{U_1\cap U_2}\circ\psi_{U_1}\restrict{U_1\cap U_2}$ is an automorphism of the constant sheaf $(\zmod{p^t})^2$ on $U_1\cap U_2$, so it may be described by a tuple of matrices in $\GL_2(\zmod{p^t})$. But since it has to respect the point of order $p^t$ each matrix has to be upper unitriangular. So the automorphism does not change the projection onto the second factor,
  i.\,e.
  \begin{equation*}
    \phi_{U_1}\restrict{U_1\cap U_2}=\pi_2\restrict{U_1\cap U_2}\circ\psi_{U_1}\restrict{U_1\cap U_2}=\pi_2\restrict{U_1\cap U_2}\circ\psi_{U_2}\restrict{U_1\cap U_2}=\phi_{U_2}\restrict{U_1\cap U_2}
  \end{equation*}
  and we indeed get a morphism of sheaves $\RD^1f_*\underline{\zmod{p^t}}\ra\underline{\zmod{p^t}}$ on $Y_1(Np^r)_{/\Q}$. By construction and our previous considerations, it is clear that after tensoring with $\quotient{\O}{p^t}$ and base change to $\C$ we get the morphism $\RD^1f_*\underline{\quotient\O{p^t}}\ra\underline {\quotient\O{p^t}}$ from before.
\end{proof}

\begin{cor}\label{cor:mmss-ist-hidas-grosse-galdarst}
  Let $F\in\SS^\ord(Np^\infty,\calI)$ be an $\calI$-adic eigenform and assume that \cref{cond:mmss-free-rank-one} is satisfied. Then the eigenspace $\MMSS^\ord(Np^\infty,\calI)[F]$ is as an $\calI$-linear representation of $\GQ$ 
  isomorphic to Hida's big representation $\rho_F$ from \cref{thm:grosse-galdarst-hida-fam}.
\end{cor}
\begin{proof}
  It follows from \cref{cond:mmss-free-rank-one} that the eigenspace $\MMSS^\ord(Np^\infty,\calI)[F]$ is free of rank $2$ over $\calI$. Moreover, under our assumptions the map
  \begin{equation*}
    \MMSS^\ord(Np^\infty,\calI)[F]\tensor_{\calI}\left(\quotient\calI P\right) \isom \MS_k^\ord(Np^r,\O)[F_P]=\MS_k(Np^r,\O)[F_P]
  \end{equation*}
  from \cref{thm:control-thm-ms} is $\GQ$-equivariant by \cref{thm:control-theory-gq-equiv}
  and the fact that the Hecke and Galois actions commute. We further have a canonical
  $\O$-linear $\GQ$-equivariant isomorphism
  \begin{equation*}
    \MS_k(Np^r,\O)[F_P]\cong\HL^1_{\et,\cp}(Y_1(Np^r)\fibertimes_\Z{\Qquer},\Sym^{k-2}\RD^1f_*\smuline\O)[F_P].
  \end{equation*}
  After tensoring this isomorphism with $L$ we get a
  two-dimensional vector space by \cref{prop:ms-eigenspaces-free-of-rank-one}. If $F_P$ is a
  newform then by definition of $\Mf$ this vector space is isomorphic to Deligne's Galois
  representation attached to $F_P$.  Otherwise $F_P$ is the unique ordinary refinement of a
  newform $F_P^\new$, and then the vector space is isomorphic to Deligne's Galois
  representation attached to $F_P^\new$.  Hence $\MMSS^\ord(Np^\infty,\calI)[F]$ has the
  same properties as Hida's big Galois representation $\rho_F$ from
  \cref{thm:grosse-galdarst-hida-fam}, and since $\rho_F$ is unique with these properties
  the claim follows.
\end{proof}

\begin{prop}\label{cor:mmss-as-proj-limit-gqp-equiv}
  The isomorphism from \cref{cor:mmss-as-proj-limit} is $\GQ$-equivariant.
\end{prop}
\begin{proof}
  This follows directly from \cref{thm:control-theory-gq-equiv} and the formula in \cref{cor:mmss-as-proj-limit}.
\end{proof}

\subsection{$p$-adic Eichler-Shimura isomorphisms}

For a fixed weight $k\ge2$ and level $N\ge4$ we have the following $p$-adic analogue of the
Eichler-Shimura isomorphism due to Faltings.
We look at the $p$-adic {Hodge-Tate comparison isomorphism for $\wnk$}
\[ \cpiso_\HT\colon \wnk_p\tensor_{\Qp} \BHT\isom \wnk_\hodge\tensor_\Q \BHT. \] It is an
isomorphism of graded vector spaces. Let us take its degree $0$ part. From
\cref{thm:wnk-betti-p} and \cref{cor:hodge-realization} we have an explicit description of the
vector spaces involved, and we obtain the following, which can be viewed as a {$p$-adic
  analogue of the Eichler-Shimura isomorphism}.

\begin{thm}[Faltings]\label{thm:faltings-padic-es}
  There is a canonical $\GQp$-equivariant and Hecke equivariant isomorphism
  \begin{multline*}
    \Hpet^1(Y_?(N)\fibertimes_\Z\Qpbar,\Sym^{k-2}\RD^1f_*\smuline\Z_p)\tensor_{\Z_p}\C_p \isom \\
    \S_k(\Gamma_?(N),\C_p)(1-k)
    \;\oplus\;
    \HL^1(X_?(N),\omega_{\overline E(N)/X(N)}^{2-k})\tensor_\O\C_p.
  \end{multline*}
  Here, \enquote{$?$} is either nothing or \enquote{$1$} and $\GQp$ acts diagonally on the left side and through $\Cp$ on the right side.
\end{thm}

This isomorphism can be interpolated in Hida families, as shows the next theorem proved by
Kings, Loeffler and Zerbes building on work of Ohta and Kato. In the statement we view
$\SS^\ord(Np^\infty,\Lambdawt)$ as a $\heckeT^\ord(Np^\infty,\O)$-module via the natural map
$\heckeT^\ord(Np^\infty,\O)\ra\hecket^\ord(Np^\infty,\O)$.

\begin{thm}\label{thm:lambda-adic-es}
  There is a canonical $\heckeT^\ord(Np^\infty,\O)\tensor_{\Lambdawt}\calI$-linear surjection (called the $\calI$-adic Eichler-Shimura map)
  \[ \MMSS^\ord(Np^\infty,\calI) \surj \SS^\ord(Np^\infty,\calI) \]
  such that the following hold.
  \begin{enumerate}
  \item\label{thm:lambda-adic-es:a} If we reduce it modulo the ideal $\omega_{k,r}$, the
    resulting $\heckeT^\ord_k(Np^r,\O)$-linear surjection\footnote{Here we use
      \cref{thm:control-thm-ms} and \cref{thm:control-thm}.}
    \[ \MS^\ord_k(Np^r,\O)\surj\S_k^\ord(X_1(Np^r),\O) \]
    fits into a commutative diagram
    \[\begin{tikzpicture}
      \matrix (m) [matrix of math nodes, row sep=2em, column sep=1.5em, text height=1.5ex, text depth=0.25ex]
      { \HL_{\et,\cp}^1(Y_1(Np^r)\fibertimes_\Z\Qpbar,\Sym^{k-2}\RD^1f_*\smuline\O)^\ord & \S_k^\ord(X_1(Np^r),\O) \\
        \HL_{\et,\parab}^1(Y_1(Np^r)\fibertimes_\Z\Qpbar,\Sym^{k-2}\RD^1f_*\smuline\C_p) & \S_k(\Gamma_1(Np^r),\C_p), \\};
      \path[->,font=\scriptsize]
      (m-1-1) edge (m-2-1)
      (m-1-2) edge (m-2-2);
      \path[->>,font=\scriptsize]
      (m-1-1) edge (m-1-2)
      (m-2-1) edge (m-2-2);
    \end{tikzpicture}\]
    where the bottom row is the $p$-adic Eichler-Shimura isomorphism from \cref{thm:faltings-padic-es} composed with the projection onto the first factor and the vertical maps are the natural ones.\footnote{Here we omitted the Tate twist from \cref{thm:faltings-padic-es} since we are not interested in the Galois action at this point.}
  \item\label{thm:lambda-adic-es:b} The kernel is the submodule $\MMSS^\ord(Np^\infty,\calI)^{\inertia_p}$ fixed under the inertia group.
  \end{enumerate}
\end{thm}
\begin{proof}
  Since $\calI$ is flat over $\Lambdawt$ and the formation of both $\MMSS^\ord(Np^\infty,-)$ and $\SS^\ord(Np^\infty,-)$ is compatible with base change from $\Lambdawt$ to $\calI$, we can assume that $\calI=\Lambdawt$. Similarly we can assume that $\O=\Zp$.

  The claim in \ref{thm:lambda-adic-es:a} is  equivalent to the existence of a commutative diagram
  \begin{equation}
    \label{eqn:klz-compatibility-my-version}
    \tag{$*$}
    \begin{tikzpicture}
      \matrix (m) [matrix of math nodes, row sep=2em, column sep=1.5em, text height=1.5ex, text depth=0.25ex]
      { \MMSS^\ord(Np^\infty,\Lambdawt) & \SS^\ord(Np^\infty,\Lambdawt) \\
        \HL_{\et,\parab}^1(Y_1(Np^r)\fibertimes_\Z\Qpbar,\Sym^{k-2}\RD^1f_*\smuline\C_p) & \S_k(\Gamma_1(Np^r),\C_p), \\};
      \path[->,font=\scriptsize]
      (m-1-1) edge (m-2-1)
      (m-1-2) edge (m-2-2);
      \path[->>,font=\scriptsize]
      (m-1-1) edge (m-1-2)
      (m-2-1) edge (m-2-2);
    \end{tikzpicture}
  \end{equation}
  where the top map is the $\calI$-adic Eichler-Shimura map, the bottom one is the $p$-adic
  comparison isomorphism and the vertical ones are the reduction maps. The existence of this
  diagram is essentially \cite[Thm.\ 9.5.2]{MR3637653}.
  The setting there is slightly different, but by ignoring Galois actions (thus omitting the
  functor denoted $\mathbf D$ there, see \cite[Prop.\ 1.7.6]{FukayaKatoConjecturesSharifi}),
  applying the Atkin-Lehner involution (thus interchanging ordinary and anti-ordinary parts)
  and restricting to the cuspidal subspaces the statement of \cite[Thm.\ 9.5.2]{MR3637653}
  translates into the existence of a commutative diagram
  \begin{equation*}
    \begin{tikzpicture}
      \matrix (m) [matrix of math nodes, row sep=2em, column sep=1.5em, text height=1.5ex, text depth=0.25ex]
      { \lim_r\HL^1_{\et,\parab}(Y_1(Np^r)\fibertimes_\Z\Qpbar,\Zp)^\ord & \lim_r\S^\ord_2(X_1(Np^r)^\arithm,\Zp) \\
        \HL_{\et,\parab}^1(Y_1(Np^r)\fibertimes_\Z\Qpbar,\TSym^{k-2}\RD^1f_*\smuline\Z_p)^\ord & \S_k(\Gamma_1(Np^r)^\arithm,\Zp). \\};
      \path[->,font=\scriptsize]
      (m-1-1) edge (m-2-1)
      (m-1-2) edge (m-2-2);
      \path[->>,font=\scriptsize]
      (m-1-1) edge (m-1-2)
      (m-2-1) edge (m-2-2);
    \end{tikzpicture}
  \end{equation*}
  Here the limits are taken along the morphisms $\Theta_{Np^{r+1},Np^r}$ and $\TSym^{k-2}$
  denotes the invariants of the action of the $(k-2)$-th symmetric group on the $(k-2)$-th
  tensor power (as opposed to coinvariants, which are $\Sym^{k-2}$)

  Now in the bottom line, we can replace $\Zp$ by $\Cp$; this allows us also to replace
  $\TSym$ by $\Sym$ because these become canonically isomorphic as soon as $(k-2)!$ is
  invertible. Of course the diagram as above still exists if we replace $\HL^1_{\et,\parab}$
  by $\HL^1_{\et,\cp}$ in the left column. Then the upper left object is isomorphic to
  $\MMSS^\ord(Np^\infty,\Lambdawt)$ by \cref{cor:mmss-as-proj-limit} and
  the upper right object is isomorphic to $\SS^\ord(Np^\infty,\Lambdawt)$ by
  \cref{cor:ss-as-proj-limit}. We arrive at the diagram
  \eqref{eqn:klz-compatibility-my-version}, which completes the proof of statement
  \ref{thm:lambda-adic-es:a}.

  For statement \ref{thm:lambda-adic-es:b}
  see \cite[Thm., p.\ 50]{MR1332907}.
\end{proof}

\section{Periods, error terms and $p$-adic $L$-functions}

\subsection{Periods of motives and the interpolation formula of Fukaya and Kato}

Recall that we fixed a number field $K$ and a place $\frakp\mid p$ of it with completion
$L$.

Let $M$ be a motive over $\Q$ with coefficients in $K$. We assume
that $M$ is critical and fix $K$-bases $\gamma$ of $M_\betti^+$ and $\delta$ of
$\tangentspace{M}$. Criticality means that the period map \[ M_\betti^+\tensor_K\C\inj
  M_\betti\tensor_K\C\ra[$\cpiso_\infty$]M_\dR\tensor_K\C\surj\tangentspace M\tensor_K\C \]
is an isomorphism (where $\cpiso_\infty$ denotes the complex comparison isomorphism), and we then have Deligne's complex period attached $M$
\[ {\Omega_\infty^{\gamma,\delta}(M)} \da
  \det_{\gamma,\delta}(\cpiso_\infty^+)\in\C^\times \]
which is the determinant with respect to the chosen bases.

Fukaya and Kato \cite{MR2276851} consider a $p$-adic period which we introduce in the
following. Additionally to criticality we need another condition on the motive. We say that
the motive $M$ satisfies the Dabrowski-Panchishkin condition at $\frakp$ if there is a
subspace ${M_\frakp^\DP}\subseteq M_\frakp$ stable under the action of $\GQp$ such that the
inclusion $M_\frakp^\DP\subseteq M_\frakp$ induces an isomorphism
\[ \DdR(M_\frakp^\DP)\isom\quotient{\DdR(M_\frakp)}{\fil^0\DdR(M_\frakp)} \] of $L$-vector
spaces. See \cite[§4.2.3, (C2)]{MR2276851}.
We assume that this condition is satisfied.

The $p$-adic period depends on the choice of an isomorphism
\[ \beta\colon\nrhat L\tensor_LM_\frakp^+\isom\nrhat L\tensor_LM_\frakp^\DP \] which we now
fix. For any de Rham $p$-adic representation $V$ with coefficients in $L$ the $p$-adic
comparison isomorphism induces an isomorphism
\[ \alpha_V\colon\BdR\tensor_L\DdR(V)\isom\BdR\tensor_L V. \] Using this notation, denote by
$\varphi$ the composition
\begin{multline*}
  \BdR\tensor_KM_\betti^+\ra[$\cpiso_\et$]\BdR\tensor_LM_\frakp^+\ra[$\beta$]\BdR\tensor_LM_\frakp^\DP\ra[$\alpha_{M_\frakp^\DP}^{-1}$]\\\ra[$\alpha_{M_\frakp^\DP}^{-1}$]\BdR\tensor_L\DdR(M_\frakp^\DP)\ra[$\operatorname{dp}$]\BdR\tensor_L\left(\quotient{\DdR(M_\frakp)}{\fil^0\DdR(M_\frakp)}\right)\ra[$\cpiso_\dR$]\BdR\tensor_K\tangentspace{M}.
\end{multline*}
Then we define the $\frakp$-adic period of $M$ as
\[
  \Omega^{\gamma,\delta,\beta}_\frakp(M)=\frac{\tdR^{\tH(M_\frakp^\DP)}}{\ep(M_\frakp^\DP)}\det_{\gamma,\delta}(\varphi)\in\BdR^\times. \]
It actually lies in $(\nrhat L)^\times$.  Here $\ep(-)$ denotes the $\ep$-factor attached to
a $p$-adic representation (for a precise definition see \cite[§3.2, §3.3.4]{MR2276851},
where it is denoted $\epsilon_L(\Dpst(-),-)$) and $\tH$ denotes the Hodge invariant, which
is denoted $m$ at \cite[§3.3.4]{MR2276851}. The actual definition of the period in
\cite[§4.1.11]{MR2276851} is less explicit than the one given here, but unravelling it one
easily translates it into this one.

When calculating the $p$-adic period the following observation will be useful.

\begin{rem}\label{rem:kann-p-adische-periode-ueber-cp-berechnen}
  Because the $p$-adic period lies in $(\nrhat L)^\times$ we have that
  \[ \det_{\gamma,\delta}(\varphi)\in \tdR^\Z\Cp\subseteq\BdR. \] Therefore, to compute the
  $p$-adic period we do not have to tensor up to $\BdR$, it can already be computed over
  $\BHT$. More precisely: let $V$ and $W$ be $K$-vector spaces with bases $\gamma$ and
  $\delta$, respectively, and let $W$ be filtered. Let
  $\varphi\colon\BdR\tensor_KV\isom\BdR\tensor_K W$ be an isomorphism of filtered vector
  spaces such that $\det_{\gamma,\delta}(\varphi)=\alpha\tdR^h$ with $\alpha\in\Cp$ and
  $h\in\Z$. We apply the functor $\gr$ to $\varphi$ to obtain an isomorphism of the
  associated graded vector spaces $\varphi'\colon\BHT\tensor_KV\isom\BHT\tensor_K\gr(W)$. If
  we use $\tdR$ to identify $\BHT$ with $\Cp[\tdR,\tdR^{-1}]$, then by construction we know
  that still $\det_{\gamma,\delta}(\varphi')=\alpha\tdR^h$.

  Let us now assume that $W$ is pure of weight $h\in\Z$ (i.\,e.\ $\fil^{h}W=W$,
  $\fil^{h+1}W=0$), which is for example the case if $W$ is one-dimensional. It then even
  suffices to tensor with $\Cp$. More precisely, we have then
  \[ \BHT\tensor_K\gr(W) = \bigoplus_{q\in\Z}\Cp(q-h)\tensor_KW \] and since
  $1\tensor\gamma\in\BHT\tensor_KV$ lies in the weight $0$ part, its image under $\varphi'$
  also lies in the weight $0$ part, i.\,e.\ $\alpha t^h\tensor\delta\in\Cp(-h)\tensor_K
  W$. Therefore if we define $\varphi''\da\gr^0(\varphi)$, which is an isomorphism
  \[ \varphi''\colon\Cp\tensor_KV\isom\Cp(-h)\tensor_KW, \] then we know that
  $\det_{\gamma,\delta}(\varphi)=\det_{\gamma,\delta}(\varphi'')$.
\end{rem}

The interpolation formula we are about to state involves another expression which we now introduce.

\begin{dfn}\label{dfn:local-correction-factor}
  Define the {local correction factor at $p$} by
  \[ {\operatorname{LF}_p(M)} \da \frac{\EulerFactorP_p(M_\frakp,T)}{\EulerFactorP_p(M_\frakp^\DP,T)}\Big|_{T=1}\cdot \EulerFactorP_p((M_\frakp^\DP)^*(1),1)\in L, \]
  where $\EulerFactorP_p$ is the polynomial defined by 
  \[ {\EulerFactorP_p(V,T)}\da\det_L(1-\frobcris T,\Dcris(V))  \]
  for a representation $V$ of $\GQp$ on a finite-dimensional $L$-vector space (where we
  denote by $\frobcris$ the Frobenius on $\Dcris$).
\end{dfn}

There is a technical restriction on the evaluation points at which we can hope to describe the value of the $p$-adic $L$-function. We thus introduce the following notion. 

\begin{dfn}\label{dfn:appropriate-pair}
  Let $\rho\colon\GQ\ra\GL_r(L')$ be an Artin representation with coefficients in a finite
  extension $L'$ of $L$ and $n\in\N$. We say that $(\rho,n)$ is an {appropriate pair}
  for $M$ if
  \begin{arabiclist}
  \item $M(\rho^*)(n)$ is still critical,
  \item $\operatorname{LF}_p(M(\rho^*)(n))\neq0$,
  \item $\Hf^i(\Q,V)=\Hf^i(\Q,V^*(1))=0$ for $V=M(\rho^*)(n)_\frakp$ and $i=0,1$.
  \end{arabiclist}
\end{dfn}

These conditions are formulated in \cite[Prop.\
4.2.21]{MR2276851}. There is a further condition related to a
set called $\Upsilon$ there, but this set will be empty in our setting, so this further
condition is vacuous and we omit it.

Now we can state the conjectural interpolation formula, which appears in \cite[Thm.\
4.2.22]{MR2276851}. It is actually formulated in the context of non-commutative Iwasawa
theory, but as we want to study this in the classical commutative setting, we only
state this special case. The Galois group we consider will be
$G\da\Gal{\Q(\mu_{Dp^\infty})}{\Q}$, where $D\in\Z$ is any integer prime to $p$.
We let $\Lambda\da\psring\O G$.

The $p$-adic $L$-function depends on a choice of a certain isomorphism ${\beta}$. In the
following, $\widetilde{(-)}$ means the notation introduced in \cite[§3.1.1]{MR2276851}.  Let
$t$ be an $\O$-stable lattice in $M_\frakp$ and put $t^\DP\da t\cap M_\frakp^\DP$.  Define
${T}\da \Lambda\tensor_\O t$ and ${T^\DP}\da \Lambda\tensor_\O t^\DP$, see
\cite[§4.2.7]{MR2276851}. Let $g\in\GQ$ act on $T$ by $x\tensor y\mapsto xg^{-1}\tensor gy$
and analogously on $T^\DP$. Then fix an isomorphism of $\tilde\Lambda$-modules
$\beta\colon\tilde\Lambda\tensor_\Lambda T^+\isom\tilde\Lambda\tensor_\Lambda T^\DP$. Such
an isomorphism exists by \cite[Lem.\ 4.2.8]{MR2276851}.
It is then easy to see that if $\rho\colon G\ra\GL_r(K')$ is a
representation with coefficients in a finite extension $K'$ of $K$ and $n\in\N$ such that
$M(\rho^*)(n)$ is still critical, $\beta$ induces canonically an isomorphism
\[ \beta(\rho^*,n)\colon\nrhat L\tensor_LM(\rho^*)(n)_\frakp^+\ra \nrhat L\tensor_LM(\rho^*)(n)_\frakp^\DP. \]
After having fixed these data, the formula reads as follows.
Here and in the following ${\kappa_\cyc}$ denotes the cyclotomic character.

\begin{conj}[Fukaya/Kato]\label{conj:palf-motive}
  There is a {$p$-adic $L$-function}, which is an element of the quotient ring of
  $\tilde\Lambda$, such that for each appropriate pair $(\rho,n)$ for $M$ (where $\rho$ is
  an Artin representation of $G$ and $n\in\Z$), the value of the $p$-adic $L$-function at
  $\rho\kappa_\cyc^{-n}$ is
  \[ \prod_{j\ge1}(j-1)!^{\dim_K\gr^{n-j}M_\dR}\operatorname{LF}_p(M(\rho^*)(n))\frac{\Omega_\frakp^{\beta(\rho^*,n)}(M(\rho^*)(n))}{\Omega_\infty(M(\rho^*)(n))}L(M(\rho^*),n). \]
\end{conj}

The existence of this $p$-adic $L$-function is a consequence of the Equivariant Tamagawa
Number Conjecture and the $\ep$-isomorphism conjecture, which is one of the main results of
Fukaya's and Kato's article.

\subsection{Calculation of the periods}

\subsubsection{Error terms}
\label{sec:error-terms}

We recall the definition of error terms defined via modular symbols (often called
periods, but we refrain from using this terminology and use the word period exclusively for
an expression defined via comparison isomorphisms).

We begin with the complex one. Fix $N\ge4$ and a normalized Hecke eigenform $f\in\S_k(X_1(N),K)$.

\begin{dfn}\label{dfn:complex-error-term}
  Choose $\O_K$-bases ${\eta_f^\pm}$ of $\MS_k(N,\O_K)^\pm[f]$, which is free of rank
  $1$ by \cref{prop:ms-eigenspaces-free-of-rank-one}. Because
  \[ \MS_k(N,\O_K)\tensor_{\O_K}\C=\MS_k(N,\C), \]
  there exist unique $\errorterm_\infty(f,\eta_f^\pm)\in\C^\times$ such that
  \[ \xi_f^\pm=\errorterm_\infty(f,\eta_f^\pm)\eta_f^\pm \]
  in the right hand side. They are called the {complex error terms attached to $f$}.
\end{dfn}

At this points let us briefly study refinements of modular symbols. Assume that $p\nmid N$
and let $\alpha$ and $\beta$ be the roots of its $p$-th Hecke polynomial (without loss of
generality assume that they lie in $\O_K$).  Using the same techniques as in
\cref{sec:refinements}, we can define canonical morphism
$\Ref_\alpha\colon\MS_k(N,K)\ra\MS_k(Np,K)$ such that the following hold:
\begin{enumerate}
\item It induces isomorphisms \[ \Ref_\alpha\colon\MS_k(N,K)^\pm[f]\isom\MS_k(Np,K)^\pm[f_\alpha]. \]
  If $f$ is ordinary at a prime $\frakp\mid p$ of $K$ and $\alpha$ is the unit root of the $p$-th Hecke polynomial, then it induces
  \[ \Ref_\alpha\colon\MS_k(N,\O_K)^\pm[f]\isom\MS_k(Np,\O_K)^\pm[f_\alpha]. \]
\item The diagram
  \[ \begin{tikzpicture}
      \matrix (m) [matrix of math nodes, row sep=2em, column sep=2.5em, text height=1.5ex, text depth=0.25ex]
      {  \MS_k(N,K) &  \MS_k(Np,K) \\
        \wnk_\betti\tensor_\Q K & \wnklevel{Np}_\betti\tensor_\Q K \\};
      \path[->,font=\scriptsize]
      (m-1-1) edge node [above] {$\Ref_\alpha$} (m-1-2)
      (m-1-1) edge (m-2-1)
      (m-1-2) edge (m-2-2)
      (m-2-1) edge node [above] {$\Ref_\alpha$} (m-2-2);
    \end{tikzpicture} \]
  commutes. Here the vertical arrows are the maps \eqref{eqn:map-ms-betti} and the bottom map is induced by the motivic refinement morphism
  from \cref{cor:motivic-refinements}.
\item Let $f_\alpha$ be the refinement of $f$ at $\alpha$. Then $\Ref_\alpha(\xi_f)=\xi_{f_\alpha}$.
\end{enumerate}

Let now $f$ be ordinary at $\frakp$. If we now choose $\eta_f^\pm\in\MS_k(N,\O)^\pm[f]$ as above, we may take $\eta_{f_\alpha}^\pm\da\Ref_\alpha(\eta_f^\pm)$ as a basis of $\MS_k(Np,\O)^\pm[f_\alpha]$. The following is then clear.
\begin{cor}\label{lem:complex-error-terms-refinements}
  If $f$ is ordinary at $\frakp$ and $\alpha$ is the unit root, then $\errorterm_\infty(f,\eta_f^\pm) = \errorterm_\infty(f_\alpha,\eta_{f_\alpha}^\pm)$.
\end{cor}

We now turn to Hida families and introduce $p$-adic error terms.
We use again the notation from \cref{setting:hida-families}. Fix an eigenform
$F\in\SS^\ord(Np^\infty,\calI)$ and assume that \cref{cond:mmss-free-rank-one} is satisfied.

We choose an $\calI$-basis ${\Xi^\pm}$ of $\MMSS^\ord(Np^\infty,\calI)^\pm[F]$ and $\O$-bases $\eta_\varphi^\pm$ of $\MS_k(Np^r,\O)^\pm[F_\varphi]$ for each $\varphi\in\calX^\arith_\calI(\O)$ of type $(k,\ep,r)$.
Let \[ \MMSS^\ord(Np^\infty,\calI)^\pm[F]\tensor_{\Lambdawt}\left(\quotient\calI \varphi\right) \isom \MS_k^\ord(Np^r,\O)^\pm[F_\varphi] \] be the canonical isomorphism from \cref{thm:control-thm-ms}. Both sides are free $\O$-modules of rank $1$ by \cref{prop:ms-eigenspaces-free-of-rank-one}. For $\varphi\in\calX^\arith_\calI(\O)$ write $\Xi^\pm_\varphi$ for the image of $\Xi^\pm\in\MMSS^\ord(Np^\infty,\calI)^\pm[F]$ in $\MS_k^\ord(Np^r,\O)^\pm[F_\varphi]$ under the above isomorphism.

\begin{dfn}\label{dfn:p-adic-error-term}
  For each $\varphi\in\calX^\arith_\calI(\O)$, let $\errorterm_\frakp(\Xi^\pm,\eta_\varphi^\pm)\in\O$ be the unique element such that
  \[ \Xi^\pm_\varphi={\errorterm_\frakp(\Xi^\pm,\eta_\varphi^\pm)}\eta_\varphi^\pm. \]
  This element is called the {$p$-adic error term at $\varphi$}.
\end{dfn}

\begin{prop}[Kitagawa]\label{prop:p-adic-errorterm-not-zero}
  One has always $\errorterm_\frakp(\Xi^\pm,\eta_\varphi^\pm)\neq0$. 
\end{prop}
\begin{proof}
  \cite[Prop.\ 5.12]{MR1279604}
\end{proof}

\subsubsection{Choosing good bases}
\label{sec:bases}

In this section we fix integers $N\ge4$ and $k\ge2$ and a newform $f\in\S_k(\Gamma_1(N),K)$,
and we assume that $K$ contains the $N$-th roots of unity (so that we may identify the
modular curves $X_1(N)^\naive$ and $X_1(N)^\arithm$). We choose bases of the tangent space
and the $\GR$-invariant subspace of the Betti realization of the critical twists of
$\Mf$. With respect to these bases we will later compute periods.

Fix an integer $n$ with $1\le n\le k-1$ and a Dirichlet character $\chi$ of arbitrary
conductor. Then $\Mf(\chi^*)(n)$ is critical 
and we have that
\[ {(\Mf(\chi^*)(n))}_\betti^+ = \Mf_\betti^{\chi(-1)(-1)^n}\tensor{K(n)}_\betti\tensor{\Motive(\chi^*)}_\betti \]
and
\[ \tangentspace{\Mf(\chi^*)(n)} = \gr^0\Mf_\dR\tensor{K(n)}_\dR\tensor{\Motive(\chi^*)}_\dR. \]
The realizations of the motives $K(n)$ and $\Motive(\chi^*)$ are one-dimensional and have
canonical bases which we denote by $\MotiveCanBasis{K(n)}_\betti$ and $\MotiveCanBasis{K(n)}_\dR$ for $K(n)$
and $\DirichletCanBasis_\betti$ and $\DirichletCanBasis_\dR$ for $\Motive(\chi^*)$, respectively 
(see \cite[Ex.\ 2.1]{MR2392359} or \cite[§1.1.3]{MR2103471}).

By \cref{cor:gr-zero-mf} we can choose any $\delta_0\in\S_k(\Gamma(N),K)^\vee$ such that
$\delta_0(w_Nf)=1$ (where $w_N$ is the Atkin-Lehner endomorphism) and use its image in
$\gr^0\Mf_\dR$ (which we denote again by $\delta_0$) as a basis of this space. Note that
$\delta_0$ is then unique with this property since $\gr^0\Mf_\dR$ is one-dimensional. So
\[ \delta\da\delta_0\tensor(\MotiveCanBasis{K(n)}_\dR)^{\tensor
    n}\tensor\DirichletCanBasis_\dR\in\tangentspace{\Mf(\chi^*)(n)} \] is a basis for the
tangent space.

We now turn to the Betti side and recall \cref{lem:ms-wnk-betti-iso}, which gives us
isomorphisms \[ \MS_k(N,K)^\pm[f]\isom\Mf_\betti^\pm \] Further we use the modified pairing
${\paarung{\cdot}{\cdot}}_\betti^\iota$ from \cref{sec:poincare} and the fact that it
induces a perfect pairing between $\Mf_\betti^\pm$ and $\Mf_\betti^\mp$ by
\cref{lem:perfectness-restricted-pairing}. Fix a basis
$\eta^+\da\eta_f^+\in\MS_k(N,\O_K)^+[f]$ as in \cref{dfn:complex-error-term}, and by abuse
of notation denote its image in $\Mf_\betti^+$ under the above map still by $\eta^+$. By the
pairing there exists a unique $\eta^-=\eta_f^-\in\Mf_\betti^-$, coming from an
$\eta^-\in\MS_k(N,\O_K)^-[f]$, such that
\[ {\paarung{\eta^\pm}{\eta^\mp}}_{\betti}^{\iota}=\MotiveCanBasis{K(1-k)}
_\betti. \] We choose
then
\[ \gamma\da\eta^{\chi(-1)(-1)^n}\tensor(\MotiveCanBasis{K(n)}_\betti)^{\tensor n}\tensor\DirichletCanBasis_\betti\in{(\Mf(\chi^*)(n))}_\betti^+ \]
as a basis of the Betti side.

As a side remark, note that the choice of $\eta^+$ is the only non-canonical choice we ever made in this whole process.

We stress that both $\gamma$ and $\delta$ of course depend on $n$ and $\chi$, but we omit
this from their notation. This dependence should be always clear from the context. Also
every element introduced in this section of course depends on $f$, and we will also often
omit this from the notation. Though, later we will consider families of motives of modular
forms parametrized by arithmetic points in the weight space, and we will then put a
subscript \enquote{$\varphi$} to all of the elements introduced here to indicate their
dependence on these points.

\subsubsection{Complex periods}
\label{sec:complex-periods}

We now compute Deligne's complex period of the critical twists of $\Mf$ with respect to our
chosen bases. We use the elements and notations introduced in \cref{sec:bases}. Our
calculation uses the well-known fact that the complex periods of the Tate and Dirichlet
motive are $2\pi\i$ resp.\ the Gauß sum of the character.

We remark that a similar idea for calculating the complex periods appears in \cite[§6.1, §6.3]{MR2264660}, although there many details are omitted.

\begin{thm}\label{thm:complex-period}
  We have for the complex period
  \[ \Omega_\infty^{\gamma,\delta}(\Mf(\chi^*)(n))=\frac{(2\pi\i)^{n+1-k}}{\Gausssum(\chi^*)}\errorterm_\infty(f,\eta^s)  \]
  with $s=-\chi(-1)(-1)^n$.
\end{thm}
\begin{proof}
  Recall that we identified $\eta^\pm$ with their images under the map \eqref{eqn:map-ms-betti}.
  By our choices of $\eta^\pm$ we have $\paarung{\eta^{-s}}{w_N\eta^s}_\betti=\MotiveCanBasis{K(1-k)}
_\betti$, and since the pairing $\paarung{\cdot}{\cdot}_\betti^\iota$ vanishes on $\Mf_\betti^\pm\times\Mf_\betti^\pm$, we have further $\paarung{\eta^{-s}}{w_N\eta^s}_\betti=\paarung{\eta^{-s}}{w_N(\eta^s+x)}_\betti$ for any $x\in\Mf^{-s}_\betti$. Therefore (by the definition of the complex error term)
  \begin{equation}
    \label{eqn:erste-rechnung-mit-paarung}\tag{$*$}
   \errorterm_\infty(f,\eta^s)\cdot\MotiveCanBasis{K(1-k)}
_\betti=\paarung{\eta^{-s}}{w_N\errorterm_\infty(f,\eta^s)\eta^s}_\betti=\paarung{\eta^{-s}}{w_N\xi^s}_\betti=\paarung{\eta^{-s}}{w_N\xi}_\betti,
  \end{equation}
  where $\xi$ and $\xi^\pm$ are as in \cref{dfn:xi-f} (which we again identify with their images under the map \eqref{eqn:map-ms-betti}).

  By the compatibility of the comparison isomorphism with the Eichler-Shimura map (see
  \cref{thm:compatibility-es-comparison}) and \cref{lem:xi-maps-to-f}, we have that the
  image of $\xi$ under the map
  \begin{equation*}
    \MS_k(N,\C)\ra\wnk_\betti\tensor\C\isom\wnk_\dR\tensor\C
  \end{equation*}
  is the image of $f$ under the inclusion $\S_k(\Gamma(N),K)\inj\wnk_\dR\tensor K$ coming from the Hodge filtration, so we denote this image by $f$. 
  Now let $\rho\in\wnk_\dR\tensor\C$ be the image of $\eta^{-s}\in\Mf_\betti$ under the comparison isomorphism. Since the comparison isomorphism identifies the pairings ${\paarung{\cdot}{\cdot}}_\dR$ and ${\paarung{\cdot}{\cdot}}_\betti$, \eqref{eqn:erste-rechnung-mit-paarung} is equivalent to
  \begin{equation*}
    {\paarung{\rho}{w_Nf}}_\dR = (2\pi\i)^{1-k}\errorterm_\infty(f,\eta^s)\cdot\MotiveCanBasis{K(1-k)}
_\dR.
  \end{equation*}
  This means that 
  the image of $\rho$ in $\gr^0\Mf_\dR$ is
  $(2\pi\i)^{1-k}\errorterm_\infty(f,\eta^s)\delta_0$.

  Altogether, we see that the isomorphism 
  \begin{equation*}
    {(\Mf(\chi^*)(n))}_\betti^+\isom\tangentspace{\Mf(\chi^*)(n)}
  \end{equation*}
  maps $\gamma=\eta^{-s}\tensor(\MotiveCanBasis{K(n)}_\betti)^{\tensor n}\tensor\DirichletCanBasis_\betti$ to
  \begin{equation*}
    ((2\pi\i)^{1-k}\errorterm_\infty(f,\eta^s)\delta_0)\tensor(2\pi\i\,\MotiveCanBasis{K(n)}_\dR)^{\tensor n}\tensor(\Gausssum(\chi^*)^{-1}\DirichletCanBasis_\dR).
  \end{equation*}
  This completes the proof.
\end{proof}

In particular, Deligne's conjecture on the motive $\Mf(\chi^*)(n)$ holds.

\subsubsection{$p$-adic periods}
\label{sec:p-adic-periods-hida-fam}

We now use again the setup for Hida families as described in
\cref{setting:hida-families}
and
fix a Hida family $F\in\SS^\ord(Np^\infty,\calI)$ which is new.

Throughout this section, we assume that \cref{cond:mmss-free-rank-one} is satisfied. Further
we fix $D\in\Z$ prime to $p$ and let $G\da\Gal{\Q(\mu_{Dp^\infty})}{\Q}$.
We assume that $\O$ contains the $D$-th roots of unity.

Let $\rho_F\colon\GQ\ra\Aut_\calI(\calT)$ be the big Galois representation attached to $F$
from \cref{thm:grosse-galdarst-hida-fam}. In the definition of the $p$-adic period there was
an isomorphism $\beta$ involved. Studying this in families will only make sense if $\beta$
is chosen consistently in the family in a way we now describe. This is motived by the
results in \cite{BarthDiss}.

Put $\bigLambda=\psring{\calI}{G}$. Note that any $\varphi$ induces a map
$\bigLambda\ra\Lambda$ which we also denote by $\varphi$.  The define
${\mathbb T}\da \bigLambda\tensor_{\calI}\calT$ and
${\mathbb T^0}\da \bigLambda\tensor_{\calI}\calT^0$. Let $g\in\GQ$ act on $\mathbb T$ by
$x\tensor y\mapsto xg^{-1}\tensor gy$ and analogously on $\mathbb T^0$. Then fix an
isomorphism of $\bigLambda$-modules
$\beta\colon\mathbb T(1)^+\isom\mathbb T^0(1)$.\footnote{The Tate twist is there in order to
  make the involved motives
  critical.} 
Then for each $\varphi\in\calX_\calI^\arithm$,
$\beta$ induces an isomorphism
\[
  \beta_\varphi\colon\calT(1)^+_\varphi\tensor_{\O}\Lambda\isom\calT^0(1)_\varphi\tensor_{\O}\Lambda. \]

We can say something meaningful about $p$-adic periods only if we can choose this $\beta$ in a \enquote{more or less canonical} way, for which we will need an extra condition which we now explain. Let $\calTexp{0}$ be as in \cref{thm:grosse-galdarst-hida-fam}.

\begin{lem}\label{lem:big-image-hida-family}
  Consider the following conditions.
  \begin{tfaelist}
  \item\label{lem:big-image-hida-family:some} For some choice of an $\calI$-basis of $\calT$ the image of $\rho_F$ contains $\SL_2(\calI)$.
  \item\label{lem:big-image-hida-family:all} For any choice of an $\calI$-basis of $\calT$ the image of $\rho_F$ contains $\SL_2(\calI)$.
  \item\label{lem:big-image-hida-family:one} There exists $\sigma\in\image\rho_F\subseteq\Aut_\calI(\calT)$ such that $\sigma(\calTexp{+})=\calTexp{0}$.
  \end{tfaelist}
  Then \ref{lem:big-image-hida-family:some} and \ref{lem:big-image-hida-family:all} are equivalent and they imply \ref{lem:big-image-hida-family:one}.
\end{lem}
\begin{proof}
  That \ref{lem:big-image-hida-family:some} and \ref{lem:big-image-hida-family:all} are equivalent is obvious. To see that they imply \ref{lem:big-image-hida-family:one}, we choose a submodule $\calT'\subseteq\calT$ complementary to $\calTexp{0}$ and isomorphisms $b_1\colon\calTexp{+}\isom\calTexp{0}$ and $b_2\colon\calTexp{-}\isom\calT'$, which gives us an automorphism $b_1\oplus b_2$ of $\calT=\calTexp{+}\oplus\calTexp{-}=\calTexp{0}\oplus\calT'$. Write $u\in\calI^\times$ for the determinant of $b_1\oplus b_2$ and change one of $b_1$ or $b_2$ by $u^{-1}$. Then the determinant using the new choices will be $1$ and the resulting automorphism $\sigma$ lies in the image of $\rho_F$.
\end{proof}

The above conditions allow us to perform the following trick, which is inspired from \cite[§3.3]{MR1479362}.
\begin{prop}\label{prop:sltwo-implies-condition}
  If the condition from \cref{lem:big-image-hida-family}~\ref{lem:big-image-hida-family:one}
  holds, then after possibly changing the complex embedding $\iota_\infty$, we can assume
  that $\calTexp{+}=\calTexp{0}$.
\end{prop}
\begin{proof}
  Let $\Phi\da b_1\oplus b_2\in\Aut_\calI(\calT)$ be as in
  \cref{lem:big-image-hida-family}~\ref{lem:big-image-hida-family:one}. Then the elements of
  $\calTexp{0}$ remain fixed under $\Phi^{-1}\rho_F(\cc)\Phi$.  Take a $\tau\in\GQ$ such
  that $\rho_F(\tau)=\Phi$. Then if we replace $\iota_\infty$ by $\iota_\infty\circ\tau$,
  the complex conjugation with respect to the new pair of embeddings is
  $\Phi^{-1}\rho_F(\cc)\Phi$, so that then $\calTexp{+}=\calTexp{0}$.
\end{proof}

From now on we assume that the following holds, which can be achieved under any of the conditions from \cref{lem:big-image-hida-family}.
\begin{cond}\label{cond:big-image-hida-family}
  The embeddings $(\iota_\infty,\iota_p)$ are chosen such that $\calTexp{+}=\calTexp{0}$.
\end{cond}

\begin{rem}
  \Cref{cond:big-image-hida-family} seems to be a moderate condition. As explained in \cref{prop:sltwo-implies-condition,lem:big-image-hida-family}, we are safe if the image of $\rho_F$ contains $\SL_2$. There are results on when this happens: in \cite[§10]{MR0860140} Mazur and Wiles show that if $\O=\Zp$, $\calI=\Lambda^\wt$ and the image of the residual representation contains $\SL_2$, then so does the image of $\rho_F$. There is work by J.\ Lang \cite{MR3474844} that can be used to extend this to more general $\calI$ and to relax the condition on the residual representation. Hence there are quite some cases in which we know that the validity of \cref{cond:big-image-hida-family} can be achieved.
\end{rem}

Note that by reducing modulo $\varphi\in\calX_\calI^\arithm$ this implies that
$\Motive(F_\varphi^\new)_\frakp^+=\Motive(F_\varphi^\new)_\frakp^0$ for any $\varphi$.

We have to choose an isomorphisms of $\bigLambda$-modules \[ \beta\colon\mathbb T(1)^+\isom\mathbb T^0(1). \]
Since we have
\begin{align*}
  \mathbb T(1)^+ &= (\bigLambda\tensor_{\calI}\calT\tensor_\Qp\Qp(1))^+ \\
  &= \big(\bigLambda^+\tensor_{\calI}\calTexp{-}\tensor_\Qp\Qp(1)\big)\,\oplus\, \big(\bigLambda^-\tensor_{\calI}\calTexp{+}\tensor_\Qp\Qp(1)\big)
\end{align*}
and
\begin{align*}
  \mathbb T^0(1) &= (\bigLambda\tensor_{\calI}\calTexp{0}\tensor_\Qp\Qp(1)) \\
  &= \big(\bigLambda^+\tensor_{\calI}\calTexp{0}\tensor_\Qp\Qp(1)\big) \,\oplus\, \big(\bigLambda^-\tensor_{\calI}\calTexp{0}\tensor_\Qp\Qp(1)\big),
\end{align*}
we see that any choice of an isomorphism of $\calI$-modules
\[ \beta_0\colon\calTexp{+}\isom\calTexp{-} \]
gives us an isomorphism $\beta$ as above (recall that $\calTexp{+}=\calTexp{0}$).

We now replace the abstract representation $\calT$ by the $\calI$-module $\MMSS^\ord(Np^\infty,\calI)[F]$. By \cref{cor:mmss-ist-hidas-grosse-galdarst} it is isomorphic to $\calT$ as a $\GQ$-representation, so we can assume without loss of generality that $\calT$ is of this concrete form.

\begin{situation}\label{situation:p-adic-periods}
  Here we choose some elements that should be fixed for the following. We have already fixed $D\in\N$ with $p\nmid N$ and defined $G=\Gal{\Q(\mu_{Dp^\infty}}\Q$.
  We choose $\calI$-bases $\Xi^\pm$ of $\MMSS^\ord(Np^\infty,\calI)^\pm[F]$ as in
  \cref{dfn:p-adic-error-term}. Having done so, we define $\beta_0$ as the isomorphism
  \[ \MMSS^\ord(Np^\infty,\calI)^+[F]\isom\MMSS^\ord(Np^\infty,\calI)^-[F] \]
  that sends $\Xi^+$ to $\Xi^-$ and write \[ \beta\colon\mathbb T(1)^+\isom\mathbb T^0(1) \] for the isomorphism induced by it.  
  For each $\varphi\in\calX_\calI^\arithm$ of type $(k,\ep,r)$ we fix bases $\eta_\varphi^\pm\in\MS_k(Np^r,\O_K)^\pm[F_\varphi]$ and $\eta_{\varphi,\new}^\pm\in\MS_k(Np^r,\O_K)^\pm[F_\varphi^\new]$. We assume that $\eta_\varphi^\pm=\eta_{\varphi,\new}^\pm$ whenever $F_\varphi=F_\varphi^\new$ and that $\eta_\varphi^\pm$ and $\eta_{\varphi,\new}^\pm$ are connected via refinement as in \cref{sec:error-terms} -- more precisely that $\Ref_\alpha(\eta_{\varphi,\new}^\pm)=\eta_\varphi^\pm$, where $\alpha$ is the unique unit root of the $p$-th Hecke polynomial of $F_\varphi^\new$.
  Moreover we assume that they are dual to each other under the pairing ${\paarung{\cdot}{\cdot}}_\betti^\iota$, as in \cref{sec:bases}.
  Finally we choose $\delta_{0,\varphi}\in\S_k(\Gamma(Np^r),K)^\vee$ such that $\delta_{0,\varphi}(w_{Np^r}F_\varphi^\new)=1$ as in \cref{sec:bases} to obtain bases of $\gr^0\Motive(F_\varphi^\new)_\dR$.
\end{situation}

\begin{dfn}\label{dfn:constant-u}
  \begin{enumerate}
  \item Let $\varphi\in\calX_\calI^\arithm$ be of type $(k,\ep,r)$ and let $P=P_\varphi$. Reducing $\beta_0$ modulo $P$ gives an isomorphism
    \[ \beta_{0,\varphi}\colon\Motive(F_\varphi^\new)_\betti^+\tensor_KL\cong\Motive(F_\varphi^\new)_\frakp^+=\calTexp{+}_\varphi
    \isom\calTexp{-}_\varphi=\Motive(F_\varphi^\new)_\frakp^-\cong\Motive(F_\varphi^\new)_\betti^-\tensor_KL. \]
    The elements $\eta_{\varphi,\new}^\pm$ are bases of $\Motive(F_\varphi^\new)_\betti^\pm$, respectively. Define \[ {C(\beta_{0,\varphi})}\da\Big(\det_{\eta_{\varphi,\new}^+,\eta_{\varphi,\new}^-}\beta_{0,\varphi}\Big)^{-1}\in L^\times. \]

  \item Let $\Psi^\pm$ be the images of $\Xi^\pm$ under the $\calI$-adic Eichler-Shimura map \[ \MMSS^\ord(Np^\infty,\calI)\ra\SS^\ord(Np^\infty,\calI) \] from \cref{thm:lambda-adic-es}. 
    Since this map is Hecke equivariant, we have in fact $\Psi^\pm\in\SS^\ord(Np^\infty,\calI)[F]$. Since the latter space is free of rank $1$ over $\calI$ and $F$ is a basis, there are unique ${U^\pm}\in\calI$ such that
    \begin{equation}
      \label{eqn:relation-xi-f-u}
      \Psi^\pm=U^\pm F.
    \end{equation}
    For $\varphi\in\calX_\calI^\arithm$, write $U_\varphi^\pm\in\O$ for the reduction of $U^\pm$ modulo $P$.
  \end{enumerate}
\end{dfn}

\begin{lem}\label{lem:calculation-of-c-beta}
  Under our choices, we have \[ \errorterm_\frakp(\Xi^+,\eta^+_\varphi) = C(\beta_{0,\varphi})\errorterm_\frakp(\Xi^-,\eta^-_\varphi). \]
\end{lem}
\begin{proof}
  Recall that $\beta_0$ was defined as
  \begin{equation*}
     \MMSS^\ord(Np^\infty,\calI)^+[F]\isom\MMSS^\ord(Np^\infty,\calI)^-[F], \quad \Xi^+\mapsto\Xi^-.
  \end{equation*}
  By \cref{thm:control-thm-ms} the reduction of this map modulo $P$
  \begin{equation*}
    \MS_k^\ord(Np^r,\O)^+[F_\varphi]\isom\MS_k^\ord(Np^r,\O)^-[F_\varphi]
  \end{equation*}
  sends $\Xi_\varphi^+$ to $\Xi_\varphi^-$,
  where $\Xi^\pm_\varphi$ denotes the images of $\Xi^\pm$ in $\MS_k^\ord(Np^r,\O)^\pm[F_\varphi]$. By \cref{dfn:p-adic-error-term} we have $\Xi^\pm_\varphi=\errorterm_\frakp(\Xi^\pm,\eta_\varphi^\pm)\eta_\varphi^\pm$. Thus by definition of $C(\beta_{0,\varphi})$ it follows
  \begin{equation*}
    C(\beta_{0,\varphi})=\frac{\errorterm_\frakp(\Xi^+,\eta_\varphi^+)}{\errorterm_\frakp(\Xi^-,\eta_\varphi^-)}
  \end{equation*}
  (in the case where $F_\varphi\neq F_\varphi^\new$ we have to take the refinement maps into account, but by our choices of $\eta_{\varphi,\new}^\pm$ and $\eta_\varphi^\pm$ the relation still holds).
\end{proof}

Before we compute the $p$-adic period, we prove the following important result about the constants $U^\pm$.
\begin{thm}\label{thm:u-is-unit}
  We have $U^+=0$, while $U^-\in\calI^\times$. In particular, $U^-_\varphi\in\O^\times$ for each $\varphi\in\calX_\calI^\arithm$.
\end{thm}
\begin{proof}
  By \cref{thm:grosse-galdarst-hida-fam} the representation $\calTexp{0}$ is an unramified direct summand of $\calT$, and since the whole representation $\calT$ is ramified at $p$, we must have $\calTexp{0}=\calTexp{\inertia_p}$ for dimension reasons. Since we further assumed that $\calTexp{0}=\calTexp{+}$ and $\calT=\MMSS^\ord(Np^\infty,\calI)$, we conclude that $\MMSS^\ord(Np^\infty,\calI)^{\inertia_p}=\MMSS^\ord(Np^\infty,\calI)^+$.

  By \cref{thm:lambda-adic-es} the kernel of the $\calI$-adic Eichler-Shimura map is $\MMSS^\ord(Np^\infty,\calI)^{\inertia_p}$, so it follows immediately from the definition of $U^+$ that $U^+=0$.

  On the other hand, fix $\varphi\in\calX_\calI^\arithm$. By a similar reasoning as above we
  get $\MS_k(Np^r,\O)^+[F_\varphi]=\MS_k(Np^r,\O)[F_\varphi]^{\inertia_p}$. We now use that the morphism
  $\MMSS^\ord(Np^\infty,\calI)^\pm[F]\ra\MS_k^\ord(Np^r,\O)^\pm[F_\varphi]$
  is $\GQ$-equivariant by \cref{thm:control-theory-gq-equiv}. This tells us that $\Xi_\varphi^-$ is not fixed under the inertia group $\inertia_p$ (note that we have $\Xi_\varphi^-\neq0$ by \cref{prop:p-adic-errorterm-not-zero}). Further it implies that the kernel of the map $\MS_k^\ord(Np^r,\O)[F_\varphi]\ra\S_k^\ord(X_1(Np^r),\O)[F_\varphi]$ contains $\MS_k(Np^r,\O)[F_\varphi]^{\inertia_p}$, hence it equals $\MS_k(Np^r,\O)[F_\varphi]^{\inertia_p}$, a\-gain for dimension reasons. Therefore we get $\Psi^-_\varphi\neq0$ (with $\Psi^\pm$ as in \cref{dfn:constant-u} and $\Psi^\pm_\varphi$ the reduction modulo $P_\varphi$) and thus $U_\varphi^-\neq0$.

  So $U^-\in\calI$ defines a global section on $\Spec\calI$ that does not vanish at any point $P_\varphi\in\calX_\calI^\arithm$. Since $\calX_\calI^\arithm$ is Zariski dense in $\Spec\calI$ (since it is Zariski dense in $\Spec\calI(\Qpbar)$ and $\Spec\calI$ is $2$-dimensional) and supports of sections are closed, $U^-$ is a nowhere vanishing section, hence a unit.
\end{proof}

We are now almost ready to compute the $p$-adic period. Before that we first prove the
following lemma.

\begin{lem}\label{lem:calculation-ep-th}
  Let $f$ be an ordinary newform of weight $k$ and fix an integer $n$ with $1\le n\le k-1$
  and a Dirichlet character $\chi$ of conductor $Dp^m$ for some $m\in\N$. Write $\alpha$ for
  the unit root of the $p$-th Hecke polynomial. Observe that $\alpha=a_p$ if $p\mid N$,
  where $a_p$ is the $p$-th Hecke eigenvalue of $f$.
  We let $M$ be the motive $\Mf(\chi^*)(n)$.
  \begin{enumerate}
  \item Write $\chi$ as a product $\chi=\chi_\nr\chi_p$ with $\chi_\nr$ of conductor $D$ and
    $\chi_p$ of conductor $p^m$. Then the $\ep$-factor is
    \[ \ep(M_\frakp^\DP)=\alpha^{m}\chi_\nr(p)^mp^{-nm}\Gausssum(\chi_p). \]
  \item The Hodge invariant is $\tH(M_\frakp^\DP)=-n$.
  \end{enumerate}
\end{lem}
\begin{proof}
  If $\delta$ is the character describing the action of $\GQp$ on the $1$-dimensional
  subrepresentatin of $Mf_\frakp$ coming from the fact that $f$ is ordinary then
  $\delta(\Frob_p)=\alpha$.
  Hence the action of $\GQp$ on $M_\frakp^\DP$ is given by the character
  $\delta\tensor\chi^*\tensor\kappa_\cyc^n$ and the first statement follows from
  \cite[Prop.\ 2.3.3]{MR3323528}.

  We turn to the second statement. Since
  $\dim_L\DdR(M_\frakp^\DP)=\dim_L\fil^0\DdR(M_\frakp)=1$ and
  \begin{equation*}
    \DdR(M_\frakp^\DP)\isom\quotient{\DdR(M_\frakp)}{\fil^0\DdR(M_\frakp)},
  \end{equation*}
  we must have $\DdR(M_\frakp^\DP)\cap\fil^0\DdR(M_\frakp)=0$. Further the functor $\DdR$ is
  exact on de Rham representations and $\DdR(M_\frakp^\DP)$ is a subobject of
  $\DdR(M_\frakp)$ as a filtered vector space. It follows that $\tH(M_\frakp^\DP)$ must be
  the unique $i\in\Z$ such that $\fil^i\DdR(M_\frakp)=\DdR(M_\frakp)$ and
  $\fil^{i+1}\DdR(M_\frakp)\neq\DdR(M_\frakp)$. It is easy to see that this means
  $\tH(M_\frakp^\DP)=-n$.
\end{proof}

\begin{thm}\label{thm:p-adic-period-hida}
  Assume that \cref{cond:big-image-hida-family} and \cref{cond:mmss-free-rank-one} are satisfied.
  Fix $\varphi\in\calX_\calI^\arithm$ of type $(k,\ep,r)$, an integer $n$ with $1\le n\le k$ and a Dirichlet character $\chi$ of conductor $Dp^m$ which we write as a product $\chi=\chi_\nr\chi_p$ with $\chi_\nr$ of conductor $D$ and $\chi_p$ of conductor $p^m$. The choices of $\eta_\varphi^\pm$ and $\delta_{0,\varphi}$ determine bases $\gamma_\varphi$ and $\delta_\varphi$ of $\Motive(F_\varphi^\new)(\chi^*)(n)_\betti^+$ resp.\ $\tangentspace{\Motive(F_\varphi^\new)(\chi^*)(n)}$ as in \cref{sec:bases}. Let $\alpha_{p,\varphi}$ be the unit root of the $p$-th Hecke eigenvalue of $F_\varphi^\new$ and $s=-\chi(-1)(-1)^n$.
  
  Then we have for the $p$-adic period
  \[ \Omega^{\gamma_\varphi,\delta_\varphi,\beta_\varphi(\chi^*,n)}_\frakp(\Motive(F_\varphi^\new)(\chi^*)(n))=-\frac{\alpha_{p,\varphi}^{-m}p^{nm} \errorterm_\frakp(\Xi^s,\eta^s_\varphi)}{U_\varphi^-\chi_\nr(p)^{m}\Gausssum(\chi_p)\Gausssum(\chi^*)}. \]
\end{thm}
\begin{proof}
  We first assume that $\varphi$ is such that $F_\varphi$ is a newform, i.\,e.\ $F_\varphi=F_\varphi^\new$. 
  Consider the commutative diagram
  \begin{equation}\label{eqn:erstes-grosses-diagram-fuer-omega-p}
    \begin{tikzpicture}
      \matrix (m) [matrix of math nodes, row sep=9ex, column sep=2.3em, text height=1.5ex, text depth=0.25ex]
      { \BdR\tensor \wnklevel{Np^r}_\betti^+ & \BdR\tensor\wnklevel{Np^r}_p^+ & \BdR\tensor\wnklevel{Np^r}_p^0 & \;\\
        \BdR\tensor \wnklevel{Np^r}_\betti & \BdR\tensor\wnklevel{Np^r}_p & \BdR\tensor\wnklevel{Np^r}_p & \;\\
      };
      \path[->,font=\scriptsize]
      (m-1-1) edge node [above] {$\cpiso_\et$} node [below] {$\sim$} (m-1-2)
      (m-1-3) edge node [above] {$\sim$} (m-1-4)
      (m-2-1) edge node [above] {$\cpiso_\et$} node [below] {$\sim$} (m-2-2)
      (m-2-3) edge node [above] {$\sim$} (m-2-4);
      \path[left hook->,font=\scriptsize]
      (m-1-1) edge ($(m-2-1)+(0.0,0.4)$)
      (m-1-2) edge ($(m-2-2)+(0.0,0.4)$)
      (m-1-3) edge ($(m-2-3)+(0.0,0.4)$);
      \draw [double equal sign distance]
      (m-1-2) -- (m-1-3)
      (m-2-2) -- (m-2-3);
    \end{tikzpicture}
    \qquad\qquad\qquad\qquad\qquad\qquad\qquad
  \end{equation}
  \begin{equation*}
    \;
    \begin{tikzpicture}
      \matrix (m) [matrix of math nodes, row sep=9ex, column sep=1.7em, text height=1.5ex, text depth=0.25ex]
      { \, & \BdR\tensor\DdR\big(\!\wnklevel{Np^r}_p^0\big) & \BdR\tensor\left(\frac{\DdR\big(\!\wnklevel{Np^r}_p\big)}{\fil^1\DdR\big(\!\wnklevel{Np^r}_p\big)}\right) & \BdR\tensor\gr^0\wnklevel{Np^r}_\dR \\
        \, & \BdR\tensor\DdR\big(\!\wnklevel{Np^r}_p\big) & \BdR\tensor\DdR\big(\!\wnklevel{Np^r}_p\big) & \BdR\tensor\wnklevel{Np^r}_\dR. \\
      };
      \path[->,font=\scriptsize]
      (m-1-1) edge node [above] {$\sim$} (m-1-2)
      (m-1-2) edge node [above] {$\sim$} (m-1-3)
      (m-1-3) edge node [above] {$\cpiso_\dR$} node [below] {$\sim$} (m-1-4)
      (m-2-1) edge node [above] {$\sim$} (m-2-2)
      (m-2-3) edge node [above] {$\cpiso_\dR$} node [below] {$\sim$} (m-2-4);
      \path[left hook->,font=\scriptsize]
      (m-1-2) edge (m-2-2);
      \path[->>,font=\scriptsize]
      (m-2-3) edge ($(m-1-3)-(0.0,0.5)$)
      ($(m-2-4)+(0.0,0.4)$) edge ($(m-1-4)-(0.0,0.2)$);
      \draw [double equal sign distance]
      (m-2-2) -- (m-2-3);
    \end{tikzpicture}
  \end{equation*}
  On the spaces in the lower row, we have the pairings $\paarung{\cdot}{\cdot}_?$ with $?$
  being \enquote{$\betti$}, \enquote{$\dR$} and \enquote{$p$}, and all the maps in the lower
  row respect these
  pairings. 

  We denote the images of elements of $\BdR\tensor_\Q\wnklevel{Np^r}_\betti$ in $\BdR\tensor_\Qp\wnklevel{Np^r}_p$ by the same symbol, by abuse of notation. If we reduce \eqref{eqn:relation-xi-f-u} modulo $\varphi$ we get $U_\varphi^\pm F_\varphi=\Psi_\varphi^\pm$, where $\Psi_\varphi^\pm$ denotes the reduction of $\Psi^\pm$ modulo $\varphi$. Using \cref{thm:lambda-adic-es} and the definition of the $p$-adic error term (\cref{dfn:p-adic-error-term}), this means that the $p$-adic comparison isomorphism (see \cref{thm:faltings-padic-es})
  \begin{multline*}
    \wnklevel{Np^r}_p\tensor_\Qp\C_p=\Hp^1(Y(Np^r)\fibertimes_\Z\Qpbar,\Sym^{k-2}\RD^1f_*\smuline\Z_p)\tensor_{\Z_p}\C_p \isom \\
    \S_k(\Gamma(Np^r),\C_p)(1-k)
    \;\oplus\;
    \HL^1(X(Np^r),\omega_{X(N)}^{2-k})\tensor_\O\C_p\\\surj\S_k(\Gamma(Np^r),\C_p)(1-k)=\gr^{k-1}\wnklevel{Np^r}_\dR\tensor_\Q\C_p(1-k)
  \end{multline*}
  maps $\errorterm_\frakp(\Xi^\pm,\eta^\pm_\varphi)\eta_\varphi^\pm\mapsto F_\varphi\tensor U_\varphi^\pm\tdR^{1-k}$, where we identify $\Cp(1-k)$ with $\Cp\cdot\tdR^{1-k}\subseteq\BdR$.\footnote{Recall that we ignored the Galois actions on the various modules in \cref{thm:lambda-adic-es}. We use $\tdR$ as a basis of $\Cp(1)$ to identify $\Cp(1-k)$ with $\Cp$. This is where the factor $\tdR^{1-k}$ comes from.}

  Let $\rho^\pm\in\BdR\tensor_\Q\wnklevel{Np^r}_\dR$ be the images of $\eta^\pm_\varphi\in\BdR\tensor_K\Motive(F_\varphi^\new)_\betti\subseteq\BdR\tensor_\Q\wnklevel{Np^r}_\betti$ under the map in the lower row in diagram \eqref{eqn:erstes-grosses-diagram-fuer-omega-p}.
  We know by definition of $\eta^\pm_\varphi$ that
  \begin{equation*}
    \paarung{\eta^\pm_\varphi}{w_{Np^r}\errorterm_\frakp(\Xi^\mp,\eta^\mp_\varphi)\tdR^{k-1}\eta_\varphi^\mp}_p=\errorterm_\frakp(\Xi^\mp,\eta^\mp_\varphi)\tdR^{k-1}\cdot\MotiveCanBasis{K(1-k)}
_p.
  \end{equation*}
  Hence since the comparison isomorphism respects the pairings and maps $\errorterm_\frakp(\Xi^\mp,\eta^\mp_\varphi)\tdR^{k-1}\eta_\varphi^\mp\mapsto U_\varphi^\pm F_\varphi$ it follows
  \begin{equation*}
    \paarung{\rho^+}{w_{Np^r}F_\varphi}_\dR=\frac{\errorterm_\frakp(\Xi^-,\eta^-_\varphi)}{U_\varphi^-}.
  \end{equation*}
  This means that the map in the upper row in diagram \eqref{eqn:erstes-grosses-diagram-fuer-omega-p} maps
  \begin{equation*}
    \eta_\varphi^+\mapsto\frac{\errorterm_\frakp(\Xi^-,\eta^-_\varphi)}{U_\varphi^-}\delta_{0,\varphi};
  \end{equation*}
  note that although this latter map is defined over $\BdR$, our  calculations over $\Cp$ suffice for seeing this, as explained in \cref{rem:kann-p-adische-periode-ueber-cp-berechnen}.

  We now look at the following variant of diagram \eqref{eqn:erstes-grosses-diagram-fuer-omega-p}, which differs only by the extra map $\beta_{0,\varphi}^{-1}$:
  \begin{equation}\label{eqn:zweites-grosses-diagram-fuer-omega-p}
    \begin{tikzpicture}
      \matrix (m) [matrix of math nodes, row sep=9ex, column sep=2.3em, text height=1.5ex, text depth=0.25ex]
      { \BdR\tensor \wnklevel{Np^r}_\betti^- & \BdR\tensor\wnklevel{Np^r}_\frakp^- & \BdR\tensor\wnklevel{Np^r}_\frakp^+ & \BdR\tensor\wnklevel{Np^r}_\frakp^0 & \;\\
        \BdR\tensor \wnklevel{Np^r}_\betti & \BdR\tensor\wnklevel{Np^r}_\frakp & \BdR\tensor\wnklevel{Np^r}_\frakp & \BdR\tensor\wnklevel{Np^r}_\frakp & \;\\
      };
      \path[->,font=\scriptsize]
      (m-1-1) edge node [above] {$\cpiso_\et$} node [below] {$\sim$} (m-1-2)
      (m-1-4) edge node [above] {$\sim$} (m-1-5)
      (m-2-1) edge node [above] {$\cpiso_\et$} node [below] {$\sim$} (m-2-2)
      (m-2-4) edge node [above] {$\sim$} (m-2-5)
      (m-1-2) edge node [above] {$\beta_{0,\varphi}^{-1}$} node [below] {$\sim$} (m-1-3)
      (m-2-2) edge node [above] {$\beta_{0,\varphi}^{-1}$} node [below] {$\sim$} (m-2-3);
      \path[left hook->,font=\scriptsize]
      (m-1-1) edge ($(m-2-1)+(0.0,0.4)$)
      (m-1-2) edge ($(m-2-2)+(0.0,0.4)$)
      (m-1-3) edge ($(m-2-3)+(0.0,0.4)$)
      (m-1-4) edge ($(m-2-4)+(0.0,0.4)$);
      \draw [double equal sign distance]
      (m-1-3) -- (m-1-4)
      (m-2-3) -- (m-2-4);
    \end{tikzpicture}
    \qquad\qquad\qquad\qquad
  \end{equation}
  \begin{equation*}
    \;
    \begin{tikzpicture}
      \matrix (m) [matrix of math nodes, row sep=9ex, column sep=1.6em, text height=1.5ex, text depth=0.25ex]
      { \, & \BdR\tensor\DdR\big(\!\wnklevel{Np^r}_p^0\big) & \BdR\tensor\left(\frac{\DdR\big(\!\wnklevel{Np^r}_p\big)}{\fil^1\DdR\big(\!\wnklevel{Np^r}_p\big)}\right) & \BdR\tensor\gr^0\wnklevel{Np^r}_\dR \\
        \, & \BdR\tensor\DdR\big(\!\wnklevel{Np^r}_p\big) & \BdR\tensor\DdR\big(\!\wnklevel{Np^r}_p\big) & \BdR\tensor\wnklevel{Np^r}_\dR. \\
      };
      \path[->,font=\scriptsize]
      (m-1-1) edge node [above] {$\sim$} (m-1-2)
      (m-1-2) edge node [above] {$\sim$} (m-1-3)
      (m-1-3) edge node [above] {$\cpiso_\dR$} node [below] {$\sim$} (m-1-4)
      (m-2-1) edge node [above] {$\sim$} (m-2-2)
      (m-2-3) edge node [above] {$\cpiso_\dR$} node [below] {$\sim$} (m-2-4);
      \path[left hook->,font=\scriptsize]
      (m-1-2) edge (m-2-2);
      \path[->>,font=\scriptsize]
      (m-2-3) edge ($(m-1-3)-(0.0,0.5)$)
      ($(m-2-4)+(0.0,0.4)$) edge ($(m-1-4)-(0.0,0.2)$);
      \draw [double equal sign distance]
      (m-2-2) -- (m-2-3);
    \end{tikzpicture}
  \end{equation*}
  Here the maps in the lower row are no longer compatible with the pairings because we introduced the map $\beta_{0,\varphi}^{-1}$. Anyway, by definition of $C(\beta_{0,\varphi})$ and our previous calculations, we know that the map in the upper row in diagram \eqref{eqn:zweites-grosses-diagram-fuer-omega-p} maps
  \begin{equation*}
    \eta_\varphi^-\mapsto\frac{C(\beta_{0,\varphi})\errorterm_\frakp(\Xi^-,\eta^-_\varphi)}{U_\varphi^-}\delta_{0,\varphi} = \frac{\errorterm_\frakp(\Xi^+,\eta^+_\varphi)}{U_\varphi^-}\delta_{0,\varphi},
  \end{equation*}
  where the equality above comes from \cref{lem:calculation-of-c-beta}.
  
  Now let $s=-\chi(-1)(-1)^n$. We know
  \begin{align*}
    \Motive(F_\varphi)(\chi^*)(n)_\betti^+ &= \Motive(F_\varphi)_\betti^{-s}\tensor\Motive(\chi^*)_\betti\tensor K(n)_\betti, \\
    \tangentspace{\Motive(F_\varphi)(\chi^*)(n)}^+ &= \gr^0\Motive(F_\varphi)_\dR^s\tensor\Motive(\chi^*)_\dR\tensor K(n)_\dR, \\
    \Motive(F_\varphi)(\chi^*)(n)_\frakp^+ &= \Motive(F_\varphi)_\frakp^{-s}\tensor\Motive(\chi^*)_\frakp\tensor K(n)_\frakp, \\
    \Motive(F_\varphi)(\chi^*)(n)_\frakp^\DP &= \Motive(F_\varphi)_\frakp^0\tensor\Motive(\chi^*)_\frakp\tensor K(n)_\frakp.
  \end{align*}
  We have to look at the determinant of the composition
  \begin{align*}
    \BdR\tensor_K\Motive(F_\varphi)(\chi^*)(n)_\betti^+ & \ra[$\cpiso_\et$] \BdR\tensor_L\Motive(F_\varphi)(\chi^*)(n)_\frakp^+ \\
    & \longra[$\beta_\varphi(\chi^*,n)$]{1.4} \BdR\tensor_L \Motive(F_\varphi)(\chi^*)(n)_\frakp^\DP\\
    & \ra[$\alpha^{-1}$] \BdR\tensor_L\DdR(\Motive(F_\varphi)(\chi^*)(n)_\frakp^\DP) \\
    & \ra[$\operatorname{dp}$]\BdR\tensor_L\left(\quotient{\DdR(\Motive(F_\varphi)(\chi^*)(n)_\frakp)}{\fil^0\DdR(\Motive(F_\varphi)(\chi^*)(n)_\frakp)}\right) \\
    & \ra[$\cpiso_\dR$]\BdR\tensor_K\tangentspace{\Motive(F_\varphi)(\chi^*)(n)}.
  \end{align*}
  By definition of $\beta\colon\mathbb T(1)^+\isom\mathbb T^0(1)$ we have that
  \begin{equation*}
    \beta_\varphi(\chi^*,n)\colon \BdR\tensor_L \Motive(F_\varphi)_\frakp^{-s}\tensor\Motive(\chi^*)_\frakp\tensor K(n)_\frakp \ra \BdR\tensor_L \Motive(F_\varphi)_\frakp^0\tensor\Motive(\chi^*)_\frakp\tensor K(n)_\frakp
  \end{equation*}
  is the identity if $s=-1$ and is $\beta_{0,\varphi}^{-1}$ (tensored with the identity) if
  $s=+1$. Hence the above composition is the tensor product of the upper row of the diagram
  \eqref{eqn:erstes-grosses-diagram-fuer-omega-p} or
  \eqref{eqn:zweites-grosses-diagram-fuer-omega-p} with the comparison isomorphisms for the
  motives $\Motive(\chi^*)$ and $\Q(n)$, according to $s=-1$ or $s=+1$. Therefore this
  composition maps
  \begin{equation*}
    \gamma_\varphi=\eta^{\chi(-1)(-1)^n}_\varphi\tensor(\MotiveCanBasis{K(n)}_\betti)^{\tensor n}\tensor\DirichletCanBasis_\betti \mapsto 
    \left(\frac{\errorterm_\frakp(\Xi^s,\eta^s_\varphi)}{U_\varphi^-}\delta_{0,\varphi}\right)\tensor\left(\tdR\MotiveCanBasis{K(n)}_\dR\right)^{\tensor n}\tensor\left(\Gausssum(\chi^*)^{-1}\DirichletCanBasis_\dR\right).
  \end{equation*}
  The claim now follows from \cref{lem:calculation-ep-th}, which completes the proof in the case where $F_\varphi$ is a newform.

  In the case where $F_\varphi$ is not a newform (then automatically of level $Np$), we replace diagram \eqref{eqn:erstes-grosses-diagram-fuer-omega-p} by
  \begin{equation*}
      \begin{tikzpicture}[cross line/.style={preaction={draw=white, -, line width=7pt}}]
        \matrix (m) [matrix of math nodes, row sep=2.3em, column sep=0pt, text height=1.5ex, text depth=0.25ex]
        {
          & \BdR\tensor \wnklevel{N}_\betti^+ & & \BdR\tensor\wnklevel{N}_p^+ & & \BdR\tensor\wnklevel{N}_p^0 & & \;\\
          \BdR\tensor \wnklevel{Np}_\betti^+ & & \BdR\tensor\wnklevel{Np}_p^+ & & \BdR\tensor\wnklevel{Np}_p^0 & & \qquad\quad\\
          & \BdR\tensor \wnklevel{N}_\betti & & \BdR\tensor\wnklevel{N}_p & & \BdR\tensor\wnklevel{N}_p & & \;\\
          \BdR\tensor \wnklevel{Np}_\betti & & \BdR\tensor\wnklevel{Np}_p & & \BdR\tensor\wnklevel{Np}_p & & \;\\
        };
        \path[->,font=\scriptsize]
        (m-1-2) edge node [above] {$\cpiso_\et$} node [below] {$\sim$} (m-1-4)
        (m-1-6) edge node [above] {$\sim$} (m-1-8)
        (m-3-2) edge (m-3-4)
        (m-3-6) edge node [above] {$\sim$} (m-3-8);
        \path[left hook->,font=\scriptsize]
        (m-1-2) edge (m-3-2)
        (m-1-4) edge (m-3-4)
        (m-1-6) edge (m-3-6);
        \draw [double equal sign distance]
        (m-1-4) -- (m-1-6)
        (m-3-4) -- (m-3-6);
        \path[->,font=\scriptsize]
        (m-2-1) edge [cross line] (m-2-3)
        (m-2-5) edge [cross line] (m-2-7)
        (m-4-1) edge node [above] {$\cpiso_\et$} node [below] {$\sim$} (m-4-3)
        (m-4-5) edge node [above] {$\sim$} (m-4-7);
        \path[left hook->,font=\scriptsize]
        (m-2-1) edge [cross line] ($(m-4-1)+(0.0,0.4)$)
        (m-2-3) edge [cross line] ($(m-4-3)+(0.0,0.4)$)
        (m-2-5) edge [cross line] ($(m-4-5)+(0.0,0.4)$);
        \draw [double equal sign distance, cross line]
        (m-2-3) -- (m-2-5);
        \draw [double equal sign distance]
        (m-4-3) -- (m-4-5);
        \path[->,font=\scriptsize]
        (m-1-2) edge (m-2-1)
        (m-1-4) edge (m-2-3)
        (m-1-6) edge (m-2-5)
        (m-3-2) edge (m-4-1)
        (m-3-4) edge (m-4-3)
        (m-3-6) edge (m-4-5);
      \end{tikzpicture}
    \qquad\qquad\qquad\qquad\qquad\qquad\qquad
  \end{equation*}
  \nopagebreak
  \begin{equation*}
    \hspace*{1.3em}
      \begin{tikzpicture}[cross line/.style={preaction={draw=white, -, line width=7pt}}]
        \matrix (m) [matrix of math nodes, row sep=2.3em, column sep=-1.5em, text height=1.5ex, text depth=0.25ex]
        {
          & \, & & \BdR\tensor\DdR(\wnklevel{N}_p^0) & &
          \BdR\tensor\left(\frac{\DdR(\wnklevel{N}_p)}{\fil^1\DdR(\wnklevel{N}_p)}\right) & & \BdR\tensor\gr^0\wnklevel{N}_\dR \\
          \; & & \BdR\tensor\DdR(\wnklevel{Np}_p^0) & & 
          \BdR\tensor\left(\frac{\DdR(\wnklevel{Np}_p)}{\fil^1\DdR(\wnklevel{Np}_p)}\right) & & \BdR\tensor\gr^0\wnklevel{Np}_\dR \\
          & \, & & \BdR\tensor\DdR(\wnklevel{N}_p) & & \BdR\tensor\DdR(\wnklevel{N}_p) & & \BdR\tensor\wnklevel{N}_\dR \\
          \; & & \BdR\tensor\DdR(\wnklevel{Np}_p) & & \BdR\tensor\DdR(\wnklevel{Np}_p) & & \BdR\tensor\wnklevel{Np}_\dR \\
        };
        \path[->,font=\scriptsize]
        (m-1-2) edge node [above] {$\sim$} (m-1-4)
        (m-1-4) edge node [above] {$\sim$} (m-1-6)
        (m-1-6) edge node [above] {$\cpiso_\dR$} node [below] {$\sim$} (m-1-8)
        (m-3-2) edge node [above] {$\sim$} (m-3-4)
        (m-3-6) edge (m-3-8);
        \path[left hook->,font=\scriptsize]
        (m-1-4) edge (m-3-4);
        \path[->>,font=\scriptsize]
        (m-3-6) edge ($(m-1-6)-(0.0,0.4)$)
        (m-3-8) edge (m-1-8);
        \draw [double equal sign distance]
        (m-3-4) -- (m-3-6);
        \path[->,font=\scriptsize]
        (m-2-3) edge [cross line] (m-2-5)
        (m-2-5) edge [cross line] (m-2-7)
        (m-4-5) edge node [above] {$\cpiso_\dR$} node [below] {$\sim$} (m-4-7);
        \path[left hook->,font=\scriptsize]
        (m-2-3) edge [cross line] (m-4-3);
        \path[->>,font=\scriptsize]
        (m-4-5) edge [cross line] ($(m-2-5)-(0.0,0.4)$)
        ($(m-4-7)+(0.0,0.4)$) edge [cross line] (m-2-7);
        \draw [double equal sign distance]
        (m-4-3) -- (m-4-5);
        \path[->,font=\scriptsize]
        (m-1-4) edge ($(m-2-3)+(0.6,0.3)$)
        (m-1-6) edge ($(m-2-5)+(0.9,0.5)$)
        (m-1-8) edge ($(m-2-7)+(0.6,0.3)$)
        (m-3-4) edge ($(m-4-3)+(0.8,0.3)$)
        (m-3-6) edge ($(m-4-5)+(0.8,0.3)$)
        (m-3-8) edge (m-4-7);
      \end{tikzpicture}
  \end{equation*}
  where the maps from the back layer to the front layer are induced from the motivic refinement morphism $\Ref_{\alpha_\varphi}\colon\wnk\tensor_\Q K\ra\wnklevel{Np}\tensor_\Q K$ from \cref{cor:motivic-refinements}, and similarly also for diagram \eqref{eqn:zweites-grosses-diagram-fuer-omega-p} (we omit drawing this).
  Here the front and back faces are just the diagrams from before, which clearly commute, and the top and bottom faces as well as the vertical ones commute since they come from a morphism of motives. By our choices of $\eta^\pm_{\varphi}$ and $\eta^\pm_{\varphi,\new}$ and the commutativity, it does not matter if we take the determinant in the front or back layer, and by a similar reasoning as before we obtain the result also in this case.
\end{proof}

\subsection{$p$-adic $L$-functions}

We begin this section by writing down as explicitely as possible the conjectural
interpolation value of Fukaya and Kato in the case of a Hida family. To do so we need to
calculate the local correction factor that appears in their formula.

\begin{lem}\label{lem:calculation-local-factor}
  Fix a newform $f\in\S_k(\Gamma_1(N),\psi,K)$, an integer $n$ with $1\le n\le k-1$ and a
  Dirichlet character $\chi$ of conductor $Dp^m$ with $p\nmid D$. Assume that $f$ is
  ordinary at $\frakp$, where $\frakp$ is the place of $K$ lying above $p$ which is fixed by
  our embedding $K\subseteq\Qbar\subseteq\Qpbar$. Write $\alpha$ for the unit root of the
  $p$-th Hecke polynomial. Observe that $\alpha=a_p$ if $p\mid N$, where $a_p$ is the $p$-th
  Hecke eigenvalue of $f$.
  \begin{enumerate}
  \item We have
    \[ \EulerFactorP_p(\Mf(\chi^*),T)=1-\alpha^{-1}\widetilde{\chi\psi}(p)p^{k-1}T. \]
  \item The {local correction factor} of the motive $M\da\Mf(\chi^*)(n)$ is
    \[ \operatorname{LF}_p(M) =
      (1-\alpha^{-1}\chi^*(p)p^{n-1})(1-\alpha^{-1}\widetilde{\chi\psi}(p)p^{k-n-1}). \]
  \end{enumerate}
  Here $\widetilde{\chi\psi}$ denotes the primitive character associated to $\chi\psi$.
\end{lem}
\begin{proof}
  We first compute $\EulerFactorP_p(\Mf(\chi^*),T)$. Choose a basis of $\Mf_\frakp$ with
  respect to which the representation has the form $\binmatrix\delta*{}\ep$ with a character
  $\ep$ of $\GQ$. Looking at the determinant we get $\delta\ep=\psi^*\kappa_\cyc^{1-k}$. We
  have an exact sequence
  \begin{equation*}
    0\ra\delta\chi^*\ra \Mf(\chi^*)_\frakp\ra\ep\chi^*\ra0
  \end{equation*}
  and we claim that the sequence
  \begin{equation}\label{eqn:dcris-mf-twist}\tag{$*$}
    0\ra\Dcris(\delta\chi^*)\ra\Dcris(\Mf(\chi^*)_\frakp)\ra\Dcris(\ep\chi^*)\ra0
  \end{equation}
  is exact. It is clear that this sequence is left exact.

  We tensor the first sequence with $\chi\ep^*$ and obtain
  \begin{equation*}
    0\ra\delta\ep^*\ra \Mf(\ep^*)_\frakp\ra L\ra0,
  \end{equation*}
  where $L$ stands for the trivial representation. This is an extension of $\delta\ep^*$ by
  the trivial representation, thus it defines a class $x\in\HL^1(\Qp,\delta\ep^*)$. Since
  all representations in the sequence are de Rham, the sequence
  \begin{equation*}
    0\ra\DdR(\delta\ep^*)\ra \DdR(\Mf(\ep^*)_\frakp)\ra \DdR(L)\ra0
  \end{equation*}
  is still exact. Therefore we have $x\in\Hg^1(\Qp,\delta\ep^*)$ and by the dimension
  formula in \cite[Cor.\ 3.8.4]{MR1086888} we have
  $\Hg^1(\Qp,\delta\ep^*)=\Hf^1(\Qp,\delta\ep^*)$ in this situation, so the sequence
    \begin{equation*}
    0\ra\Dcris(\delta\ep^*)\ra \Dcris(\Mf(\ep^*)_\frakp)\ra \Dcris(L)\ra0
  \end{equation*}
  is exact. We now tensor this sequence with $\Dcris(\ep\chi^*)$ and obtain a morphism of
  exact sequences
  \begin{equation*}
    \begin{tikzpicture}
      \matrix (m) [matrix of math nodes, row sep=2em, column sep=1.5em, text height=1.5ex, text depth=0.25ex]
      { 0 & \Dcris(\delta\ep^*)\tensor\Dcris(\ep\chi^*) & \Dcris(\Mf(\ep^*)_\frakp)\tensor\Dcris(\ep\chi^*) & \Dcris(\ep\chi^*) & 0  \\
        0 & \Dcris(\delta\chi^*) & \Dcris(\Mf(\chi^*)_\frakp) & \Dcris(\ep\chi^*). \\};
      \draw [double equal sign distance] (m-1-4) -- (m-2-4);
      \path[->,font=\scriptsize]
      (m-1-1) edge (m-1-2)
      (m-1-2) edge (m-2-2)
      (m-1-2) edge (m-1-3)
      (m-1-3) edge (m-1-4)
      (m-1-4) edge (m-1-5)
      (m-2-1) edge (m-2-2)
      (m-2-2) edge (m-2-3)
      (m-1-3) edge (m-2-3)
      (m-2-3) edge (m-2-4);
    \end{tikzpicture}
  \end{equation*}
  From this diagram we see that the lower right
  map is surjective, hence the sequence \eqref{eqn:dcris-mf-twist} is exact.

  Since $\delta$ is unramified and $\chi$ is ramified, we have
  $\Dcris(\delta\chi^*)=0$. From the sequence \eqref{eqn:dcris-mf-twist} we get then
  $\Dcris(\Mf(\chi^*)_\frakp)\cong\Dcris(\ep\chi^*)$ and we need to determine when the
  character $\ep\chi^*$ is crystalline, which depends only on its restriction to the inertia
  group $\inertia_p$. Decompose the characters $\chi$ and $\psi$ into $p$-power and
  prime-to-$p$ conductor parts $\chi=\chi_p\chi_\nr$ and $\psi=\psi_p\psi_\nr$. We have
  $\ep\restrict{\inertia_p}=(\delta\ep)\restrict{\inertia_p}=(\psi^*\kappa_\cyc^{1-k})\restrict{\inertia_p}$
  because $\delta$ is unramified, so
  $(\ep\chi^*)\restrict{\inertia_p}=((\psi\chi)^*\kappa_\cyc^{1-k})\restrict{\inertia_p}$
  and we see that the character is crystalline if and only if $\psi_p=\chi^*_p$. In this
  case, $\frobcris$ acts on
  $\Dcris(\delta\ep\chi^*)=\Dcris((\psi_\nr\chi_\nr)^*\kappa_\cyc^{1-k})$ by
  $\psi_\nr\chi_\nr(p)p^{k-1}$, hence it acts on $\Dcris(\ep\chi^*)$ as
  $\delta(\Frob_p)^{-1}\psi_\nr\chi_\nr(p)p^{k-1}=\alpha^{-1}\psi_\nr\chi_\nr(p)p^{k-1}$ as
  claimed. Otherwise we get $\Dcris(M_\frakp)=0$ and the Euler factor is $1$.

  Now we turn to the second statement.
  The action of $\GQp$ on $M_\frakp^\DP$ is given by the character
  $\delta\tensor\chi^*\tensor\kappa_\cyc^n$.
  We have
  \begin{equation}\label{eqn:hecke-pol-mf-prooflabel}\tag{$*$}
    \EulerFactorP_p(\Mf_\frakp,T) = 1-a_p T+\psi(p)p^{k-1}T^2;
  \end{equation}
  note that the roots of this polynomial are the inverses of the roots of the Hecke polynomial, in particular $\alpha^{-1}$ is a root.
  
  We distinguish two cases.
  \begin{arabiclist}
  \item We first assume that $m=0$, i.\,e.\ $\chi$ is unramified at $p$.
    As a first step we compute the expression
    \begin{equation*}
      \frac{\EulerFactorP_p(M_\frakp,T)}{\EulerFactorP_p(M_\frakp^\DP,T)}
    \end{equation*}
    occurring in the definition of the local factor. For this we distinguish again two cases.
    \begin{arabiclist}
    \item First let $p\nmid N$.  In this case $\Motive(f)_\frakp$ is crystalline, hence also
      $M_\frakp$ is crystalline. We have then
      $\Dcris(M_\frakp)=\Dcris(\Mf_\frakp)\tensor_\Qp\Dcris(\chi^*)\tensor_\Qp\Dcris(\Qp(n))$
      so
      \begin{equation*}
        \EulerFactorP_p(M_\frakp,T) = 1-a_p\chi(p)p^{-n}T+\psi\chi^2(p)p^{k-2n-1}T^2.
      \end{equation*}
      On the other hand, by the above description of $M_\frakp^\DP$ we know that 
      \begin{equation*}
        \EulerFactorP_p(M_\frakp^\DP,T) = 1-\alpha\chi(p)p^{-n}T.
      \end{equation*}
      We claim that the quotient of these polynomials is
      \begin{equation*}
        \frac{\EulerFactorP_p(M_\frakp,T)}{\EulerFactorP_p(M_\frakp^\DP,T)} = 1-\alpha^{-1}\psi\chi(p)p^{k-n-1}T.
      \end{equation*}
      Indeed, multiplying we get
      \begin{multline*}
        (1-\alpha^{-1}\psi\chi(p)p^{k-n-1}T)(1-\alpha\chi(p)p^{-n}T)=\\1-\chi(p)p^{-n}(\alpha+\alpha^{-1}\psi(p)p^{k-1})T+\psi\chi^2(p)p^{k-2n-1}T^2
      \end{multline*}
      and since $\alpha^{-1}$ is a zero of $\EulerFactorP_p(\Mf_\frakp,T)$, it follows that $\alpha+\alpha^{-1}\psi(p)p^{k-1}=a_p$.
    \item Now we let $p\mid N$. Then $\psi(p)=0$ and \eqref{eqn:hecke-pol-mf-prooflabel}
      tells us that $\Mf_\frakp$ is not crystalline, so $M_\frakp$ is not crystalline either
      (since $\chi$ and $\kappa_\cyc^n$ are crystalline). But $\Mf_\frakp^0$ is crystalline,
      hence so is $M_\frakp^\DP$ and we get $\Dcris(M_\frakp)=\Dcris(M^\DP_\frakp)$ for
      dimension reasons. Thus the expression
      \begin{equation*}
        \frac{\EulerFactorP_p(M_\frakp,T)}{\EulerFactorP_p(M_\frakp^\DP,T)}
      \end{equation*}
      appearing in the definition of $\operatorname{LF}_p$ is equal to $1$ in this case.
    \end{arabiclist}
    It remains to compute $\EulerFactorP_p((M_\frakp^\DP)^*(1),T)$, still assuming that
    $\chi$ is unramified at $p$. We have
    $(M_\frakp^\DP)^*(1)=(\delta\tensor\chi^*\tensor\kappa_\cyc^n)^*(1)=\delta^{-1}\tensor\chi\tensor\kappa_\cyc^{1-n}$,
    so
    \begin{equation*}
      \EulerFactorP_p((M_\frakp^\DP)^*(1),T)=1-\alpha^{-1}\chi^*(p)p^{n-1}T.
    \end{equation*}
  \item Now we let $m>0$. Since $\chi$ is then ramified at $p$, the one-dimensional
    representations $M_\frakp^\DP$ and $(M_\frakp^\DP)^*(1)$ are both not crystalline, so
    $\EulerFactorP_p(M_\frakp^\DP,T)=\EulerFactorP_p((M_\frakp^\DP)^*(1),T)=1$. The value of
    the remaining expression $\EulerFactorP_p(M_\frakp,1)$ was computed in the first part.
  \end{arabiclist}
\end{proof}

\begin{rem}\label{rem:p-adic-multiplier}
  In \cite[§14]{MR830037}, Mazur, Tate and Teitelbaum introduce an expression they call the \enquote{$p$-adic multiplier}. They consider the same situation as we do here, and their $p$-adic multiplier is (up to a power of $\alpha$ which we ignore here)
  defined as
  \begin{equation}
    e_p(M)\da(1-\alpha^{-1}\chi^*(p)p^{n-1})(1-\alpha^{-1}\chi\psi(p)p^{k-n-1}),
  \end{equation}
  so by our above calculation it essentially equals the local correction factor, except that it contains $\chi\psi(p)=\chi(p)\psi(p)$ instead of the value at $p$ of the associated primitive character. We will later need to know when the two expressions differ. Recall that we always assumed $\chi$ to be primitive (while $\psi$ may be imprimitive). Write $\chi_p$ and $\psi_p$ for the $p$-parts of $\chi$ and $\psi$, respectively. Then we have $e_p(M)\neq\operatorname{LF}_p(M)$ if and only if $\chi_p\psi_p(p)\neq\widetilde{\chi_p\psi_p}(p)$, and it is elementary to check that this happens precisely when $\psi_p$ and $\chi_p$ are nontrivial and inverse to each other. In these cases, we have
  \begin{align*}
    e_p(M)&=(1-\alpha^{-1}\chi^*(p)p^{n-1}),\\
    \operatorname{LF}_p(M)&=(1-\alpha^{-1}\chi^*(p)p^{n-1})(1-\alpha^{-1}\widetilde{\chi\psi}(p)p^{k-n-1}).
  \end{align*}
\end{rem}

We have now computed all the expressions that occur in the conjectural interpolation formula
by Fukaya and Kato. Inserting all this into the general formula, the conjectural value of
the $p$-adic $L$-function evaluated at $(\varphi,\chi\kappa_\cyc^{-n})$ becomes
\begin{multline}
  \label{eqn:conjectural-palf-value}
-(n-1)!(1-\alpha_\varphi^{-1}\chi^*(p)p^{n-1})(1-\alpha_\varphi^{-1}\widetilde{\chi\ep\psi\omega^{-k}}(p)p^{k-n-1}) \\ \cdot \frac{\alpha_{p,\varphi}^{-m}p^{nm} \errorterm_\frakp(\Xi^s,\eta^s_\varphi)}{(2\pi\i)^{n+1-k}U_\varphi^-\chi_\nr(p)^{m}\Gausssum(\chi_p)\errorterm_\infty(F_\varphi,\eta_\varphi^s)} L(F_\varphi^\new\tensor\chi,n).
\end{multline}
Here $\varphi\in\calX_\calI^\arithm$ is of type $(k,\ep,r)$, $\chi$ is a primitive Dirichlet character of
conductor $Dp^m$ with $m\ge0$ and $p\nmid D$, $1\le n\le k-1$ and $s=-\chi(-1)(-1)^n$.
Further $F_\varphi^\new\tensor\chi$ denotes the unique newform almost all of whose Fourier
coefficients are the same as the ones of $\sum_n\chi(n)a_nq^n$, if $\sum_na_nq^n$ is the
Fourier expansion of $F_\varphi^\new$.

We want to express this using the naively twisted $L$-function $L(F_\varphi^\new,\chi,n)$
instead of $L(F_\varphi^\new\tensor\chi,n)$. The relation between these two twisted
$L$-functions is given by
\[
  L(F^\new_\varphi\tensor\chi,n)=L(F^\new_\varphi,\chi,n)\cdot\!\!\!\!\!\!\prod_{\ell\mid(Np^{r'},Dp^m)}\!\!\!\!\!\!\EulerFactorP_\ell(\Motive(F_\varphi^\new)(\chi^*),\ell^{-n})^{-1}, \]
where $r'$ below the product sign should mean either $r$ or $0$, depending on whether
$F_\varphi=F_\varphi^\new$ or not (so that $Np^{r'}$ is the level of $F_\varphi^\new$) and
$\EulerFactorP_\ell$ denotes the Euler factor at $\ell$ of the corresponding motive.  In the
product, the prime $\ell$ can be our fixed prime $p$ if and only if $r'>0$ and $m>0$, and by
\cref{lem:calculation-local-factor} the corresponding factor then equals
\[
  \EulerFactorP_p(\Motive(F_\varphi^\new)(\chi^*),p^{-n})=(1-\alpha_\varphi^{-1}\widetilde{\chi\ep\psi\omega^{-k}}(p)p^{k-n-1}). \]
This factor is nontrivial if and only if the $p$-parts of $\chi$ and $\ep\psi\omega^{-k}$
are inverse to each other (and nontrivial, since $m>0$ and $\chi$ is
primitive).\footnote{The case that $n=k-1$ and
  $\alpha_\varphi=\widetilde{\chi\ep\psi\omega^{-k}}(p)$ cannot occur because this would
  contradict the generalized Ramanujan conjecture proved by Deligne.} By
\cref{rem:p-adic-multiplier}, the cases in which we have a nontrivial Euler factor at $p$
are therefore precisely the cases in which
\[ (1-\alpha_\varphi^{-1}\widetilde{\chi\ep\psi\omega^{-k}}(p)p^{k-n-1})\neq (1-\alpha_\varphi^{-1}\chi\ep\psi\omega^{-k}(p)p^{k-n-1}), \]
and in these cases the first expression (coming from the local correction factor in the interpolation formula) just cancels the Euler factor at $p$ and the second expression is $1$. This discussion shows that
\begin{multline*}
   (1-\alpha_\varphi^{-1}\widetilde{\chi\ep\psi\omega^{-k}}(p)p^{k-n-1})L(F_\varphi^\new\tensor\chi,n) =\\ (1-\alpha_\varphi^{-1}\chi\ep\psi\omega^{-k}(p)p^{k-n-1})
  L(F_\varphi^\new,\chi,n)\cdot
  \!\!\!\!\!\!\prod_{\substack{\ell\mid(Np^{r'},Dp^m)\\\ell\neq p}}\!\!\!\!\!\!\EulerFactorP_\ell(\Motive(F_\varphi^\new)(\chi^*),\ell^{-n})^{-1}
\end{multline*}
in each case. To simplify the formulas below, we set
\[ \operatorname{EEF}(F_\varphi^\new,\chi,n)\da
  \!\!\!\prod_{\ell\mid(N,D)}\!\!\!\EulerFactorP_\ell(\Motive(F_\varphi^\new)(\chi^*),\ell^{-n})^{-1} \]
(where \enquote{EEF} stands for \enquote{extra Euler factors}).  In conclusion, the value of
the conjectural $p$-adic $L$-function at $(\varphi,\chi\kappa_\cyc^{-n})$ from
\eqref{eqn:conjectural-palf-value} becomes
\begin{multline*}
  -(n-1)!(1-\alpha_\varphi^{-1}\chi^*(p)p^{n-1})(1-\alpha_\varphi^{-1}\chi\ep\psi\omega^{-k}(p)p^{k-n-1}) \operatorname{EEF}(F_\varphi^\new,\chi,n) \\ \cdot \frac{\alpha_{p,\varphi}^{-m}p^{nm} \errorterm_\frakp(\Xi^s,\eta^s_\varphi)}{(2\pi\i)^{n+1-k}U_\varphi^-\chi_\nr(p)^{m}\Gausssum(\chi_p)\errorterm_\infty(F_\varphi,\eta_\varphi^s)} 
  L(F_\varphi^\new,\chi,n).
\end{multline*}

We now take Kitagawa's $p$-adic $L$-function $\mu^\Kit_F\in\psring{\calI}{G}$ for $F$,
whose interpolation formula (after a twist to have the same evaluation point) is
\begin{multline*}
  \varphi\left(\int_{\Z_{p,D}^\times}\chi\kappa^{-n}\d\mu^\Kit_F\right) = (n-1)!(1-\alpha_\varphi^{-1}\chi^*(p)p^{n-1})(1-\alpha_\varphi^{-1}\chi\ep\psi\omega^{-k}(p)p^{k-n-1}) \\ \cdot \frac{D^{n-1}p^{m(n-1)}\Gausssum(\chi^*)\errorterm_\frakp(\Xi^\pm,\eta_\varphi^\pm)}{\alpha_\varphi^{m}(2\pi\i)^{n+1-k}\errorterm_\infty(F_\varphi,\eta_\varphi^s)}L(F_\varphi^\new,\chi,n),
\end{multline*}
with $\varphi$, $\chi$, $n$ and $s$ as above.

Let us calculate the quotient of these two interpolation formulas to see what the difference is. In the calculation below, we use two classical relations for Gauß sums:
\begin{enumerate}
\item\label{item:first-gauss-sum-relation} $\Gausssum(\chi_p)\Gausssum(\chi_p^*)=\chi_p(-1)p^m$, see \cite[Lem.\ 3.1.1 (2)]{MR1021004}, 
\item $\Gausssum(\chi^*)=\chi^*_p(D)\chi^*_\nr(p^m)\Gausssum(\chi^*_p)\Gausssum(\chi^*_\nr)$, see \cite[Lem.\ 3.1.2]{MR1021004}. 
\end{enumerate}
So if we divide Kitagawa's value by Fukaya-Kato's value (at $(\varphi,\chi\kappa_\cyc^{-n})$, respectively), we get
\begin{equation}\label{eqn:quotient-palf-values}
  \begin{aligned}
    \frac{\text{Kitagawa's value}}{\text{Fukaya-Kato's value}}
    &= -\frac{U_\varphi^-D^{n-1}p^{-m}\Gausssum(\chi^*)\Gausssum(\chi_p)\chi_\nr(p)^m}{\operatorname{EEF}(F_\varphi^\new,\chi,n)} \\
    &= -\frac{U_\varphi^-D^{n-1}p^{-m}\chi_p^*(D)\chi_\nr^*(p)^m\Gausssum(\chi_p^*)\Gausssum(\chi_\nr^*)\Gausssum(\chi_p)\chi_\nr(p)^m}{\operatorname{EEF}(F_\varphi^\new,\chi,n)} \\
    &= -\frac{\chi_p(-1)U_\varphi^-D^{n-1}\chi_p^*(D)\Gausssum(\chi_\nr^*)}{\operatorname{EEF}(F_\varphi^\new,\chi,n)}.
  \end{aligned}
\end{equation}

We discuss now how to interpolate (most of) the expressions in this quotient.

Regarding the extra Euler factors, we distinguish the automorphic type (principal series,
special or supercuspidal) of the Hida family at each of the primes $\ell$. By
\cite[§3.2]{HsiehHidaFamiliesTripleProducts} and the references given there these are rigid
in the Hida family and their local Galois representations can be described as follows.
Here the character $[\kappa_\cyc]\colon\GQell\ra\calI^\times$ means the composition
\[ [\kappa_\cyc]\colon\GQell\ra[$\kappa_\cyc$]\Z_p^\times\surj 1+p\Z_p=\Gamma^\wt\ra[$\lbrack\cdot\rbrack$](\Lambda^\wt)^\times\ra\calI^\times. \]
\begin{enumerate}
\item If $F$ is principal series at $\ell$, then
  \[ \rho_F\restrict{\GQell} \cong \alpha\xi_1\kappa_\cyc^{1/2}[\kappa_\cyc]^{-1/2} \oplus \alpha^{-1}\xi_2\kappa_\cyc^{1/2}[\kappa_\cyc]^{-1/2} \]
  with $\alpha\colon\GQell\ra\calI^\times$ unramified and $\xi_1,\xi_2\colon\GQell\ra\O^\times$ of finite order.
\item If $F$ is  special at $\ell$, then
  \[ \rho_F\restrict{\GQell} \cong \binmatrix{\xi\kappa_\cyc[\kappa_\cyc]^{-1/2}}{*}{}{\xi[\kappa_\cyc]^{-1/2}} \]
  with $\alpha\colon\GQell\ra\calI^\times$ unramified and $\xi\colon\GQell\ra\O^\times$ of finite order.
\item If $F$ is  supercuspidal at $\ell$, then
  \[ \rho_F\restrict{\GQell} \cong \rho_0\tensor[\kappa_\cyc]^{-1/2} \]
  with $\rho_0\colon\GQell\ra\GL_2(\O)$ an irreducible Artin representation.
\end{enumerate}

Take an unramified character $\alpha\colon\GQell\ra\calI^\times$, a character
$\xi\colon\GQell\ra\O^\times$ of finite order and further $i,j\in\frac12\Z$. As $\xi$ is of
finite order, it factors through $\Gal{\Qp(\mu_M)}{\Qp}$ for some $M\in\N$ and we can write
it as a product $\xi=\xi_{\ell}\xi'$ as before. Assume that $M$ is chosen minimally (i.\,e.\
$M=\conductor\xi$) and let $\mu\da\ord_\ell M$. If $\nu\ge\mu$ we can and do view $\xi$ as a
character of $\Z_{p,D}^\times$. We then define
\[ \mu(\alpha,\xi,i,j)\da
  \begin{cases}
    \ell^{-i}\alpha\xi'[\kappa_\cyc]^j(\Frob_\ell)\delta_{\xi_\ell}\in\psring\calI{\Z_{p,D}^\times} & \text{if }\nu\ge\mu,\\
    0&\text{otherwise}.
  \end{cases}
\]
Using this notation and distinguishing the above three cases we can define elements
\[ \mu_\ell(F) \da
  \begin{cases}
    (\phi(\ell^\nu)-\mu(\alpha,\xi_1,\frac12,-\frac12))(\phi(\ell^\nu)-\mu(\alpha^{-1},\xi_2,\frac12,-\frac12)) & \text{if $F$ in the principal series at $\ell$}, \\
    (\phi(\ell^\nu)-\mu(1,\xi,1,-\frac12))(\phi(\ell^\nu)-\mu(1,\xi,0,-\frac12)) & \text{if $F$ special at $\ell$}, \\
    \phi(\ell^\nu)^2 & \text{if $F$ supercuspidal at $\ell$}, \\
  \end{cases} \]
It is then a straightforward calculation to check that for each
$\varphi\in\calX_\calI^\arithm$, each finite order character
$\chi\colon\Z_{p,D}^\times\ra\O^\times$ and each $n\in\Z$ we have
\[ \varphi\left(\int_{\Z_{p,D}^\times}\!\!\!\chi\kappa_\cyc^{-n}\,\d\mu_\ell(F)\right) =
  \phi(\ell^\nu)^2\EulerFactorP_\ell(\Motive(F_\varphi^\new)(\chi^*)(n),0). \]

Finally it is easy to see that there exists a unit $\mu_D\in(\psring{\calI}G)^\times$ such that for each $n\in\Z$ and Dirichlet character $\chi$ of conductor $Dp^m$ for some $m\in\N_0$, which we write as a product $\chi=\chi_\nr\chi_p$ of characters of conductors $D$ and $p^m$, we have \[ \int_G\chi\kappa_\cyc^{-n}\d\mu_D = D^{n-1}\chi_p^*(D)\Gausssum(\chi_\nr^*). \]

We now have all ingredients ready to arrive at our main result. To state it more clearly, we introduce the following notations for the technical conditions we need.
\begin{cond}
  We consider the following conditions on $F$ and primes $\ell\mid(N,C)$:
  \begin{align}
    \label{eqn:minor-condition-on-tame-levels}
    \tag{$\text{nd}_\ell$}
    \begin{minipage}{.8\textwidth}
      \begin{center}
        $p\nmid\ell-1$
      \end{center}
    \end{minipage}
    \\
    \label{eqn:in-fact-unit}
    \tag{$\mathrm{triv}_\ell$}
    \left.
    \begin{minipage}{.8\textwidth}
      \begin{center}
        $F$ is in the principal series at $\ell$ and $\ordnung_\ell\conductor\xi_i<\ord_\ell D$ for $i=1,2$, or\\
        $F$ is special at $\ell$ and $\ordnung_\ell\conductor\xi<\ordnung_\ell D$, or\\
        $F$ is supercuspidal at $\ell$
      \end{center}
    \end{minipage}
    \right\}
  \end{align}
\end{cond}

If \eqref{eqn:minor-condition-on-tame-levels} holds for $\ell$, then $\phi(\ell^\nu)$ is a unit in $\O$. If \eqref{eqn:in-fact-unit} holds for $\ell$, then $\mu_\ell(F)=\phi(\ell^\nu)^2$.
This follows directly from the definition of $\mu_\ell(F)$.

\begin{thm}\label{thm:main-thm}
  Fix coefficient rings as in \cref{setting:hida-families}, a Hida family $F$ which is new
  and basis elements as in \cref{situation:p-adic-periods}. Assume that
  \cref{cond:big-image-hida-family} and \cref{cond:mmss-free-rank-one} hold.
  \begin{enumerate}
  \item There exists a $p$-adic $L$-function $\mu^\FK_F\in\bigLambda[\frac1p]=\Qp\tensor_\Zp\psring{\calI}{G}$ such that for each $\varphi\in\calX_\calI^\arithm$ of type $(k,\ep,r)$, each primitive Dirichlet character $\chi$ of conductor $Dp^m$ for some $m\ge0$ and each $1\le n\le k-1$ the evaluation
    \[ \varphi\left(\int_{G} \chi\kappa_\cyc^{-n}\d\mu^\FK_F\right) \]
    is the value predicted by Fukaya and Kato up to a factor $\chi_p(-1)$.
  \item If we assume in addition that for each prime $\ell\mid(N,D)$ at least one of \eqref{eqn:minor-condition-on-tame-levels} and \eqref{eqn:in-fact-unit} holds,
    then $\mu^\FK_F\in\bigLambda$. In this case the ideals generated by $\mu_F^\FK$ and $\mu _F^\Kit$ differ by $\prod_{\ell\mid(N,C)}\mu_\ell(F)$.
  \item Assume that for each prime $\ell\mid(N,D)$ both \eqref{eqn:minor-condition-on-tame-levels} and \eqref{eqn:in-fact-unit} hold.
    Then $\mu^\FK_F$ and $\mu_F^\Kit$ generate the same ideal in $\psring\calI{G}$.
  \end{enumerate}
\end{thm}
\begin{proof}
  We just define
  \begin{equation*}
    \mu_F^\FK\da-\frac{\mu_F^\Kit}{U^-\cdot\mu_D}\prod_{\ell\mid(N,C)}\frac{\mu_\ell(F)}{\phi(\ell^{\nu_\ell})^2}.
  \end{equation*}
  This is an element of $\bigLambda$ if \eqref{eqn:minor-condition-on-tame-levels} or \eqref{eqn:in-fact-unit} holds for each $\ell\mid(N,D)$ and of $\Qp\tensor_\Zp\bigLambda$ otherwise.
  The claims follow from our previous calculations
  and the fact that $U_\varphi^-$ comes by definition from an element $U^-\in\calI^\times$ which is a unit by \cref{thm:u-is-unit}. 
\end{proof}

We close with some remarks about our result.

\begin{rem}
  \begin{enumerate}
  \item
    The conditions in \cref{thm:main-thm} that \eqref{eqn:minor-condition-on-tame-levels} or \eqref{eqn:in-fact-unit} or even both of them hold for all $\ell\mid(N,D)$ are the most general ones we could get for these results to hold, but of course they are bit ugly. Each of the following conditions implies them:
    \begin{arabiclist}
    \item $F$ is supercuspidal at all primes dividing $(N,D)$,
    \item $F$ is supercuspidal at all primes dividing $N$,
    \item $(N,D)=1$,
    \item $N=1$,
    \item $D=1$.
    \end{arabiclist}
  \item In \cref{thm:main-thm} we need the Dirichlet character $\chi$ to be primitive,
    i.\,e.\ its \enquote{$D$-part} away from $p$ must be primitive. This is because the
    interpolation formula in Kitagawa's construction is known to hold only in this case. Of
    course we can perform the construction for each $D$, but there does not seem to be an
    easy relation between the measures obtained for different $D$, see
    \cite[\enquote{Warning} on p.\ 19]{MR830037}. Fukaya's and
    Kato's formula predicts a similar interpolation behaviour also for imprimitive
    characters, but we do not know how to prove this. Of course this problem disappears if
    we set $D=1$.
  \end{enumerate}
\end{rem}

\begin{rem}\label{rem:problem}
  In \cref{thm:main-thm} we obtained the desired interpolation property only up to a factor
  of $\chi_p(-1)$. This is in fact a serious problem: It is easy to see that there cannot
  exist an element $\mu\in\bigLambda$ such that for each finite order character $\chi$ of
  $G$ and each $n$ in a fixed subset of $\Z$ containing at least one even and one odd number
  we have \[ \int_{\Gcyc}\chi^*\kappa_\cyc^n\d\mu = \chi(-1) \] (not even in the quotient
  ring of $\bigLambda$). Therefore we assume that Fukaya's and Kato's interpolation formula
  is wrong. If one uses the conjectural interpolation formula proposed by Coates and
  Perrin-Riou \cite{MR1129081} instead, the resulting formula seems to be correct.

  This is particularly intriguing because several other texts in the literature seem to
  contain errors in their interpolation formula as well. For example, Kitagawa's
  interpolation formula which we cited before is different in his original work; however the
  author carefully rewrote Kitagawa's construction in \cite[Appendix B]{Diss} and obtained
  the previously cited version and found the following error in Kitagawa's text: The formula
  in \cite[Prop.\ 4.7]{MR1279604} (or alternatively, \cite[Thm.\
  6.2]{MR1279604}) contains a factor $(-\Delta)^\nu$, where $\Delta$ is what we
  called $D$ (and $\nu$ there is what we called $n$). The same formula contains an
  expression $A(\xi,\chi,\nu+1)$. The meaning of this latter expression is defined in
  \cite[§4.1]{MR1279604}, where the defining formula also contains a factor
  $(-1)^\nu$. Hence if one inserts the definition of $A(\xi,\chi,\nu+1)$ into the formula in
  \cite[Prop.\ 4.7]{MR1279604} (or alternatively, \cite[Thm.\
  6.2]{MR1279604}), the two factors $(-1)^\nu$ should disappear. However, in the
  final formula in \cite[Thm.\ 1.1]{MR1279604}, the factor $(-\Delta)^\nu$ is
  still present, while $(-1)^\nu$ is not. Other texts, e.\,g.\ \cite{MR830037}, seem to
  contain the correct interpolation formula. The following table gives an overview on other
  texts.
  
  \bigskip
  
  \begin{tabular}[c]{p{.45\textwidth} | p{.45\textwidth}}
    \bfseries{same as Mazur/Tate/Teitelbaum} & \bfseries{same as Kitagawa} \\
    \hline
    Vishik \cite{MR0412114} & Amice-Velu \cite{MR0376534} \\
    Pollack-Stevens \cite{MR2760194}, \cite{MR3046279} & Kato \cite{MR2104361} \\
    Bella\"iche \cite{MR2929082} & Hida \cite{MR1216135} \\
    Delbourgo \cite{MR2444858} & \\
  \end{tabular}
  
  \bigskip

  A solution to this problem that was suggested by Y.\ Zaehringer \cite[Ex.\
  9.2.14]{ZaehringerDiss} is to replace the character $\psi$ in Fukaya's and Kato's text by
  its inverse, which is equivalent to changing a convention we made at the very beginning:
  on \cpageref{c-cp-orientation} we required $\C$ and $\Cp$ to be \enquote{oriented
    compatibly}, i.\,e.\ the system $\xi$ of $p$-power roots of unity in $\Cp$ should be
  identified with $(e^{2\pi\i p^{-n}})_n$ by our pair of embeddings. If we reverse this
  (thus using $(e^{-2\pi\i p^{-n}})_n$ instead), this indeed makes the unwanted factor
  $\chi(-1)$ disappear. The same choice is also made by Perrin-Riou, see \cite[top of p.\
  91]{MR1743508}. In \cite[beginning of
  §4]{MR2276851} it is written explicitly that the system
  $(\e^{2\pi\i p^{-n}})_n$ should be used, but this could be wrong. In order to figure out
  whether this can be a satisfactory solution, one should check which consequences this has
  for other motives, such as e.\,g.\ the motives attached to Dirichlet characters. But this
  lies outside the scope of this work.
\end{rem}

\begin{rem}
  At a first glance, the fact that $\mu_F^\Kit$ and $\mu_F^\FK$ do not always generate the
  same ideal looks problematic with regard to the Main Conjecture, but in fact this is
  consistent.  Since this work focuses on the analytic side of Iwasawa Theory, we do not go
  into detail here, but let us sketch the reason.

  Let $\Sigma$ be a finite set of places of $\Q$ not containing $p$.  Then one can consider
  $\Sigma$-primitive $p$-adic $L$-functions, which should satisfy similar interpolation
  properties as above but with the Euler factors of the complex $L$-function for the primes
  in $\Sigma$ omitted. In our case, since we can interpolate the Euler factors away from
  $p$, we know that such $\Sigma$-primitive $p$-adic $L$-functions exist for any
  $\Sigma$. For example, Kitagawa's $p$-adic $L$-function is such a $\Sigma$-primitive one
  for $\Sigma$ being the set of primes dividing $(N,D)$.

  On the algebraic side, the Iwasawa modules one considers are (duals of) {Selmer groups},
  which are subgroups of an $\HL^1$ in continuous Galois cohomology whose elements satisfy
  certain local conditions. As on the analytic side, there is also a notion of
  $\Sigma$-primitivity here: we can form $\Sigma$-primitive Selmer groups by omitting the
  local conditions for the places in $\Sigma$.

  Obviously, the usual Selmer group (the one with $\Sigma=\varnothing$) is then contained in
  the $\Sigma$-primitive one for any nonempty $\Sigma$, and the cokernel can be described in
  terms of the local conditions at the primes in $\Sigma$. In accordance with the analytic
  side, the characteristic ideal of this cokernel is generated by a product of elements
  interpolating the Euler factors at primes in $\Sigma$. See \cite[§2, esp.\ Prop.\ 2.4,
  Cor.\ 2.3]{MR1784796} for these facts and for a treatment of
  $\Sigma$-primitive Selmer groups.

  In their proof of the Iwasawa Main Conjecture for (certain) modular forms and families,
  Skinner and Urban a priori use those $\Sigma$-primitive objects on both the algebraic and
  analytic side for $\Sigma$ containing all primes dividing $N$ (among others); using the
  techniques mentioned above they deduce from this the Main Conjecture for any set
  $\Sigma$. See \cite[§3, esp.\ §3.6 and the proof of Thm.\ 3.29]{MR3148103}.

  In Fukaya's and Kato's work, the Selmer group (or rather Selmer complex) which is related
  to their $p$-adic $L$-function via the Main Conjecture is the usual one (i.\,e.\
  $\Sigma=\varnothing$), as it should be, since their $p$-adic $L$-function has all Euler
  factors away from $p$, i.\,e.\ it is $\varnothing$-primitive. Hence our result and the
  fact that $\mu_F^\Kit$ and $\mu_F^\FK$ may generate different ideals in $\bigLambda$ is
  consistent with the Main Conjecture proved by Skinner and Urban.

  Let us finally remark that Fukaya and Kato also formulate a version of their conjecture
  involving $\Sigma$-primitive objects. They formulate this in terms of an open subset $U$
  of $\Spec\Z$ which plays the role of the complement of $\Sigma$. See \cite[§4.1.2 resp.\
  §4.2.11]{MR2276851} for the definitions of the
  $\Sigma$-primitive resp.\ $\varnothing$-primitive Selmer complexes and \cite[Thm.\ 4.2.22,
  case with \enquote{(resp.\ \dots)}]{MR2276851} for the
  conjectural interpolation formula and the Main Conjecture in the $\Sigma$-primitive
  case. Since we can interpolate all Euler factors away from $p$, we can show with our
  methods that also these $p$-adic $L$-functions exist for any $U$ (of course up to the
  factor $\chi_p(-1)$), and this is still consistent with the Main Conjecture.
\end{rem}

\printbibliography

\end{document}